\documentclass[12pt,reqno]{amsart}
\usepackage{amsmath}
\usepackage{amssymb}
\usepackage[mathscr]{eucal}

\makeatletter

\def\subsection{\@startsection{subsection}{2}%
  \z@{.5\linespacing\@plus.7\linespacing}{1pt}%
      {\normalfont\bfseries}}

\def\l@section{\@tocline{1}{0pt}{1pc}{}{\bfseries}}
\def\l@subsection{\@tocline{2}{0pt}{3.1em}{5pc}{}}

\makeatother

\allowdisplaybreaks[1]
\newtheorem{thm}{Theorem}[section]
\newtheorem{lem}[thm]{Lemma}
\newtheorem{cor}[thm]{Corollary}
\newtheorem{prop}[thm]{Proposition}

\newtheorem*{mthm}{Main Theorem}

\theoremstyle{definition}

\newtheorem{rem}[thm]{Remark}
\newtheorem{defn}[thm]{Definition}

\theoremstyle{remark}

\numberwithin{equation}{section}

\newcommand{\C}{{\mathbb{C}}}
\newcommand{\R}{{\mathbb{R}}}
\newcommand{\Z}{{\mathbb{Z}}}

\newcommand{\N}{{\mathbb{N}}}
\newcommand{\T}{{\mathbb{T}}}

\DeclareMathOperator{\Ad}{Ad}
\DeclareMathOperator{\Tr}{Tr}
\DeclareMathOperator{\CTr}{CTr}

\DeclareMathOperator{\id}{id}

\DeclareMathOperator{\End}{End}
\DeclareMathOperator{\Sect}{Sect}
\DeclareMathOperator{\Aut}{Aut}
\DeclareMathOperator{\Cnt}{Cnt}
\DeclareMathOperator{\Cnd}{Cnd}
\DeclareMathOperator{\Int}{Int}
\def\oInt{\ovl{\Int}}

\DeclareMathOperator{\mo}{mod}

\DeclareMathOperator{\Ind}{Ind}

\def\mE{{\mathcal{E}}}
\def\mF{\mathcal{F}}

\def\mM{\mathcal{M}}
\def\mN{\mathcal{N}}

\def\meC{\mathscr{C}}

\def\meF{\mathscr{F}}
\def\meH{\mathscr{H}}

\def\meM{\mathscr{M}}
\def\meN{\mathscr{N}}

\def\meS{\mathscr{S}}
\def\meT{\mathscr{T}}

\def\al{\alpha}
\def\be{\beta}

\def\de{\delta}

\def\la{\lambda}

\def\vep{\varepsilon}
\def\ph{{\phi}}
\def\ps{{\psi}}
\def\vph{\varphi}

\def\om{\omega}
\def\si{\sigma}
\def\ta{\tau}
\def\th{\theta}

\def\Ph{\Phi}
\def\Ps{\Psi}
\def\Th{\Theta}

\def\De{\Delta}

\def\La{\Lambda}

\def\el{\ell}

\def\ovl{\overline}                 
\def\wdh{\widehat}
\def\wdt{\widetilde}

\def\hchi{\hat{\chi}}

\def\tvph{{\tilde{\vph}}}
\def\hth{{\hat{\th}}}

\def\tal{\wdt{\al}}

\def\trho{\wdt{\rho}}
\def\tsi{\wdt{\si}}

\def\tM{\wdt{M}}

\def\brho{{\bar{\rho}}}

\def\bsi{{\bar{\si}}}

\def\brrho{{\breve{\rho}}}

\def\subs{\subset}                   

\def\setm{\setminus}
\def\nin{\notin}

\def\oti{\otimes}                    
\def\rti{\rtimes}                    

\def\lra{\longrightarrow}        
\def\col{\colon}
\def\ra{\rightarrow}

\author[T. Masuda]{Toshihiko Masuda$^1$}
\address{$^1$ 
Graduate School of Mathematics, Kyushu University, 
Hakozaki, \mbox{812-8581}\\
\indent JAPAN}
\email{masuda@math.kyushu-u.ac.jp}

\author[R. Tomatsu]{Reiji Tomatsu$^{\,2,3}$}
\address{$^{2}$Department of 
Mathematical Sciences,
\hspace{-0.6mm}University of Tokyo,
\hspace{-0.6mm}Komaba,\mbox{153-8914}\\
\indent JAPAN,}
\address{$^3$Department of Mathematics, 
K.U. Leuven, Celestijnenlaan 200B,
B-3001\\ 
\indent BELGIUM}
\email{tomatsu@ms.u-tokyo.ac.jp}

\subjclass[2000]{Primary 46L10; Secondary 46L37}
\thanks{$^{1,2,3}$ Supported by JSPS}
\begin{document}

\title{Approximate innerness and central triviality of endomorphisms}

\begin{abstract}
We introduce the notions
of approximate innerness and central triviality
for endomorphisms on separable von Neumann factors,
and
we characterize them for hyperfinite factors
by Connes-Takesaki modules
of endomorphisms and modular endomorphisms
which are introduced by Izumi.
Our result is a generalization of the corresponding result
obtained by Kawahigashi-Sutherland-Takesaki
in automorphism case.
\end{abstract}

\maketitle

\section{Introduction}

The purpose of this paper is first to introduce two notions,
``approximate innerness'' and ``central triviality''
for endomorphisms on factors,
and second to generalize the result of Y. Kawahigashi, C. E. Sutherland
and M. Takesaki \cite{KST} to endomorphism case.

The study of automorphisms or group actions
has drawn attentions in studies of operator algebras.
From the viewpoint of the classification theory of group actions,
two classes of automorphisms have been considered significant,
i.e., approximately inner automorphisms $\oInt(M)$
and centrally trivial automorphisms $\Cnt(M)$ on a factor $M$, 
which are studied by A. Connes \cite{C2,C4}.
In particular,
A. Ocneanu obtained the uniqueness result for
approximately inner and centrally free actions of discrete amenable groups
on McDuff factors \cite{Oc1}. 

On those two properties, Connes announced
the following characterization without a proof,
using the flow of weights \cite{C3}:
for any hyperfinite factor $M$,
\begin{enumerate}
\item $\oInt(M)=\mathrm{Ker}(\mo)$,

\item
$\Cnt(M)=\{\Ad u \cdot\si_c^\vph\mid u\in U(M),
\ c\in Z^1(F^M,\C)\}$,
\end{enumerate}
where $\mo\col \Aut(M)\ra \Aut(F^M)$ is the Connes-Takesaki module map
\cite{CT}, $Z^1(F^M,\C)$ is the set of scalar valued 1-cocycles for
the flow of weights $F^M$ and $U(M)$
is the set of all unitary elements in $M$.

A proof of this theorem was first presented
by Kawahigashi, Sutherland and Takesaki \cite{KST}.
Their result well motivates us to consider a generalization to
endomorphisms, but our work also relies on the study of actions
because an action of a compact group dual
(or more generally a discrete quantum group) is
essentially identical to a Roberts action \cite{Ro}.
Indeed, using the present work,
we will obtain a uniqueness result
for approximately inner and centrally free actions
of an amenable discrete Kac algebra on hyperfinite type III factors
\cite{MT2}.

Now let us explain the details of this paper.
In the first half, we study approximately inner endomorphisms on
a factor $M$.
We introduce a topology on
the set of endomorphisms with finite index,
and then discuss approximation by inner endomorphisms.
Hence it is convenient to introduce the notion of rank
for an approximately inner endomorphism, that is,
if $\rho$ is approximated by inner endomorphisms with dimension $r$,
then we will say that $\rho$ has rank $r$.
In fact, we can define the rank for any positive real
number so that we match approximate innerness to
the theory of the Connes-Takesaki modules for endomorphisms
introduced by M. Izumi \cite{Iz1}.
Then we study the set of approximately inner endomorphisms
with rank $r>0$, which we denote by $\oInt_r(M)$.

In the latter half,
we study the set of centrally trivial endomorphisms
on a factor $M$ which we denote by $\Cnd(M)$.
Our aim is to clarify the relation between $\Cnd(M)$
and $\End(M)_{\rm m}$, the set of modular endomorphisms
introduced by Izumi \cite{Iz1}.
Our main result is the following 
(Theorem \ref{thm: main-app}, \ref{thm: mod-cent}):

\begin{mthm}Let $M$ be a hyperfinite factor.
Then one has the following:
\begin{enumerate}
\item
$\oInt_r(M)=\{\rho\in\End(M)_{\rm CT}\mid \mo(\rho)=\th_{\log(r/d(\rho))}\}$
for any $r>0$.
\item 
$\Cnd(M)=\End(M)_{\rm m}$. 
\end{enumerate}
\end{mthm}
Here, $\End(M)_{\rm CT}$ is the set of endomorphisms on $M$
which have Connes-Takesaki modules.

We should emphasize that we make use of
the main result of \cite{KST}
and their idea on discrete decompositions,
but we mainly use Popa's theory on
approximate innerness and central freeness of subfactors \cite{Po1}.
Indeed, some discussions of \cite{KST} involve
the classification results of discrete group actions 
\cite{Oc1,ST},
and it does not seem that those are applicable
to the endomorphism case at ease.

\section{Approximate innerness of endomorphisms}

First we fix notations.
In this paper,
we treat only von Neumann algebras with separable preduals
except for ultraproduct von Neumann algebras.
Let $M$ be a von Neumann algebra.
For $\vph\in M_*$ and $a\in M$,
we define the functionals $\vph a$ and $a\vph$ in $M_*$
by $\vph a(x):=\vph(ax)$ and $a\vph(x):=\vph(xa)$ for all $x\in M$.
We denote by $U(M)$ the set of all unitary elements in $M$.
We denote by $W(M)$ and $W_{\rm lac}(M)$
the sets of faithful normal semifinite weights
and faithful normal semifinite lacunary weights on $M$,
respectively \cite{C0,Ta2}.
For a faithful state $\ph\in M_*$, we set $|x|_\ph:=\ph(|x|)$
and $\|x\|_\ph:=\ph(x^* x)^{1/2}$.
Note that $|\cdot |_\ph$ satisfies the triangle inequality
on the centralizer of $\ph$.
We denote by $M_\ph$ the centralizer of $\ph$, that is,
$x\in M_\ph$ if and only if $\ph(xy)=\ph(yx)$ for any $y\in M$.

Let $\meH\subs M$ be a subspace.
We say that $\meH$ is a Hilbert space in $M$ if
$\meH\subs M$ is $\si$-weakly closed 
and $\eta^* \xi\in\C$ for all $\xi,\eta\in\meH$
\cite{Ro}.
The smallest projection $e\in M$ such that $e\meH=\meH$ is
called the support of $\meH$.

We denote by $\End(M)$ and $\Sect(M)$ the set of
normal endomorphisms and sectors on $M$,
that is, $\Sect(M)$ 
is the set of equivalence classes of endomorphisms on $M$ 
by unitary equivalence.
For two endomorphisms $\rho,\si\in \End(M)$,
we let
$(\rho,\si)=\{v\in M\mid v\rho(x)=\si(x)v\ \mbox{for all}\ x\in M\}$.
If $\rho$ is irreducible, then $(\rho,\si)$ is a Hilbert space 
with the inner product $(V,W)=W^* V$ for $V,W\in (\rho,\si)$.
For an endomorphism $\rho$ on $M$,
a left inverse $\ph$ of $\rho$
means a faithful normal unital completely
positive map on $M$ with $\ph\circ\rho=\id$.
For a factor $M$, we denote by $\End(M)_0$ the set of
endomorphisms with finite index.
For $\rho\in \End(M)_0$, $E_\rho$ denotes the minimal expectation 
from $M$ onto $\rho(M)$ \cite{Hi}.
We define the standard left inverse of $\rho$
by $\ph_\rho=\rho^{-1}\circ E_\rho$.

\subsection{A topology on the set of endomorphisms} 
We define the topology on the set of endomorphisms on a factor,
which is related to \cite[Definition 3.1]{MT1}.

\begin{defn}\label{defn: topology}
Let $N$ be a factor.
We introduce the topology $\meT_\ph$ on $\End(N)_0$
by giving the following neighborhoods at $\rho_0\in\End(N)_0$ 
which are defined 
for $n\in\N$, a finite family $\{\vph_i\}_{i=1}^n\subs N_*$
and $\vep>0$ by 
\[
U(\rho_0; \vph_1,\dots,\vph_n,\vep)
=
\{
\rho\in \End(N)_0\mid 
\|\vph_i\circ \ph_\rho-\vph_i\circ \ph_{\rho_0}\|<\vep
\ \mbox{for all}\ i=1,\dots,n
\}.
\]
\end{defn}

The topology $\meT_\ph$ is metrizable.
Take a norm dense sequence $\{\vph_n\}_{n=1}^\infty$
in the set of normal states on $N_*$.
The following metric $d_\ph$ defines the topology $\meT_\ph$,
\[
d_\ph(\rho,\si)
=\sum_{n=1}^\infty \frac{1}{2^n}\|\vph_n\circ \ph_\rho
-\vph_n\circ\ph_\si\| 
\ \mbox{for}\ \rho,\si\in \End(N)_0. 
\]
The restriction of $\meT_\ph$ on $\Aut(N)$
coincides with the $u$-topology
\cite[Definition 3.4]{Ha} as seen below.

\begin{lem}
Let $\al^\nu$, $\nu\in\N$, and $\al$ be in $\Aut(N)$. 
Then $\al^\nu\to\al$ as $\nu\to\infty$ 
in the topology $\meT_\ph$ 
if and only if $\al^\nu\to \al$ as $\nu\to\infty$ in the $u$-topology. 
\end{lem}
\begin{proof}
The if part is trivial.
We show the only if part as follows.
The sequence $\{\al^\nu\}_\nu$ converges to $\al$ as $\nu\to\infty$
in the topology $\meT_\ph$
if and only if
$\|\vph\circ (\al^\nu)^{-1}-\vph\circ \al^{-1}\|\to0$
as $\nu\to\infty$
for all $\vph\in N_*$.
If we put $\ps\circ \al$ ($\ps\in N_*$) for $\vph$,
we obtain
$\|\ps\circ\al\circ(\al^\nu)^{-1}-\ps\|=\|\ps\circ\al-\ps\circ\al^\nu\|$.
Hence
$\|\vph\circ (\al^\nu)^{-1}-\vph\circ \al^{-1}\|\to0$
as $\nu\to\infty$
for all $\vph\in N_*$
if and only if
$\|\ps\circ \al^\nu-\ps\circ \al\|=0$
as $\nu\to\infty$ for all $\ps\in N_*$,
which means
the convergence of $\{\alpha^\nu\}_\nu$ to $\al$ in the $u$-topology.
\end{proof}

Note that the $u$-topology on $\Aut(N)$ is metrizable and complete.
By the previous lemma, the restriction of the metric
$d_\ph$ on $\Aut(N)$ gives the $u$-topology,
but $\Aut(N)\subs \End(N)_0$ is not closed in general.
Namely, that restriction may not be a complete metric.

\subsection{Approximate innerness} 
Let $\meH\subset N$ be a finite dimensional Hilbert space
with support $1$,
and $\{v_i\}_i\subset \meH$ an orthonormal basis.
Then $\rho_\meH(x):=\sum_i v_ix v_i^*$ gives an endomorphism of $N$.
We say that $\rho\in \End(N)$ is \emph{inner}
if there exists a Hilbert space $\meH\subs N$ such that
$\rho=\rho_\meH$.
Then we have $d(\rho)=\dim(\meH)$.
We also say that $\rho$ has \emph{rank} $\dim(\meH)$.
We denote by $\Int_d(N)$ the set of inner endomorphisms of rank $d$.
Now we introduce approximate innerness of endomorphisms.

\begin{defn}
Let $\rho\in\End(N)_0$ and $r\in\N$.
We say that $\rho$ is an \emph{approximately inner endomorphism of rank}
$r$
if for each $\nu\in\N$,
there exists an $r$-dimensional Hilbert space
$\meH^\nu\subs N$ with support $1$ 
such that
$(\rho_{\meH^\nu})_\nu$ converges to $\rho$
with respect to the topology $\meT_\ph$.
\end{defn}

We generalize this notion for general $r>0$ as follows. 

\begin{defn}\label{defn: app-inner-general}
Let $N$ be a factor. 
Let $\rho\in\End(N)$ and $r>0$. 
We say that $\rho$ is an \emph{approximately inner endomorphism of rank} 
$r$ 
if there exist sequences of partial isometries 
$\{v_i^\nu\}_{i=1}^{[r]+1}\subs N$, $\nu\in\N$,
such that 
\begin{enumerate}
\item 
$(v_i^\nu)^* v_i^\nu=1$ for $1\leq i\leq [r]$ for all $\nu\in\N$, 

\item 
if $r\in\N$, $v_{[r]+1}^\nu=0$ for all $\nu\in\N$, 

\item 
$\displaystyle\sum_{i=1}^{[r]+1}v_i^\nu (v_i^\nu)^*=1$ for all $\nu\in\N$, 

\item 
$\displaystyle\lim_{\nu\to\infty}
\left\| \frac{1}{r}\,v_i^\nu \vph -(\vph\circ\ph_\rho) v_i^\nu\right\|=0$ 
for all $\vph\in N_*$. 
\end{enumerate} 
\end{defn}

We denote by $\oInt_r(N)$
the set of approximately inner endomorphisms of rank $r$.
By definition, we have $\oInt(N)=\Aut(M)\cap \oInt_1(N)$.

\subsection{Locally trivial subfactors}\label{LTS}
We recall locally trivial subfactors introduced in \cite[Chapter 2]{Po1}.
Let $P$ be a factor.
For $\rho_0:=\id$ and $\rho_1:=\rho\in\End(P)_0$,
the locally trivial subfactor
$N^{(\id,\rho)}\subs M^{(\id,\rho)}$ is defined as follows:
\begin{align*}
M^{(\id,\rho)}
=&\,
P\oti M_{2}(\C),
\\
N^{(\id,\rho)}
=&\,
\left\{x\oti e_{00}+\rho(x)\oti e_{11} \mid x\in P\right\}, 
\end{align*}
where $\{e_{ij}\}_{i,j=0}^1$ denotes a system of matrix units of $M_2(\C)$.
The canonical isomorphism from $P$ onto $N^{(\id,\rho)}$
is denoted by $\al^{(\id,\rho)}$.

For $\mu_0,\mu_1>0$ with $\mu_0+\mu_1=1$,
we set $\mu:=(\mu_0,\mu_1)$.
We define the unital completely positive map $\ph_\rho^\mu \col M^{(\id,\rho)}\ra P$ by
\[
\ph_\rho^\mu(x)=\sum_{i=0}^1 \mu_i \ph_{\rho_i}(x_{ii}), 
\quad x\in M^{(\id,\rho)}. 
\]
Then $\ph_\rho^\mu$ has the following property:
\[
\ph_\rho^\mu(\al^{(\id,\rho)}(a)x\al^{(\id,\rho)}(b))=a\ph_\rho^\mu(x)b. 
\]
This implies the map
$E^\mu:=\al^{(\id,\rho)}\circ \ph_\rho^\mu$ is
a faithful normal conditional expectation
from $M^{(\id,\rho)}$ onto $N^{(\id,\rho)}$.
By using the local index formula \cite[Theorem 4.4]{Ko}, 
we have the following. 

\begin{lem}\label{lem: IndE}
One has 
$\Ind (E^\mu)
=\mu_0^{-1}+\mu_1^{-1}d(\rho_1)^2
=\mu_0^{-1}+\mu_1^{-1}\Ind (E_\rho)$. 
\end{lem}

\subsection{Ultraproduct von Neumann algebras and Central sequence inclusions}

We recall the notion of ultraproduct von Neumann algebras
and central sequence inclusions.
Our standard references are \cite{MT1,Oc1,Po1}.

Let $M$ be a von Neumann algebra and $\om$ a free ultrafilter on $\N$.
Denote by $\meT_\om(M)\subs \el^\infty(\N,M)$ the $C^*$-subalgebra
which consists of sequences $\om$-converging to $0$ in the strong* topology.
Let $\meN(\meT_\om(M))$ be
the $C^*$-subalgebra of $\el^\infty(\N,M)$ normalizing $\meT_\om(M)$.
Then the quotient $C^*$-algebra
$M^\om:=\meN(\meT_\om(M))/\meT_\om(M)$
has a predual and hence is a von Neumann algebra.
We call $M^\omega$ the ultraproduct von Neumann algebra of $M$.
The quotient map is denoted by $\pi_\om$.
We say that $(x^\nu)_\nu\in \el^\infty(\N, M)$
is a representing sequence of $x\in M^\omega$
if $x=\pi_\om((x^\nu)_\nu)$.
We denote by $\ta^\om$ the canonical faithful normal conditional expectation
from $M^\om$ onto $M$, that is,
for $x=\pi_\om((x^\nu)_\nu)\in M^\om$,
we have $\displaystyle\ta^\om(x)=\lim_{\nu\to\om}x^\nu$,
where the ultralimit is
taken with respect to the $\si$-weak topology of $M$.
For $\vph\in M_*$, we define the normal functional $\vph^\om\in (M^\om)_*$
by $\displaystyle\vph^\om(x):=\vph(\ta^\om(x))$ for $x\in M^\om$.

Next we consider an inclusion $N\stackrel{E}{\subs} M$,
where $E$ is a faithful normal conditional expectation from $M$
onto $N$.
We define the following $C^*$-subalgebras in $\el^\infty(\N,M)$:
\begin{align*}
M_\om^0(E)
=&\,
\{(x^\nu)_\nu\in \el^\infty(\N,M)
\mid \lim_{\nu\to\om}\|[\vph\circ E,x^\nu]\|=0
\ \mbox{for all}\ \vph\in N_*\},
\\
N_\om^0(E)
=&\,
\{(x^\nu)_\nu\in \el^\infty(\N,N)
\mid \lim_{\nu\to\om}\|[\vph,x^\nu]\|=0
\ \mbox{for all}\ \vph\in N_*\}.
\end{align*}
Then we have an inclusion $N_\om^0(E)\subs M_\om^0(E)$
with the conditional expectation
$E_\om^0\col M_\om^0(E)\ra N_\om^0(E)$ defined by
$E_\om^0((x^\nu)_\nu)=(E(x^\nu))_\nu$.
We define the central sequence von Neumann algebras $M_\omega(E)$ and $N_\omega(E)$
by
\[
M_\om(E):=M_\om^0(E)/\meT_\om(M),
\quad
N_\om(E):=(N_\om^0(E)+\meT_\om(M))/\meT_\om(M). 
\]
Since $E_\om^0$ preserves $\meT_\om(M)$,
that is, $E_\om^0(\meT_\om(M))\subs \meT_\om(M)$,
we can naturally define the conditional expectation
$E_\om\col M_\om(E)\ra N_\om(E)$, which is faithful and normal.
The inclusion $N_\om(E)\stackrel{E_\om}{\subs} M_\om(E)$
is called the central sequence inclusion of $N\stackrel{E}{\subs} M$.
Note that $M_\om(E)$ is finite.
Indeed, a functional $\vph^\om\circ E_\om$
is a faithful normal tracial state
for all faithful state $\vph\in N_*$.

If $N=M$, we denote by $M_\om$ for $M_\om(\id_M)$.
Elements in $M_\omega^0(\id_M)$ are said to be $\omega$-centralizing.

Now we consider central sequence inclusions arising from locally trivial subfactors.
Let $P$ be a factor and $\rho\in \End(P)_0$.
For each $\mu_0,\mu_1>0$ with $\mu_0+\mu_1=1$,
the following locally trivial subfactor is
defined as in the previous subsection:
\[
N^{(\id,\rho)}\stackrel{E^\mu}{\subs} M^{(\id,\rho)}.
\]
Consider its central sequence inclusion
\[
N_\om^{(\id,\rho)}(E^\mu)
\stackrel{E_\om^\mu}{\subs} M_\om^{(\id,\rho)}(E^\mu). 
\] 
Note that $M_\om^{(\id,\rho)}(E^\mu)$ is a subalgebra of
$P^\om \oti M_{2}(\C)$.

\begin{lem}\label{lem: mux}
Set 
$
x:=\sum_{i,j=0}^1 x_{ij}\oti e_{ij}\in P^\om\oti M_{2}(\C)$.
Then $x\in M_\om^{(\id,\rho)}(E^\mu)$ 
if and only if 
for each $0 \leq i,j\leq 1$, 
a representing sequence $(x_{ij}^\nu)_{\nu}$ of $x_{ij}$ satisfies 
\[
\lim_{\nu\to\om}
\|\mu_i (\vph\circ \ph_{\rho_i}) x_{ij}^\nu
-\mu_j x_{ij}^\nu (\vph\circ \ph_{\rho_j}) 
\|=0\quad\mbox{for all}\ 
\vph\in P_*. 
\]
\end{lem}
\begin{proof}
Let $\vph\in P_*$.
Then for
$
y=\sum_{i,j=0}^1y_{ij}\oti e_{ij}$,
we have
\[
\vph\circ(\al^{(\id,\rho)})^{-1}\circ E^\mu((x_{ij}^\nu \oti e_{ij}) y)
=\mu_i\vph(\ph_{\rho_i}(x_{ij}^\nu y_{ji})), 
\]
and
\[
\vph\circ(\al^{(\id,\rho)})^{-1}\circ E^\mu(y(x_{ij}^\nu \oti e_{ij}) )
=\mu_j\vph(\ph_{\rho_j}(y_{ji}x_{ij}^\nu)). 
\]
Setting
$
x^\nu:=\sum_{i,j=0}^1 x_{ij}^\nu\oti e_{ij}$, we have
\begin{align*}
[\vph\circ(\al^{(\id,\rho)})^{-1}\circ E^\mu,x^\nu](y)
=&\,
\sum_{i,j=0}^1
\mu_i\vph(\ph_{\rho_i}(x_{ij}^\nu y_{ji}))
-
\mu_j\vph(\ph_{\rho_j}(y_{ji}x_{ij}^\nu))
\\
=&\,
\sum_{i,j=0}^1
\left(
\mu_i(\vph\circ \ph_{\rho_i})x_{ij}^\nu 
-
\mu_jx_{ij}^\nu (\vph\circ \ph_{\rho_j})
\right)
(y_{ji}). 
\end{align*}
This implies the following inequalities:
\[
\|[\vph\circ(\al^{(\id,\rho)})^{-1}\circ E^\mu,x^\nu]\|
\leq 
\sum_{i,j=0}^1
\left\|
\mu_i(\vph\circ \ph_{\rho_i})x_{ij}^\nu 
-
\mu_jx_{ij}^\nu (\vph\circ \ph_{\rho_j})
\right\|
\]
and 
\[
\left\|
\mu_i(\vph\circ \ph_{\rho_i})x_{ij}^\nu 
-
\mu_jx_{ij}^\nu (\vph\circ \ph_{\rho_j})
\right\|
\leq 
\|[\vph\circ(\al^{(\id,\rho)})^{-1}\circ E^\mu,x^\nu]\|. 
\]
Therefore $x\in M_\om^{(\id,\rho)}(E)$ if and only if
$\left\|
\mu_i(\vph\circ \ph_{\rho_i})x_{ij}^\nu 
-
\mu_jx_{ij}^\nu (\vph\circ \ph_{\rho_j})
\right\|\to0$ as $\nu\to\om$ for all $0\leq i,j\leq 1$.
\end{proof}

The previous lemma implies that the projection $1\oti e_{ii}$
is in the relative commutant
$(N_\om^{(\id,\rho)}(E^\mu))'\cap M_\om^{(\id,\rho)}(E^\mu)$. 
Then we see that $x\oti e_{ii}\in P^\om\oti \C e_{ii}$ is contained
in $(1\oti e_{ii})M_\om^{(\id,\rho)}(E^\mu)(1\oti e_{ii})$
if and only if
$\|[\vph\circ\ph_{\rho_i},x^\nu]\|\to0$ as $\nu\to\om$
for all $\vph\in P_*$. 
This means $\pi_\om((x^\nu)_\nu)$ is contained in $P_\om(E_{\rho_i})$,
where $\rho_i(P_\om)\stackrel{(E_{\rho_i})_\om}{\subs} P_\om(E_{\rho_i})$
is the central sequence inclusion of
$\rho_i(P)\stackrel{E_{\rho_i}}{\subs} P$.
Summarizing these arguments, we have
\begin{align*}
(1\oti e_{ii})M_\om^{(\id,\rho)}(E^\mu)(1\oti e_{ii})
=&\,
P_\om(E_{\rho_i})\oti \C e_{ii}, 
\\
(1\oti e_{ii})N_\om^{(\id,\rho)}(E^\mu)(1\oti e_{ii})
=&\,
\rho_i(P_\om)\oti \C e_{ii}. 
\end{align*}
Namely, the central sequence inclusion of
$\rho_i(P)\stackrel{E_{\rho_i}}{\subs} P$ is isomorphic to the corner
of $N_\om^{(\id,\rho)}(E^\mu)\subs M_\om^{(\id,\rho)}(E^\mu)$
cut by $1\oti e_{ii}$.
Hence we have the following lemma.

\begin{lem}
One has 
\begin{align*}
(1\oti e_{00})M_\om^{(\id,\rho)}(E^\mu)(1\oti e_{00})
=&\,
P_\om \oti \C e_{00},
\\
(1\oti e_{00})Z(M_\om^{(\id,\rho)}(E^\mu))(1\oti e_{00})
=&\,
Z(P_\om)\oti \C e_{00}.
\end{align*} 
\end{lem}

The following proposition clarifies the relations between
approximately inner endomorphisms and central sequence inclusions.

\begin{prop}\label{prop: common-center}
Let $N^{(\id,\rho)} \stackrel{E^\mu}{\subs} M^{(\id,\rho)}$ as before.
Let $\CTr_{M_\om^{(\id,\rho)}(E^\mu)}$ be the center valued
trace of $M_\om^{(\id,\rho)}(E^\mu)$.
Then the following statements are equivalent:
\begin{enumerate}
\item 
The central support of $1\oti e_{00}$ in $M_\om^{(\id,\rho)}(E^\mu)$ 
is equal to $1$. 

\item
$Z(N_\om^{(\id,\rho)}(E^\mu))=Z(M_\om^{(\id,\rho)}(E^\mu))$. 

\item
$\CTr_{M_\om^{(\id,\rho)}(E^\mu)}(1\oti e_{ii})=\mu_i$ for $i=0,1$.

\item 
$\rho \in \oInt_{\mu_1/\mu_0}(P)$. 

\item 
There exist a finite set $I$ and 
$x_j=\pi_\om((x_j^\nu)_{\nu})\in P^\om$, $j\in I$ 
such that 
\begin{enumerate}
\renewcommand{\labelenumii}{(\alph{enumii})}

\item $\displaystyle \bigvee_{j\in I}s(x_jx_j^*)=1$, 

\item 
$\displaystyle
\lim_{\nu\to\om}
\|\mu_1 (\vph\circ\ph_{\rho})x_j^\nu-\mu_0 x_j^\nu \vph\|
=0
\quad\mbox{for all}\ \vph\in P_*. 
$
\end{enumerate}
\end{enumerate}
\end{prop}
\begin{proof}
(1)$\Rightarrow$(2).
The inclusion
$Z(N_\om^{(\id,\rho)}(E^\mu))\subs Z(M_\om^{(\id,\rho)}(E^\mu))$
always holds \cite[Corollary 1.3.7 (i)]{Po1}.
Let $z\in Z(M_\om^{(\id,\rho)}(E^\mu))$.
Since $Z(M_\om^{(\id,\rho)}(E^\mu))(1\oti e_{00})=Z(P_\om)\oti e_{00}$
by the previous lemma,
there exists $z_0\in Z(P_\om)$ such that
$z(1\oti e_{00})=z_0\oti e_{00}$.
Set $z'=(\al^{(\id,\rho)})^\om(z_0)$, where
$(\al^{(\id,\rho)})^\om$ is an embedding $P^\om\ra M_\om^{(\id,\rho)}(E^\mu)$
naturally defined through $\al^{(\id,\rho)}$.
Then $z'\in Z(N_\om^{(\id,\rho)}(E^\mu))$,
and $z'\in Z(M_\om^{(\id,\rho)}(E^\mu))$.
By assumption, $z(1\oti e_{00})=z'(1\oti  e_{00})$ yields $z=z'$.

(2)$\Rightarrow$(3).
Let $\CTr_{N_\om^{(\id,\rho)}(E^\mu)}$ and $\CTr_{M_\om^{(\id,\rho)}(E^\mu)}$
be the center valued traces
of $N_\om^{(\id,\rho)}(E^\mu)$
and $M_\om^{(\id,\rho)}(E^\mu)$, respectively.
Then the maps
$E_\om^\mu\circ \CTr_{M_\om^{(\id,\rho)}(E^\mu)}$ and
$\CTr_{N_\om^{(\id,\rho)}(E^\mu)}\circ E_\om^\mu$
are faithful normal conditional
expectations from $M_\om^{(\id,\rho)}(E^\mu)$
onto $Z(N_\om^{(\id,\rho)}(E^\mu))$.
Since those conditional expectations preserve
a faithful normal trace of the form $\vph^\om \circ E_\om^\mu$, $\vph\in N_*$,
we have the equality
\[
E_\om^\mu\circ \CTr_{M_\om^{(\id,\rho)}(E^\mu)} 
=\CTr_{N_\om^{(\id,\rho)}(E^\mu)}\circ E_\om^\mu. 
\]
Since $\CTr_{M_\om^{(\id,\rho)}(E^\mu)}(1\oti e_{ii})$ is contained
in $Z(M_\om^{(\id,\rho)}(E^\mu))$,
it is also contained in $Z(N_\om^{(\id,\rho)}(E^\mu))$
by the assumption of (2).
Using $E_\om^\mu(1\oti e_{ii})=\mu_i$, we have
\begin{align*}
\CTr_{M_\om^{(\id,\rho)}(E^\mu)}(1\oti e_{ii})
=&\,
E_\om^\mu(\CTr_{M_\om^{(\id,\rho)}(E^\mu)}(1\oti e_{ii}))
=
\CTr_{N_\om^{(\id,\rho)}(E^\mu)}(E_\om^\mu(1\oti e_{ii}))
\\
=&\,
\mu_i. 
\end{align*}

(3)$\Rightarrow$(4). 
Since $1\oti e_{00}$ and $1\oti e_{11}$ have central traces
$\mu_0$ and $\mu_1$, respectively,
there exist partial isometries $\{u_j\}_{j=1}^{[\mu_1/\mu_0]+1}$
such that
$u_j=(1\oti e_{11})u_j (1\oti e_{00})$,
$u_j^* u_j=(1\oti e_{00})$ for $1\leq j\leq [\mu_1/\mu_0]$
and
$
\sum_{j=1}^{[\mu_1/\mu_0]+1}u_j u_j^*=1\oti e_{11}$.
Let $\{v_j\}_{j=1}^{[\mu_1/\mu_0]+1}$ be as $u_j=v_j\oti e_{10}$.
Take a representing sequence of $v_j=(v_j^\nu)_\nu$
with $(v_j^\nu)^* v_j^\nu=1$ for $1\leq j\leq [\mu_1/\mu_0]$ and
$\sum_{j=1}^{[\mu_1/\mu_0]+1}v_j^\nu (v_j^\nu)^*=1$.
By Lemma \ref{lem: mux},
$v_j^\nu$ satisfies
\[
\lim_{\nu\to\om}
\|\mu_1 (\vph\circ\ph_{\rho_1})v_j^\nu- \mu_0 v_j^\nu\vph\|=0 
\quad\mbox{for all}\ \vph\in P_*. 
\]
This means that $\rho\in \oInt_{\mu_1/\mu_0}(P)$.

(4)$\Rightarrow$(5). 
Take partial isometries $(v_j^\nu)_\nu$,
$1\leq j\leq [\mu_1/\mu_0]+1$,
as Definition \ref{defn: app-inner-general}
for $\rho\in\oInt_{[\mu_1/\mu_0]}(P)$.
We set $I:=\{1,\dots,[\mu_1/\mu_0]+1\}$ and
$x_j^\nu=v_j^\nu\oti e_{10}$.
Then by Lemma \ref{lem: mux},
the sequence $(x_j^\nu)_\nu$ represents
the element $x_j$ in $M_\om^{(\id,\rho)}(E^\mu)$, 
we have
$
\sum_{j\in I} x_j x_j^*=1$.
In particular,
$
\bigvee_{j\in I} s(x_j x_j^*)=1$.

(5)$\Rightarrow$(1).
Take a finite family
$\{x_j\}_{j\in I}\subs P^\om$ which satisfies
the conditions in (5).
The condition (b) implies that $x_j\oti e_{10}\in M_\om^{(\id,\rho)}(E^\mu)$.
Let $x_j=v_j|x_j|$ be the polar decomposition.
Then $v_j\oti e_{10}\in M_\om^{(\id,\rho)}(E^\mu)$,
and
$\bigvee_{j\in I} s(v_j v_j^*)=1$ from (a).
We assume that
$z(1\oti e_{00})=0$ for some $z\in Z(M_\om^{(\id,\rho)}(E^\mu))$.
Then $z(v_j\oti e_{10})=(v_j\oti e_{10})z=0$,
and we have $z(v_j v_j^*\oti e_{11})=0$.
Since
$\bigvee_{j\in I_i} s(v_j v_j^*)=1$,
we have $z(1\oti e_{11})=0$.
Hence $z$ must be $0$, and equivalently the central support of $1\oti e_{00}$
in $M_\om^{(\id,\rho)}(E^\mu)$ is equal to $1$.
\end{proof}

Now we generalize \cite[Proposition 2.3 (ii)]{Po1} 
to the case of endomorphisms. 
Readers are referred to \cite[Definition 2.1]{Po1} for the definition of 
the approximately inner inclusion of factors. 

\begin{thm}\label{thm: rho-app}
Let $\mu_0,\mu_1>0$ with $\mu_0+\mu_1=1$. 
Then the following statements are equivalent:
\begin{enumerate}
\item 
The inclusion $N^{(\id,\rho)}\stackrel{E^\mu}{\subs}M^{(\id,\rho)}$ 
is approximately inner. 

\item 
The inclusion $\rho(P)\stackrel{E_\rho}{\subs} P$ 
is approximately inner and 
$\rho\in \oInt_{\mu_1/\mu_0}(P)$. 
\end{enumerate}

If $\rho$ is irreducible, the above statements are also 
equivalent with the following:

\hspace{2.5pt}
{\upshape(3)}
The inclusion $\rho(P)\stackrel{E_\rho}{\subs} P$ is approximately inner 
and 
there exists a sequence 
\\
\hspace{36pt}of partial isometries 
$(v^\nu)_{\nu=1}^\infty$ in $P$ 
such that 
\[
\lim_{\nu\to\infty}
\|\mu_1 (\vph\circ\ph_{\rho})v^\nu-\mu_0 v^\nu \vph\|
=0
\quad\mbox{for all}\ \vph\in P_*,
\]
\hspace{36pt}and $v^\nu$ is an isometry (coisometry) 
when $\mu_0<\mu_1$ (resp. $\mu_0\geq \mu_1$). 
\end{thm}
\begin{proof}
(1)$\Rightarrow$(2). 
The inclusion $\rho(P)\stackrel{E_\rho}{\subs} P$ is 
isomorphic to 
the reduced inclusion 
$N^{(\id,\rho)} (1\oti e_{11})
\subs(1\oti e_{11})M^{(\id,\rho)} (1\oti e_{11})$ 
with a conditional expectation 
$(E^\mu)_{1\oti e_{11}}$. 
Hence the inclusion $\rho(P)\stackrel{E_\rho}{\subs} P$ 
is approximately inner by \cite[Proposition 2.7 (i)]{Po1}. 

Next we show that $\rho\in\oInt_{\mu_1/\mu_0}(P)$.
By Proposition \ref{prop: common-center},
it suffices to prove that the central support of $1\oti e_{00}$
in $M_\om^{(\id,\rho)}(E^\mu)$ is equal to $1$.
We make use of a Pimsner-Popa basis for $E_\om^\mu$ as follows.

On the corner at $1\oti e_{00}$, $m^{00}:=1\oti e_{00}$ is a basis
because
$N_\om^{(\id,\rho)}(E^\mu)(1\oti e_{00})
=(1\oti e_{00})M_\om^{(\id,\rho)}(E^\mu)(1\oti e_{00})$ holds.

Next we consider the corner at $1\oti e_{11}$. 
By \cite[Proposition 2.2]{Po1}, 
the central sequence inclusion 
$\rho(P_\om) \stackrel{(E_\rho)_\om}{\subs} P_\om(E_\rho)$ 
is a $\Ind(E_\rho)^{-1}$-Markov inclusion, 
and so is
$N_\om^{(\id,\rho)}(E^\mu)(1\oti e_{11})
\subs
(1\oti e_{11}) M_\om^{(\id,\rho)}(E^\mu)(1\oti e_{11})$ 
with a conditional expectation 
$(E_\om^\mu)_{(1\oti e_{11})}$. 
Take an orthonormal basis $(m_j^{11})_{j\in I_{11}}$ 
for $(E_\om^\mu)_{(1\oti e_{11})}$. 
Then 
$\mu_1^{-1}E_\om^\mu((m_j^{11})^*m_k^{11})
=\de_{jk}f_j$
is a projection in
$(1\oti e_{11})M_\om^{(\id,\rho)}(E^\mu)(1\oti e_{11})$,
and 
\begin{equation}\label{eq: m11}
\sum_{j\in I_{11}}m_j^{11} (m_j^{11})^*=\Ind(E_\rho)(1\oti e_{11}). 
\end{equation}

Now we have an orthonormal family 
$\{\mu_0^{-1/2} m^{00}\}\cup \{\mu_1^{-1/2} m_j^{11}\}_{j\in I_{11}}$ 
with respect to the expectation $E_\om^\mu$. 
By adding other elements $\{m_p\}_{p\in I}$, 
we can extend the orthonormal family to an orthonormal basis for $E_\om$. 

Then the Markov property of $E_\om^\mu$ implies 
the following equality: 
\begin{align}
\Ind(E^\mu)
=&\,
\mu_0^{-1}m^{00}(m^{00})^* 
+
\sum_{p\in I}m_p m_p^*
+
\sum_{j\in I_{11}}\mu_1^{-1}m_j^{11}(m_j^{11})^*
\notag\\
=&\,
\mu_0^{-1}(1\oti e_{00})
+
\sum_{p\in I}m_p m_p^*
+
\mu_1^{-1}\Ind(E_\rho)(1\oti e_{11})
\quad(\mbox{by}\ (\ref{eq: m11}))
\label{eq: Emu}.
\end{align}
We prove that the central support of $1\oti e_{00}$ in $M_\om^{(\id,\rho)}(E)$ 
is equal to $1$.
Assume that $z(1\oti e_{00})=0$ 
for a central projection $z$ in $M_\om^{(\id,\rho)}(E)$.
Note that $m_p$ is off-diagonal
because $m_j^{00}$ and $m_j^{11}$ are orthonormal bases of 
the corner of $M_\om^{(\id,\rho)}(E^\mu)$ reduced by $1\oti e_{00}$ 
and $1\oti e_{11}$, respectively.
Hence we have $zm_p=0$.
By multiplying $z$ to (\ref{eq: Emu}), we obtain 
\[
\Ind (E^\mu) z
=
\mu_1^{-1} \Ind(E_\rho) z (1\oti e_{11}). 
\]
However the formula of Lemma \ref{lem: IndE} 
implies that 
$\Ind (E^\mu)>\mu_1^{-1} \Ind(E_\rho)$. 
This shows that $z$ must be equal to $0$, 
and the central support of $1\oti e_{00}$ is equal to $1$. 

(2)$\Rightarrow$(1). 
We prove that
$N_\om^{(\id,\rho)}(E^\mu)
\stackrel{(E^\mu)_\om}{\subs} M_\om^{(\id,\rho)}(E^\mu)$ 
is an $\Ind(E^\mu)^{-1}$-Markov inclusion. 
Then $N^{(\id,\rho)}\stackrel{E^\mu}{\subs} M^{(\id,\rho)}$ is 
approximately inner by \cite[Proposition 2.2]{Po1}. 
By Proposition \ref{prop: common-center}, 
$1\oti e_{00}$ and $1\oti e_{11}$ have the scalar central traces 
$\mu_0$ and $\mu_1$, respectively. 
We may assume $\mu_0<\mu_1$ because the similar proof works
in the case of $\mu_0\geq \mu_1$. 
This allows us to take a family of partial isometries 
$\{v_j\oti e_{10}\}_{j=1}^{m+1}$ in $M_\om^{(\id,\rho)}(E^\mu)$ 
such that 
$v_j^* v_k=0$ for $j\neq k$, $v_j^*v_j=1$ for $j\leq m$ 
and
$\sum_{j=1}^{m+1} v_j v_j^*=1$.
We construct a basis for $E^\mu$ as follows.

For the $e_{00}$-entry, we set $m^{00}=\mu_0^{-1/2}(1\oti e_{00})$.
For the $e_{11}$-entry, we take a basis $\{m_j^{11}\}_{j\in I_{11}}$ 
as before.

For the $e_{10}$-entry, 
we set $m_{j}^{10}:=(v_j\oti e_{10}) m^{00}$. 
Then $\{m_{j}^{10}\}_{j=1}^{m+1}$ is an orthonormal family satisfying 
\begin{equation}\label{eq: m10}
\mu_0^{-1}(1\oti e_{11})=\sum_{j=1}^{m+1} m_j^{10}(m_j^{10})^*,
\end{equation}
and 
\begin{equation}\label{eq: 10}
(1\oti e_{11}) M_\om^{(\id,\rho)}(E^\mu)(1\oti e_{00})
=\sum_{j,k=1}^{m+1} m_{j}^{10} N_\om^{(\id,\rho)}(E^\mu). 
\end{equation}
For the $e_{01}$-entry, 
we set $m_{j}^{01}:=(v_{i_0}^*\oti e_{01})m_j^{11}$ for a fixed $i_0$. 
Although $\{m_j^{01}\}_j$ is not an orthonormal family, 
for any $x \in M_\om^{(\id,\rho)}(E^\mu)$ we have 
\begin{align*}
\sum_{j=1}^{m+1} m_j^{01} (E^\mu)_\om ((m_j^{01})^* x)
=&\,
\sum_{j=1}^{m+1} 
(v_{i_0}^*\oti e_{01})m_j^{11} 
(E^\mu)_\om (((v_{i_0}^*\oti e_{01})m_j^{11})^* x)
\\
=&\,
\sum_{j=1}^{m+1} 
(v_{i_0}^*\oti e_{01})m_j^{11}
(E^\mu)_\om((m_j^{11})^* (v_{i_0}\oti e_{10})x(1\oti e_{11}))
\\
=&\,
(v_{i_0}^*\oti e_{01}) (v_{i_0}\oti e_{10})x(1\oti e_{11})
\\
=&\,
(1\oti e_{00})x(1\oti e_{11}). 
\end{align*}
This shows that $(m_j^{01},(m_j^{01})^*)_j$
is a quasi-basis for $e_{01}$-entry of $M_\om^{(\id,\rho)}(E^\mu)$
in the sense of \cite{Wa}.
Moreover we have 
\begin{equation}\label{eq: m01}
\sum_{i=1}^{m+1}m_i^{01}(m_i^{01})^*=\mu_1^{-1}\Ind(E_\rho)(1\oti e_{00}). 
\end{equation}

Then the family $\{(m_i^{j},(m_i^{j})^*)\}_{i,j}$ is a quasi-basis
for $(E^\mu)_\om$.
Using (\ref{eq: m11}), (\ref{eq: m10}) and (\ref{eq: m01}),
we have 
\begin{align*}
\sum_{i,j}m_i^{j}(m_i^{j})^*
=&\,
\sum_{i}m_i^{00}(m_i^{00})^*
+\sum_{i}m_i^{11}(m_i^{11})^*
+\sum_{i}m_i^{10}(m_i^{10})^*
+\sum_{i}m_i^{01}(m_i^{01})^*
\\
=&\,
\mu_0^{-1}(1\oti e_{00})
+
\mu_1^{-1}\Ind(E_\rho)(1\oti e_{11})
\\
&\quad
+
\mu_0^{-1}(1\oti e_{11})
+
\mu_1^{-1}\Ind(E_\rho)(1\oti e_{00})
\\
=&\,
(\mu_0^{-1}+\mu_1^{-1}\Ind(E_\rho))
(1\oti e_{00}+1\oti e_{11})
\\
=&\,
\Ind(E^\mu)\qquad(\mbox{by Lemma}\ \ref{lem: IndE}).
\end{align*}
Hence the inclusion
$N_\om^{(\id,\rho)}(E^\mu)\stackrel{(E^\mu)_\om}{\subs}M_\om^{(\id,\rho)}(E^\mu)$ 
is $\Ind(E^\mu)^{-1}$-Markov, and 
the inclusion $N^{(\id,\rho)}\stackrel{E^\mu}{\subs} M^{(\id,\rho)}$ 
is approximately inner by \cite[Proposition 2.2]{Po1}. 

(2)$\Rightarrow$(3). 
In fact, the irreducibility of $\rho$ is unnecessary. 
By Proposition \ref{prop: common-center}, 
$e_{00}\oti 1$ and $e_{11}\oti 1$ have scalar central traces 
$\mu_0$ and $\mu_1$, respectively. 
If $\mu_0<\mu_1$, then there exists an isometry 
$v\in P^\om$ such that $v\oti e_{10}\in M_\om^{(\id,\rho)}(E^\mu)$. 
Let $(v^\nu)_\nu$ be a representing sequence of $v$ 
which consists of isometries. 
By Lemma \ref{lem: mux}, $(v^\nu)_\nu$ satisfies 
\[
\lim_{\nu\to\om}
\|\mu_1 (\vph\circ\ph_{\rho})v^\nu-\mu_0 v^\nu \vph\|
=0. 
\]
Then we take a subsequence of $(v^\nu)_\nu$ so that
the above equality holds as $\nu\to\infty$.

If $\mu_0\geq \mu_1$, then the similar argument still works, and 
we can find a desired sequence which consists of coisometries. 

(3)$\Rightarrow$(2). 
We prove that the central support of $1\oti e_{00}$ 
is equal to $1$. 
Take a sequence of 
partial isometries $(v^\nu)_\nu$ as in (3). 
Put $v:=\pi_\om((v^\nu)_\nu) \in P^\om$. 
Then the element $v\oti e_{10}$ is contained in 
$M_\om^{(\id,\rho)}(E^\mu)$ by Lemma \ref{lem: mux}. 

If $\mu_0\geq \mu_1$, 
we have $1\oti e_{11}=(v\oti e_{10})(v\oti e_{10})^*$ 
which is equivalent to $v^*v\oti e_{00}$. 
Hence the central support of $1\oti e_{00}$ is equal to 
1 in this case. 

If $\mu_0<\mu_1$, 
we have $1\oti e_{00}=(v\oti e_{10})^*(v\oti e_{10})$. 
Suppose that a non-zero projection $z\in Z(M_\om^{(\id,\rho)}(E^\mu))$ 
satisfies $z(1\oti e_{00})=0$. 
Then $z$ is of the form $z=z_1\oti e_{11}$, where $z_1$ is a projection 
in $Z(P_\om(E_\rho))$. 
Since $z$ is central, we have 
$z(v\oti e_{10})=(v\oti e_{10})z=(v\oti e_{10})(1\oti e_{00})z=0$. 
This means $z_1 vv^*=0$. 
Note that $vv^*\in P_\om(E_\rho)$, and 
$\ta^\om(vv^*)\in \rho(P)'\cap P=\C$ because $\rho$ is irreducible. 
Now for $z_1$ and $v$, 
we apply the Fast Reindexation Lemma \cite[Lemma 5.3]{Oc1}. 
Then there exists a map $\Ps\col W^*(v)\ra P^\om$ such that 
$\ta^\om(z_1 \Ps(vv^*))=\ta^\om(z_1)\ta^\om(vv^*)$, 
which is not equal to $0$. 
In particular, $z_1\Ps(v)\neq0$. 
By the construction of the fast reindexation map $\Ps$ in the proof of 
\cite[Lemma 5.3]{Oc1}, 
we may assume that $\Ps(v)$ is given by mapping $(v^\nu)_\nu$ 
to some subsequence $(v^{\nu(k)})_k$. 
Hence $\Ps(v)\oti e_{10}$ is also an element of 
$M_\om^{(\id,\rho)}(E^\mu)$. 
However, 
$z_1\Ps(v)\oti e_{10}=z(\Ps(v)\oti e_{10})=(\Ps(v)\oti e_{10})z=0$, 
and this is a contradiction. 

Therefore in the both cases, the central support of $1\oti e_{00}$ 
is equal to $1$. 
Then by Proposition \ref{prop: common-center},
$\rho\in \oInt_{\mu_1/\mu_0}(P)$. 
\end{proof}

\subsection{Basic properties of approximately inner endomorphisms}

Let $M$ be a factor and $\rho\in \oInt_r(M)$ with $r>0$.
Then for any $u\in U(M)$, $\Ad u\circ \rho$ is also in $\oInt_r(M)$.
Hence approximate innerness is a property for sectors.
The sector space $\Sect(M)$ has the basic operations, i.e.,
composition, decomposition, direct sum and conjugation.
We study how approximate innerness behaves for these operations.
For sector theory, readers are referred to \cite{Iz,L1,L2}.

\begin{lem}[Decomposition rule]
Let $\rho\in\End(M)_0$ and $[\rho]=\oplus_{i\in I} m_i[\rho_i]$
be the irreducible decomposition
where $m_i$ is the multiplicity of $[\rho_i]$ in $[\rho]$.
Then
$\rho\in\oInt_r(M)$ if and only if $\rho_i\in \oInt_{rd(\rho_i)/d(\rho)}(M)$ 
for all $i\in I$. 
\end{lem}
\begin{proof}
Suppose that $\rho\in\oInt_r(M)$. 
Take $\{v_j^\nu\}_{j=1}^{[r]+1}$, $\nu\in\N$, 
as in Definition \ref{defn: app-inner-general}. 
Let $\{w_k^i\}_{k=1}^{m_i}$ be an orthonormal basis of 
the Hilbert space $(\rho_i,\rho)$. 
Then we have 
\[
d(\rho)\ph_\rho(x)
=\sum_{i\in I}\sum_{k=1}^{m_i}
d(\rho_i)\ph_{\rho_i}((w_k^i)^*x w_k^i). 
\]
For the proof of this equality, readers are referred to 
\cite[Lemma A.2]{Ma}.
Then for all $\vph\in M_*$ and $1\leq k\leq m_i$,
we have
\[
\lim_{\nu\to\infty}
\left \|\frac{1}{r} ((w_k^i)^* v_j^\nu)\cdot \vph
-\frac{d(\rho_i)}{d(\rho)}\vph\circ\ph_{\rho_i}\cdot((w_k^i)^* v_j^\nu)
\right\|
=0.
\]
It is equivalent to 
$\pi_\om(((w_k^i)^* v_j^\nu)_\nu)\oti e_{10}
\in M_\om^{(\id,\rho_i)}(E^{(\mu_0^i,\mu_1^i)})$, 
where $\mu_0^i+\mu_1^i=1$
and $\mu_0^i/\mu_1^i=d(\rho)/(r d(\rho_i))$.
Setting $x_{k,j}^{i,\nu}:=(w_k^i)^* v_j^\nu$, we have
\[\sum_{j=1}^{[r]+1}\sum_{k=1}^{m_i}
x_{k,j}^{i,\nu}(x_{k,j}^{i,\nu})^* 
=
\sum_{j=1}^{[r]+1}\sum_{k=1}^{m_i}(w_k^i)^* v_j^\nu (v_j^\nu)^* w_k^i
=
\sum_{k=1}^{m_i} (w_k^i)^* w_k^i 
=\dim (\rho_i,\rho). 
\]
By Proposition \ref{prop: common-center}, 
we see that $\rho_i\in \oInt_{rd(\rho_i)/d(\rho)}(M)$. 

Conversely we suppose that $\rho_i\in \oInt_{rd(\rho_i)/d(\rho)}(M)$ 
for all $i\in I$. 
For each $i\in I$, 
we take sequences of partial isometries 
$\{v_j^{i \,\nu}\}_{j=1}^{[rd(\rho_i)/d(\rho)]+1}$, $\nu\in\N$, 
as in Definition \ref{defn: app-inner-general}. 
Then for all $i,j$ and $\vph\in M_*$, it satisfies 
\[
\lim_{\nu\to\infty}
\left\|\frac{1}{r} 
v_j^{i\,\nu}\vph- \frac{d(\rho_j)}{d(\rho)}
(\vph\circ\ph_{\rho_i})\cdot v_j^{i\,\nu}
\right\|=0. 
\]
Then for all $1\leq k\leq m_i$, 
\[
\lim_{\nu\to\infty}
\left\|r^{-1} 
w_k^i v_j^{i\,\nu}\vph
-(\vph\circ\ph_{\rho})\cdot (w_k^i v_j^{i\,\nu})
\right\|=0. 
\]
Hence $\pi_\om((w_k v_j^{i\,\nu})_\nu)\oti e_{10}
\in M_\om^{(\id,\rho)}(E^{(\mu_0,\mu_1)})$, where $\mu_0+\mu_1=1$ 
and $\mu_0/\mu_1=1/r$. 

Since 
\[
\sum_{i\in I}\sum_{k=1}^{m_i}
\sum_{j=1}^{[rd(\rho_i)/d(\rho)]+1}
w_k^i v_j^{i\,\nu} (w_k^i v_j^{i\,\nu})^*
=1, 
\]
$\rho\in \oInt_r(M)$ by Proposition \ref{prop: common-center}. 
\end{proof}

\begin{cor}
Let $\rho,\si\in\End(M)_0$. 
Suppose that $\si\prec\rho$ and $\rho\in\oInt_r(M)$. 
Then $\si\in\oInt_{rd(\si)/d(\rho)}(M)$. 
\end{cor}
\begin{proof}
Let $[\si]=\oplus_i [\si_i]$ be the irreducible decomposition.
By applying the previous lemma to $\si_i\prec\rho$,
we have $\si_i\in \oInt_{rd(\si_i)/d(\rho)}(M)$.
Note that $rd(\si_i)/d(\rho)=(rd(\si)/d(\rho))d(\si_i)/d(\si)$.
Using again the previous lemma,
we see that $\si\in \oInt_{rd(\si)/d(\rho)}(M)$.
\end{proof}

On composition of endomorphisms, the following result holds. 

\begin{lem}[Composition rule]\label{lem: composition}
Let $\rho_i\in \oInt_{r_i}(N)$ for $i=1,2$. 
Then $\rho_1\circ\rho_2\in\oInt_{r_1 r_2}(N)$. 
\end{lem}
\begin{proof}
Take sequences of partial isometries 
$\{v_j^{1\, \nu}\}_{j=1}^{[r_1]+1}$ 
and $\{v_j^{2\, \nu}\}_{j=1}^{[r_2]+1}$ satisfying 
the conditions in Definition \ref{defn: app-inner-general} 
for $\rho_1$ and $\rho_2$, respectively. 
Then for all $\vph\in M_*$, $i=1,2$ and $1\leq j\leq [r_i]+1$, 
\[
\lim_{\nu\to\infty}
\left\|
r_i^{-1} v_j^{i\,\nu}\vph
-(\vph\circ\ph_{\rho_i})\cdot v_j^{i\,\nu}
\right\|
=0. 
\]
It is easy to see that 
\[
\lim_{\nu\to\infty}
\left\|
(r_1 r_2)^{-1} v_j^{1\,\nu} v_k^{2\,\nu}\vph
-
(\vph\circ\ph_{\rho_2}\circ\ph_{\rho_1})
\cdot v_j^{1\,\nu}v_k^{2\,\nu}
\right\|
=0 
\]
for all $1\leq j\leq [r_1]+1$ and $1\leq k\leq[r_2]+1$. 
Hence 
$\pi_\om\big{(}
(v_j^{1\,\nu}v_k^{2\,\nu}\oti e_{10})_\nu 
\big{)}
\in M_\om^{(\id,\rho_1\rho_2)}(E^{\mu_0,\mu_1})$, 
where $\mu_0+\mu_1=1$ and $\mu_0/\mu_1=1/(r_1r_2)$. 
Since 
\[
1
=\sum_{j=1}^{[r_1]+1}\sum_{k=1}^{[r_2]+1}
v_j^{1\,\nu}v_k^{2\,\nu}
(v_j^{1\,\nu}v_k^{2\,\nu})^*,
\]
$\rho_1\circ\rho_2\in\oInt_{r_1 r_2}(M)$ 
by Proposition \ref{prop: common-center}. 
\end{proof}

On conjugation, we have a result only for hyperfinite factors
(Corollary \ref{cor: conjugation-rule}). 

\subsection{Descriptions of $\oInt_r(N)$ for hyperfinite semifinite factors}

\begin{lem}\label{lem: I-factor}
Let $N$ be a type I factor,
then $\oInt_r(N)=\emptyset$ for all $r\nin\N$
and $\oInt_r(N)=\Int_r(N)$ for all $r\in \N$.
\end{lem}
\begin{proof}
Suppose $\rho\in \Int_d(N)\cap \oInt_r(N)$ for some $r>0$,
where $d=d(\rho)$.
We will show $r=d$.
Take sequences of partial isometries
$\{v_i^\nu\}_{i=1}^{[r]+1}\subs N$, $\nu\in\N$
such that
$(v_i^\nu)^* v_i^\nu=1$ for $1\leq i\leq [r]$,
$
\sum_{i=1}^{[r]+1}v_i^\nu (v_i^\nu)^*=1$ for all $\nu\in\N$
and
\begin{equation}\label{eq: Irv}
\lim_{\nu\to\infty}
\left\| r^{-1}v_i^\nu \vph -(\vph\circ\ph_\rho) v_i^\nu\right\|=0
\end{equation}
for all $1\leq i\leq [r]+1$, $\vph\in N_*$.
Take a Hilbert space $\meH\subs N$ implementing $\rho$.
Let $\{w_j\}_{j=1}^d$ be an orthonormal basis of $\meH$.
Then
$
\ph_\rho(x)=d^{-1}\sum_{j=1}^d w_j^* x w_j$ for $x\in N$.
Hence (\ref{eq: Irv}) is equivalent with
\[
\lim_{\nu\to\infty}
\left\| r^{-1}(w_j^*v_i^\nu)\cdot \vph 
-d^{-1}\vph\cdot (w_j^*v_i^\nu)
\right\|=0
\]
for all $1\leq i\leq [r]+1$, $1\leq j\leq d$, $\vph\in N_*$.

Let $a\in N$ be a trace class operator.
We apply the above limit equality to $\vph=\ta_a$.
Then the trace norm of
$r^{-1}w_j^*v_i^\nu a-d^{-1} a w_j^*v_i^\nu$ converges to $0$.
Since the trace norm dominates the uniform norm,
we have 
\[
\lim_{\nu\to\infty}
\|r^{-1}w_j^*v_i^\nu a-d^{-1} a w_j^*v_i^\nu\|=0. 
\]
Let $a$ be a finite projection, and we have
\[
\lim_{\nu\to\infty}\|(r^{-1}-d^{-1})a w_j^*v_i^\nu a\|=0,
\ 
\lim_{\nu\to\infty}\|(1-a) w_j^*v_i^\nu a\|=0,
\ 
\lim_{\nu\to\infty}\|aw_j^*v_i^\nu (1-a)\|=0. 
\]
If $r\neq d$,
then the above equalities imply that
$\|w_j^* v_i^\nu a\| \to 0$ and $\|a w_j^* v_i^\nu\|\to0$
as $\nu\to\infty$ for a finite projection $a\in N$.
Thus $w_j^* v_i^\nu\to0$ strongly$*$
as $\nu\to\infty$.
Since
$
v_i^\nu=\sum_{j=1}^d w_j(w_j^* v_i^\nu)$,
$v_i^\nu\to0$ strongly* as $\nu\to\infty$.
This is a contradiction with
$
1=\sum_{i=1}^{[r]+1} v_i^\nu(v_i^\nu)^*$.
Hence $r=d$, and $\Int_d(N)\cap \oInt_r(N)\neq\emptyset$
yields $r=d$.

Since any endomorphism on $N$ is inner,
the statement of this lemma holds.
\end{proof}

\begin{lem}\label{lem: finite-int}
If $N$ is a type II$_1$ factor with the tracial state $\ta$, 
then $\oInt_r(N)=\emptyset$ for all $r\neq 1$. 
Moreover if $N$ is a hyperfinite factor, 
then $\oInt_1(N)=\{\rho\in\End(N)_0\mid \ta\circ\ph_\rho=\ta\}$. 
\end{lem}
\begin{proof}
If $\oInt_r(N)\neq\emptyset$, we can take $\rho \in \oInt_r(N)$. 
By definition, there exist sequences of partial isometries 
$\{v_i^\nu\}_{i=1}^{[r]+1}\subset N$, $\nu\in\N$, with the conditions in
Definition \ref{defn: app-inner-general}.
At least $v_1^\nu$ is an isometry (or coisometry) 
when $r\geq1$ (resp. $0<r<1$), 
but we note that any isometry (or coisometry) is a unitary
because $N$ is finite.
Hence $v_1^\nu$ is unitary.
Then we have
\[
\lim_{n\to\infty}
\left\|r^{-1}\vph\circ \Ad(v_1^\nu)^*-\vph\circ \ph_\rho\right\|=0
\quad \mbox{for all}\ \vph\in N_*. 
\]
In particular, 
$r^{-1}\vph(1)-\vph(1)
=(r^{-1}\vph\circ \Ad(v_1^\nu)^*-\vph\circ \ph_\rho)(1)$ 
is equal to $0$. 
Hence $r$ must be equal to $1$.
The latter assertion follows from \cite[Lemma 3.9]{MT1}. 
\end{proof}

\begin{lem}\label{lem: II-infty factor}
Let $N$ be the hyperfinite type II${}_\infty$ factor of
with the trace $\ta$.
Let $\rho\in\End(N)_0$
and $\mo(\rho)$ be the module of $\rho$,
i.e., $\ta\circ\ph_\rho=d(\rho)^{-1}\mo(\rho)^{-1}\ta$.
Let $\la>0$.
Then $\rho\in\oInt_\la(N)$ if and only if $\la=d(\rho)\mo(\rho)$.
\end{lem}
\begin{proof}
We will show $\rho\in\oInt_\la(N)$ for $\la=d(\rho)\mo(\rho)$.
Set $\mu_0=(1+\la)^{-1}$, $\mu_1=\la(1+\la)^{-1}$.
Consider the locally trivial subfactor
$N^{(\id,\rho)}\subs M^{(\id,\rho)}$
with the conditional expectation $E:=E^{(\mu_0,\mu_1)}$.
Then $E$ preserves the trace $\ta\oti \Tr$ on $M^{(\id,\rho)}$,
and
the locally trivial inclusion
$N^{(\id,\rho)}\stackrel{E^\mu}{\subs} M^{(\id,\rho)}$
is approximately inner by \cite[Theorem 2.9 (i)]{Po1}.
Then Theorem \ref{thm: rho-app} implies that $\rho\in \oInt_\la(N)$.

Conversely we assume $\rho\in\oInt_\la(N)$ for some $\la>0$.
We set $\mu_0:=(1+\la)^{-1}$ and $\mu_1:=\la(1+\la)^{-1}$.
Then the expectation $E_\rho$ preserves $\ta$, 
and $\rho(N)\stackrel{E_\rho}{\subs} N$ is approximately inner
\cite[Theorem 2.9 (i)]{Po1}.
Hence
the locally trivial subfactor
$N^{(\id,\rho)}\stackrel{E^\mu}{\subs} M^{(\id,\rho)}$
is approximately inner by Theorem \ref{thm: rho-app}.
Again by \cite[Theorem 2.9 (i)]{Po1},
$E^\mu$ preserves the trace
$\ta\oti \Tr$, that is, $\la=d(\rho)\mo(\rho)$.
\end{proof}

We will use the following generalization of the previous lemma 
to non-factorial case. 
The definition of approximate innerness is naturally extended to 
this case as Definition \ref{defn: app-inner-general}, 
but we have to fix a left inverse of an endomorphism. 

\begin{lem}\label{lem: typeII}
Let $N$ be a hyperfinite type II$_\infty$ von Neumann algebra 
with a faithful normal semifinite trace $\ta$. 
Let $\rho\in \End(N)$ with a left inverse $\ph_\rho$. 
If $\rho|_{Z(N)}=\id$ and $\ta\circ \ph_\rho=\la^{-1}\ta$ 
for some $\la>0$, 
then $\rho$ is approximately inner of rank $\la$ 
with respect to $\ph_\rho$. 
\end{lem}
\begin{proof}
By assumption on the hyperfiniteness,
we can regard $N=Z(N)\oti R_{0,1}$.

First we assume that $\la=1$.
We have $\ta\circ \ph_\rho=\ta$, and $\ta\circ\rho=\ta$.
Since $\rho=\id$ on $Z(N)$, $\rho(p)$ is equivalent to $p$
for any projection $p\in N$.

Take a system of matrix units $\{e_{i,j}\}_{i,j=1}^\infty$ in $R_{0,1}$ 
such that $e_{i,i}$ are finite projections for all $i$.
We take a partial isometry $w\in N$ such that 
$w^*w=1\oti e_{1,1}$ and $ww^*=\rho(1\oti e_{11})$. 
Then we set a unitary
$
v:=\sum_{i=1}^\infty \rho(1\oti e_{i1})w(1\oti e_{1i})$.
It is easy to see that
$v (1\oti e_{ij})=\rho(1\oti e_{ij})v$ for all $i,j$.
Hence $\si:=\Ad v^*\circ\rho$ fixes $Z(N)$
and the type I subfactor $B$ generated by
$\{1\oti e_{i,j}\}_{i,j=1}^\infty$.
Considering the left inverse $\ph_\si:=\ph_\rho\circ \Ad v$ of $\si$, 
we may assume that $\rho$ fixes $B$. 
We note that $\ph_\rho$ also fixes them. 
Indeed if $\rho(x)=x$ for $x\in N$, then $\ph_\rho(x)=\ph_\rho(\rho(x))=x$. 

Now consider the reduced endomorphism $\rho^{1\oti e_{11}}$ 
on the hyperfinite type II$_1$ von Neumann algebra 
$(1\oti e_{11})N(1\oti e_{11})=Z(N)\oti e_{11}R_{0,1}e_{11}$. 
Using the natural isomorphism from $R_{0,1}$ onto $e_{11}R_{0,1}e_{11}\oti B$, 
we see that $\rho$ and $\ph_\rho$ are of the form 
$\rho^{1\oti e_{11}}\oti \id_B$ and $\ph_\rho^{1\oti e_{11}}\oti \id_B$ 
on $Z(N)\oti e_{11}R_{0,1}e_{11}\oti B$, respectively. 
Then the same proof of \cite[Lemma 3.9]{MT1} works
after a slight modification on a treatment of the center.
Hence $\rho^{1\oti e_{11}}$ is approximately inner of rank $1$,
and so is $\rho$.

Second we consider a general case. 
Take $\th\in\Aut(R_{0,1})$ with module $\la$. 
By the previous lemma, 
$\th\in \Aut(R_{0,1})$ is approximately inner of rank $\la$. 
The trace $\ta$ is given by $\ta_0\oti \ta_1$ where 
$\ta_0$ and $\ta_1$ are the traces on $Z(N)$ and $R_{0,1}$,
respectively. 
Then the automorphism $\id\oti \th$ satisfies
$\ta \circ (\id\oti\th)=\la\ta$.
We simply write $\id\oti \th$ as $\th$ below.
It is easy to see that $\th\in\Aut(N)$ is also approximately inner
of rank $\la$ with respect to the left inverse $\th^{-1}$.
Obviously $\theta$ is trivial on $Z(N)$.
Set $\rho_0:=\th^{-1}\rho$ and $\ph_{\rho_0}=\ph_\rho\circ\th$. 
Then we have $\rho_0|_{Z(N)}=\id$ and $\ta\circ\ph_{\rho_0}=\ta$, 
and the first part of the proof 
implies that $\rho_0$ is approximately inner of rank $1$. 

Take sequences of partial isometries 
$\{v_j^\nu\}_{j=1}^{[\la]+1}$, $\{u^\nu\}_\nu$, $\nu\in\N$, 
in $N$
such that 
\begin{enumerate}
\item 
$(v_j^\nu)^* v_j^\nu=1$ for $1\leq j\leq [\la]$, 
$\displaystyle\sum_{j=1}^{[\la]+1}v_j^\nu(v_j^\nu)^* =1$. 

\item 
$(u^\nu)^* u^\nu=1=u^\nu(u^\nu)^*$. 

\item For all $\vph\in N_*$ and $1\leq j\leq [\la]+1$, 
\[
\lim_{\nu\to\infty}
\left\|
\la^{-1} v_j^\nu \vph
-(\vph\circ\th^{-1})\cdot v_j^\nu
\right\|
=0, \quad
\lim_{\nu\to\infty}
\left\|
u^\nu \vph
-(\vph\circ\ph_{\rho_0})\cdot u^\nu
\right\|
=0. 
\]
\end{enumerate}
Then it is easy to see that 
\[
\lim_{\nu\to\infty}
\left\|
\la^{-1} v_j^\nu u^\nu \vph
-(\vph\circ\ph_\rho)\cdot v_j^\nu u^\nu 
\right\|
=0.
\]
Since $v_j^\nu u^\nu$ is an isometry for $1\leq j\leq [\la]$ 
and
$
\sum_{j=1}^{[\la]+1} v_j^\nu u^\nu(v_j^\nu u^\nu)^*=1$, 
$\rho$ is approximately inner of rank $\la$ by definition.
\end{proof}

\section{Canonical extensions and approximately inner endomorphisms} 

In this section, 
we discuss a generalization of the result proved by Kawahigashi,
Sutherland and Takesaki \cite{KST}, 
that was first announced by Connes without a proof \cite{C3}.
Their result says that
for any hyperfinite factor $M$, 
an automorphism on $M$ is approximately inner if and only if it has
trivial Connes-Takesaki module,
that is,
\[
\oInt(M)=\mathrm{Ker}(\mo). 
\]

\subsection{Canonical extension}

We recall canonical extensions of endomorphisms introduced by Izumi
\cite{Iz1}. 

Let $M$ be a factor and $\tM$ the canonical core extension of $M$
\cite[Definition 2.5]{FT}, 
which is the von Neumann algebra generated 
by $M$ and 
one-parameter unitary groups $\{\la^\vph(t)\}_{t\in\R}$, $\vph\in W(M)$, 
satisfying the relations 
\[
\si_t^\vph(x)=\la^\vph(t)x\la^\vph(t)^*,
\quad
\la^\ps(t)=[D\ps:D\vph]_t \la^\vph(t)
\]
for all $x\in M$, $t\in\R$ and $\vph,\ps\in W(M)$. 

Let us represent $M$ on a Hilbert space $H$. 
The crossed product $M\rti_{\si^\vph}\R$ 
for the modular automorphism group $\si^\vph$ is
the von Neumann algebra
generated by $\pi_{\si^\vph}(M)$ and $\la(\R)$ in $B(H\oti L^2(\R))$ 
such that 
\[
(\pi_{\si^\vph}(x)\xi)(s)=\si_{-s}^\vph(x)\xi(s),
\quad
(\la(t)\xi)(s)=\xi(-t+s)
\]
for all $x\in M$, $\xi\in H\oti L^2(\R)=L^2(\R,H)$ and $s,t\in\R$. 
By \cite[Theorem 2.4]{FT}, we have the isomorphism
$\Pi_\vph\col \tM\ra M\rti_{\si^\vph}\R$ satisfying
\[
\Pi_\vph(x)=\pi_{\si^\vph}(x),
\quad
\Pi_\vph(\la^\vph(t))=\la(t)
\quad
\mbox{for all}\ x\in M, t\in\R.
\]

Let $\th$ be the $\R$-action on $\tM$ satisfying
\[
\th_s(x)=x,\quad \th_s(\la^\vph(t))=e^{-ist}\la^\vph(t)
\quad
\mbox{for all}\ x\in M, s,t\in\R.
\]
Then $\Pi_\vph\circ \th_s=\wdh{\si^\vph}_s \circ \Pi_\vph$ 
for all $s\in\R$.
The action $\th$ is also called the dual action. 

Let $\rho$ be an endomorphism on $M$ with finite index. 
Then the \emph{canonical extension} $\trho$ of $\rho$ is
the endomorphism on $\tM$ defined by
\[
\trho(x)=\rho(x),
\quad
\trho(\la^\vph(t))=d(\rho)^{it}[D\vph\circ\ph_\rho:D\vph]_t\la^\vph(t)
\]
for all $x\in M$, $t\in \R$ and $\vph\in W(M)$. 
Note that $\trho$ commutes the dual action $\th$. 

\subsection{Normalized canonical extension}

Let $\,\wdt{}\ \col \End(M)_0\ra \End(\tM)$ be the canonical
extension.
It is known that
the canonical extension is continuous
on $\Aut(M)$ with respect to the $u$-topology.
However in general,
it is not continuous on $\End(M)_0$
because the statistical dimension map
$d\col \End(M)_0\ra [1,\infty)$ is not continuous
with respect to our topology
(recall Definition \ref{defn: topology}).
In \S \ref{subsect: CT-app}, we will discuss a relation
between approximate innerness and Connes-Takesaki modules. 
Then we need the continuity for that purpose. 
Hence we introduce a modified canonical extension map as follows. 

\begin{defn}
Let $M$ be a factor and $\rho\in \End(M)_0$. 
We define the \emph{normalized canonical extension map} 
$\breve{}\ \col\End(M)_0\ra \End(\tM)$ by 
\[
\breve{\rho}(x)=\rho(x),
\quad
\breve{\rho}(\la^\vph(t))=[D\vph\circ\ph_\rho:D\vph]_t\la^\vph(t)
\]
for all $x\in M$, $t\in \R$ and $\vph\in W(M)$. 
\end{defn}

Indeed, we have
\[
\breve{\rho}=\th_{\log(d(\rho))}\circ \trho
=\trho\circ\th_{\log(d(\rho))}, 
\]
which shows the existence of $\breve{\rho}$.

Next we want to discuss a convergence in $\End(\tM)$ 
by using particular left inverses.
Note that $\tM$ may not be a factor.
Even in non-factorial case, minimal expectations can be
also defined as in \cite{FI},
and it will be possible to give a topology.
However, 
in order to avoid using a disintegration of factors and
left inverses, 
we do not take such a way.
For our purpose,
the following notion presented in \cite[Definition 3.1]{MT1} is sufficient.

\begin{defn}
Let $M$ be a von Neumann algebra and 
$\rho^\nu$, $\nu\in\N$, $\rho$ endomorphisms on $M$.
Let $\ph^\nu$ and $\ph$ be left inverses
of
$\rho^\nu$ and $\rho$, respectively.
We say that the sequence of the pairs
$\{(\rho^\nu,\ph^\nu)\}_\nu$
converges to $(\rho,\ph)$
if 
\[
\lim_{\nu\to\infty}\|\vph\circ \ph^\nu-\vph\circ\ph\|=0 
\quad\mbox{for all}\ \vph\in M_*. 
\]
\end{defn}

Note that the above convergence implies pointwise strong* convergence,
that is,
if $(\rho^\nu,\ph^\nu)$ converges to $(\rho,\ph)$, 
then $\rho^\nu(x)\to\rho(x)$ strongly* as $\nu\to\infty$ for any $x\in M$
\cite[Lemma 3.8]{MT1}.
When $M$ is a factor,
$\rho^\nu$ converges to $\rho$ in $\End_0(M)$ in the topology
defined in Definition \ref{defn: topology}
if and only if $(\rho^\nu,\ph_{\rho^\nu})$ converges to
$(\rho,\ph_\rho)$.
We study the relationship between the convergence of
endomorphisms and that of implementing isometries \cite{GL}.

Let $M$ be a factor as before.
We represent $M$ on the standard Hilbert space $L^2(M)$.
The positive cone is denoted by $L^2(M)_+$.
In what follows,
we use the following useful equalities for $\rho\in\End(M)_0$:
\begin{equation}\label{eq: si-ph}
\si_t^{\ps\circ\ph_\rho}\circ\rho=\rho\circ\si_t^\ps,
\quad
[D\ps\circ\ph_\rho:D\chi\circ\ph_\rho]_t=\rho([D\ps:D\chi]_t)
\end{equation}
for all $\ps,\chi\in W(M)$ and $t\in\R$ \cite[p.5--7]{Iz1}.
Since $\ps\circ \ph_\rho\circ E_\rho=\ps\circ\ph_\rho$,
$E_\rho$ and $\si_t^{\ps\circ\ph_\rho}$ commute \cite[p.317]{Ta1}.
Hence we also have
\begin{equation}\label{eq: si-Erho}
\si_t^{\ps\circ\ph_\rho}\circ E_\rho
=E_\rho\circ \si_t^{\ps\circ\ph_\rho}, 
\quad
\si_t^{\ps}\circ \ph_\rho
=\ph_\rho\circ \si_t^{\ps\circ\ph_\rho}.
\end{equation}
where the latter equality follows from the former one and (\ref{eq: si-ph}).

Now let us fix a faithful state $\ps\in M_*$.
We take a unit vector $\xi_\ps\in L^2(M)_+$ such that
$\ps(x)=(x\xi_\ps,\xi_\ps)$ for $x\in M$.
For each $\rho\in \End(M)_0$, we take a unit vector
$\xi_{\ps\circ\ph_\rho}\in L^2(M)_+$ such that
$\ps(\ph_\rho(x))
=(x\xi_{\ps\circ\ph_\rho},\xi_{\ps\circ\ph_\rho})$ for $x\in M$.
Following \cite[Appendix A]{GL},
we define the standard implementation $V_\rho$ for $\rho$ by
\[
V_\rho (x\xi_\ps)=\rho(x)\xi_{\ps\circ\ph_\rho}
\quad\mbox{for all}\ x\in M. 
\]
Then the isometry $V_\rho$ satisfies the following
\cite[Proposition A.2]{GL}:
\[
V_\rho x=\rho(x) V_\rho,\ 
\ph_\rho(x)=V_\rho^* x V_\rho
\quad\mbox{for all}\ x\in M. 
\]
Let $\De_\ps$ and $\De_{\ps\circ \ph_\rho}$
be the modular operators
of $\ps$ and $\ps\circ \ph_\rho$, respectively.
Then,
\begin{equation}\label{eq: U-De}
V_\rho \De_\ps^{it}=\De_{\ps\circ \ph_\rho}^{it} V_\rho
\quad\mbox{for all}\ t\in \R. 
\end{equation}
Indeed, using a formula
$\si_t^{\ps\circ \ph_\rho}\circ \rho=\rho\circ \si_t^{\ps}$,
we have
\begin{align*}
V_\rho \De_\ps^{it} (x\xi_{\ps})
=&\,
V_\rho (\si_t^\ps(x)\xi_{\ps})
=
\rho(\si_t^\ps(x))\xi_{\ps\circ\ph_\rho}
\\
=&\,
\si_t^{\ps\circ \ph_\rho}(\rho(x))\xi_{\ps\circ\ph_\rho}
=
\De_{\ps\circ\ph_\rho}^{it} \rho(x)\xi_{\ps\circ\ph_\rho}
\\
=&\,
\De_{\ps\circ\ph_\rho}^{it} V_\rho (x\xi_{\ps}). 
\end{align*}

\begin{lem}\label{lem: rho-nu-rho}
Let $(\rho^\nu)_{\nu\in\N}$ and $\rho$ be
endomorphisms on a factor $M$ with finite index.
Then the following statements are equivalent.
\begin{enumerate}
\item $(\rho^\nu)_\nu$ converges to $\rho$.

\item 
$(\rho^\nu(x))_\nu$ converges to $\rho(x)$ strongly*
for all $x\in M$
and
$\displaystyle\lim_{\nu\to\infty}\xi_{\ps\circ \ph_{\rho^\nu}}
=\xi_{\ps\circ \ph_\rho}$.

\item 
$(V_{\rho^\nu})_\nu$ converges to $V_\rho$ strongly.
\end{enumerate}
\end{lem}
\begin{proof}
(1)$\Rightarrow$ (2). 
The strong* convergence of $\rho^\nu(x)$ follows from 
\cite[Lemma 3.8]{MT1}. 
Since $M$ acts on $L^2(M)$ standardly, 
the convergence $\ps\circ \ph_{\rho^\nu}\to \ps\circ\ph_\rho$ 
implies the convergence 
$\xi_{\ps\circ \ph_{\rho^\nu}}\to \xi_{\ps\circ \ph_\rho}$
\cite[Lemma 2.10]{Ha}. 

\noindent
(2)$\Rightarrow$ (3). 
By (2), we see that for all $x\in M$, 
\[
\lim_{\nu\to\infty}V_{\rho^\nu}(x\xi_\ps)
=\lim_{\nu\to\infty}\rho^\nu(x)\xi_{\ps\circ\ph_{\rho^\nu}}
=\rho(x)\xi_{\ps\circ\ph_\rho}
=V_\rho (x\xi_\ps). 
\]
The norm-boundedness of $V_{\rho^\nu}$ implies the strong convergence. 

\noindent
(3)$\Rightarrow$ (1). 
For vectors $\xi,\eta\in L^2(M)$, 
we denote by $\om_{\xi,\eta}\in M_*$ the functional 
$\om_{\xi,\eta}(x)=(x\xi,\eta)$ for $x\in M$. 
Since $V_{\rho^\nu}$ implements $\ph_{\rho^\nu}$, we have 
$\om_{\xi,\eta}\circ \ph_{\rho^\nu}
=\om_{V_{\rho^\nu}\xi,V_{\rho^\nu}\eta}$. 
By elementary calculation, we have 
\[
\|\om_{\xi,\eta}\circ \ph_{\rho^\nu}
-\om_{\xi,\eta}\circ \ph_{\rho}
\|
\leq 
\|\eta\|\|V_{\rho^\nu}\xi-V_\rho \xi\|
+\|\xi\|\|V_{\rho^\nu}\eta-V_\rho \eta\|. 
\]
Hence we have the norm convergence 
$\om_{\xi,\eta}\circ \ph_{\rho^\nu}
\to\om_{\xi,\eta}\circ \ph_{\rho}$ as $\nu\to\infty$. 
Since any normal functional on $M$ is
of the form $\om_{\xi,\eta}$, $\xi,\eta\in L^2(M)$, 
we have done. 
\end{proof}

\begin{lem}
Let $M$ be an infinite factor, 
$\rho_1,\rho_2\in \End(M)_0$ 
and $v_1,v_2\in M$ isometries with $v_1v_1^*+v_2v_2^*=1$. 
We define $\rho\in \End(M)$
by $\rho(x):=v_1\rho_1(x)v_1^*+v_2\rho_2(x)v_2^*$. 
Then for any weight $\vph$ on $M$, we have 
\[
d(\rho)^{it}[D\vph\circ \ph_\rho:D\vph]_t
=\sum_{k=1}^2 d(\rho_k)^{it}v_k[D\vph\circ\ph_{\rho_k}:D\vph]_t 
\si_t^\vph(v_k^*). 
\]
\end{lem}
\begin{proof}
It is shown by using
$
d(\rho)\ph_\rho(x)
=\sum_{k=1}^2 d(\rho_k)\ph_{\rho_k}(v_k^* xv_k)$
and the relative modular condition \cite[Theorem VIII.3.3]{Ta2}.
\end{proof}

Now we construct a left inverse of a canonical extension.

\begin{lem}
Let $\rho\in\End(M)_0$ and $\wdt{\rho}$  be the canonical extension of $\rho$.
Then there exists a left inverse $\ph_{\wdt{\rho}}$ on $\tM$ such that
\[
\ph_{\wdt{\rho}}(x\la^\vph(t))
=d(\rho)^{-it}\ph_\rho(x [D\vph:D\vph\circ \ph_\rho]_t)\la^\vph(t)
\]
for all $x\in M$, $t\in\R$ and $\vph\in W(M)$. 
\end{lem}
\begin{proof}
When $M$ is finite, we consider $P:=B(\el^2)\oti M$
and $\si:=\id\oti \rho$.
Then $\wdt{P}=B(\el^2)\oti\tM$ and $\tsi=\id\oti \wdt{\rho}$.
If the statement holds for infinite case,
there exists a left inverse $\ph_{\tsi}$ on $\wdt{P}$
with the above property.
Since $\tsi$ is trivial on $B(\el^2)$, so is $\ph_{\tsi}$.
Then we can define the map $\ph_{\trho}$ on $\tM$ by
$\ph_{\tsi}=\id\oti \ph_{\trho}$, which has the desired property.
Hence we may and do assume that $M$ is infinite.

Take an isometry $v\in (\id,\brho\rho)$.
We set $\ph_{\wdt{\rho}}(x):=v^* \wdt{\brho}(x)v$ for $x\in\tM$.
By \cite[Proposition 2.5 (1)]{Iz1}, $v\in (\id,\wdt{\brho}\wdt{\rho})$.
Hence $\ph_{\wdt{\rho}}$ is a left inverse of $\wdt{\rho}$.
By the previous lemma,
we have 
\[
d(\brho\rho)^{it}
v^* [D\vph\circ\ph_{\brho\rho}:D\vph]_t=\si_t^\vph(v^*). 
\]
Using $\ph_{\brho\rho}=\ph_\rho\ph_\brho$ and (\ref{eq: si-ph}), 
we have 
\begin{align*}
\ph_{\wdt{\rho}}(x\la^\vph(t))
=&\,
v^*\wdt{\brho}(x \la^\vph(t))v
=
v^* \brho(x) d(\rho)^{it} [D\vph\circ\ph_\brho:D\vph]_t \la^\vph(t)v
\\
=&\,
d(\rho)^{it}
v^* \brho(x) [D\vph\circ\ph_\brho:D\vph]_t \si_t^\vph(v)\la^\vph(t)
\\
=&\,
d(\rho)^{it}
v^* \brho(x) [D\vph\circ\ph_\brho:D\vph]_t 
\cdot
d(\brho\rho)^{-it}
[D\vph\circ\ph_{\brho\rho}:D\vph]_t^*
v\la^\vph(t)
\\
=&\,
d(\rho)^{-it}
v^*\brho(x)[D\vph\circ\ph_\brho:D\vph\circ\ph_{\brho\rho}]_t 
v \la^\vph(t)
\\
=&\,
d(\rho)^{-it}
v^* \brho(x)\brho([D\vph:D\vph\circ\ph_\rho]_t)v\la^\vph(t)
\\
=&\,
d(\rho)^{-it}
\ph_\rho(x[D\vph:D\vph\circ\ph_\rho]_t)\la^\vph(t).
\end{align*}
\end{proof}
Hence the map $\ph_\brrho:=\ph_{\wdt{\rho}}\circ \th_{-\log(d(\rho))}$ 
is a left inverse of $\brrho$ such that
\[
\ph_{\brrho}(x\la^\vph(t))
=\ph_\rho(x [D\vph:D\vph\circ \ph_\rho]_t)\la^\vph(t). 
\]

Let $\vph\in M_*$ be a faithful state.
We identify $\tM$ with $M\rti_{\si^{\vph}}\R$ via $\Pi_\vph$.
For $\rho\in\End(M)_0$,
we define the operator $U_\brrho$
on $L^2(M)\oti L^2(\R)=L^2(\R,L^2(M))$
by
\[
(U_\brrho\xi)(s)=[D\vph\circ\ph_\rho:D\vph]_{-s}^* V_\rho\xi(s)
\quad\mbox{for all}\ \xi\in L^2(\R,L^2(M)), s\in\R,
\]
where $V_\rho$ is an isometry defined as before 
by the fixed faithful normal state $\ps$. 

\begin{lem}
For $\rho\in\End(M)_0$, 
$U_\brrho$ has the following properties: 

\begin{enumerate}
\item $U_\brrho$ is an isometry. 

\item $U_\brrho x=\brrho(x)U_\brrho$ for all $x\in M\rti_{\si^\vph}\R$. 

\item 
$\ph_\brrho(x)=U_\brrho^* x U_\brrho$ 
for all $x\in M\rti_{\si^\vph}\R$.

\end{enumerate}
\end{lem}
\begin{proof}
(1) It is trivial. 

\noindent(2) 
Let $x\in M$ and $t\in\R$. Then we have for $s\in\R$, 
\begin{align*}
\big{(}
U_\brrho (\pi_{\si^\vph}(x)\la(t))\xi
\big{)}(s)
=&\,
[D\vph\circ\ph_\rho:D\vph]_{-s}^* 
V_\rho\big{(}(\pi_{\si^\vph}(x)\la(t)\xi\big{)}(s)
\\
=&\,
[D\vph\circ\ph_\rho:D\vph]_{-s}^* 
V_\rho\si_{-s}^\vph(x)\xi(-t+s)
\\
=&\,
[D\vph\circ\ph_\rho:D\vph]_{-s}^* 
\rho(\si_{-s}^\vph(x))V_\rho\xi(-t+s)
\\
=&\,
[D\vph\circ\ph_\rho:D\vph]_{-s}^* 
\si_{-s}^{\vph\circ\ph_\rho}(\rho(x))V_\rho\xi(-t+s)
\\
=&\,
\si_{-s}^{\vph}(\rho(x))
[D\vph\circ\ph_\rho:D\vph]_{-s}^* V_\rho\xi(-t+s)
\\
=&\,
\si_{-s}^{\vph}(\rho(x))
[D\vph\circ\ph_\rho:D\vph]_{-s}^* [D\vph\circ\ph_\rho:D\vph]_{t-s}
\\
&\quad\cdot 
(U_\brrho\xi)(-t+s)
\\
=&\,
\si_{-s}^{\vph}(\rho(x))
\si_{-s}^\vph([D\vph\circ\ph_\rho:D\vph]_{t} )\cdot
(U_\brrho\xi)(-t+s)
\\
=&\,
\big{(}\pi_{\si^\vph}(\rho(x) [D\vph\circ\ph_\rho:D\vph]_{t})
\la(t)U_\brrho\xi\big{)}(s)
\\
=&\,
\big{(}\brrho(\pi_{\si^\vph}(\rho(x) )\la(t))U_\brrho
\xi\big{)}(s). 
\end{align*}
Hence the equality of (2) holds. 

\noindent(3) 
By (2) and (3), 
we have for $x\in M$, $t\in \R$ and $\xi,\eta\in L^2(M)\oti L^2(\R)$, 
\begin{align*}
&(U_\brrho^* \pi_{\si^\vph}(x)\la(t)U_\brrho\xi,\eta)
\\
=&\,
(\pi_{\si^\vph}(x)\la(t)U_\brrho\xi,U_\brrho\eta)
\\
=&\,
\int_{-\infty}^\infty
\big{(}
\big{(}\pi_{\si^\vph}(x)\la(t)U_\brrho\xi\big{)}(s),(U_\brrho\eta)(s)
\big{)}ds
\\
=&\,
\int_{-\infty}^\infty
\big{(}
\si_{-s}^\vph(x)
\big{(}U_\brrho\xi\big{)}(-t+s),(U_\brrho\eta)(s)
\big{)}ds
\\
=&\,
\int_{-\infty}^\infty
\big{(}
\si_{-s}^\vph(x)
[D\vph\circ\ph_\rho:D\vph]_{t-s}^* V_\rho\xi(-t+s),
[D\vph\circ\ph_\rho:D\vph]_{-s}^* V_\rho \eta(s)
\big{)}ds
\\
=&\,
\int_{-\infty}^\infty
\big{(}
[D\vph\circ\ph_\rho:D\vph]_{-s}
\si_{-s}^\vph(x)
[D\vph\circ\ph_\rho:D\vph]_{t-s}^* V_\rho\xi(-t+s),
 V_\rho \eta(s)
\big{)}ds
\\
=&\,
\int_{-\infty}^\infty
\big{(}
\si_{-s}^{\vph\circ\ph_\rho}(x)
[D\vph\circ\ph_\rho:D\vph]_{-s}
[D\vph\circ\ph_\rho:D\vph]_{t-s}^* V_\rho\xi(-t+s),
 V_\rho \eta(s)
\big{)}ds
\\
=&\,
\int_{-\infty}^\infty
\big{(}
\si_{-s}^{\vph\circ\ph_\rho}(x)
\si_{-s}^{\vph\circ\ph_\rho}([D\vph\circ\ph_\rho:D\vph]_{t}^*)
 V_\rho\xi(-t+s), V_\rho \eta(s)
\big{)}ds
\\
=&\,
\int_{-\infty}^\infty
\big{(}
V_\rho^*
\si_{-s}^{\vph\circ\ph_\rho}
(x[D\vph:D\vph\circ\ph_\rho]_{t}) V_\rho\xi(-t+s), \eta(s)
\big{)}ds
\\
=&\,
\int_{-\infty}^\infty
\big{(}
\ph_\rho(
\si_{-s}^{\vph\circ\ph_\rho}(x[D\vph:D\vph\circ\ph_\rho]_{t}))
\xi(-t+s), \eta(s)
\big{)}ds
\\
=&\,
\int_{-\infty}^\infty
\big{(}
\si_{-s}^{\vph}
(\ph_\rho(x[D\vph:D\vph\circ\ph_\rho]_{t}))
\xi(-t+s), \eta(s)
\big{)}ds
\qquad (\mbox{by}\ (\ref{eq: si-Erho}))
\\
=&\,
\big{(}
\pi_{\si^\vph}(\ph_\rho(x[D\vph:D\vph\circ\ph_\rho]_{t}))
\la(t)\xi,\eta
\big{)}
\\
=&\,
\big{(}
\ph_{\brrho}(\pi_{\si^\vph}(x)\la(t))\xi,\eta\big{)}. 
\end{align*}
\end{proof}

\begin{lem}
Assume that $(\rho^\nu)_{\nu\in\N}$
converges to $\rho$ in $\End(M)_0$.
Then $U_{\breve{\rho^\nu}}$ converges to $U_\brrho$ strongly.
\end{lem}
\begin{proof}
Since $U_{\breve{\rho^\nu}}$ and $U_\brrho$ are isometries, 
it suffices to show that 
$(U_{\breve{\rho^\nu}}\xi,\eta)$ converges to 
$(U_\brrho\xi,\eta)$ for all $\xi,\eta\in L^2(M)\oti L^2(\R)$. 
Moreover we may and do assume that $\xi,\eta$ have 
compact supports on $\R$. 
Then we have 
\begin{align*}
(U_{\breve{\rho^\nu}}\xi,\eta)
=&\,
\int_{-\infty}^\infty
((U_{\breve{\rho^\nu}}\xi)(s),\eta(s))ds
\\
=&\,
\int_{-\infty}^\infty
([D\vph\circ\ph_{\rho^\nu}:D\vph]_{-s}^*V_\rho\xi(s),\eta(s))ds,
\end{align*}
which converges to 
\[
\int_{-\infty}^\infty
([D\vph\circ\ph_{\rho}:D\vph]_{-s}^* V_\rho\xi(s),\eta(s))ds
\]
because for each $r>0$,
the cocycle $[D\vph\circ\ph_{\rho^\nu}:D\vph]_t$
uniformly converges to $[D\vph\circ\ph_{\rho}:D\vph]_t$
strongly as $\nu\to\infty$ for all $t\in[-r,r]$
\cite[Theorem IX.1.19]{Ta2}. 
\end{proof}

\begin{thm}\label{thm: cont-nce}
Let $M$ be a factor. 
Then the normalized canonical extension is continuous,
that is, 
if $(\rho^\nu)_{\nu\in\N}$ converges to $\rho$ 
in $\End(M)_0$, then the pairs 
$\{(\breve{\rho^\nu},\ph_{\breve{\rho^\nu}})\}_{\nu\in\N}$ 
converge to $(\brrho,\ph_\brrho)$. 
In particular, 
$\breve{\rho^\nu}(x)$ converges to $\brrho(x)$ 
strongly* as $\nu\to\infty$ for all $x\in\tM$. 
\end{thm}
\begin{proof}
By the previous lemma, $U_{\breve{\rho^\nu}}$ converges to $U_\brrho$.
Then the same proof as
(3)$\Rightarrow$(1) and (1)$\Rightarrow$(2)
of Lemma \ref{lem: rho-nu-rho} works. 
\end{proof}

\subsection{Dominant weights and canonical extensions}

In this subsection, we treat an infinite factor $M$.
By the Takesaki duality, $M$ is isomorphic to $\tM\rti_\th\R$
and $\tM$ is regarded as the centralizer of a dominant weight on $M$.
We will study the canonical extension and the restriction of
an endomorphism on the centralizer of a dominant weight.

Let $\vph$ be a dominant weight on $M$.
Then the covariant system $\{M, \si^\vph\}$ is dual,
that is, there exists a one-parameter unitary group
$\{v(t)\}_{t\in\R}$ in $M$ such that
$\si_s^\vph(v(t))=e^{-ist}v(t)$.
Using the action $\th_t:=\Ad v(t)$ on $M_\vph$,
we have the natural isomorphism of the covariant systems
$\{M,\si^\vph\}\cong\{M_\vph\rti_\th \R,\hth\}$.
The dual action on $\tM$ is denoted by $\wdh{\si^\vph}$ for a while.

\begin{lem}
Let $\rho\in\End(M)_0$.
Then
there exists $u\in U(M)$ such that
\begin{enumerate}
\item 
$(\vph,\Ad u\circ\rho)$ is an invariant pair in the sense of
\cite[Definition 2.2]{Iz1},

\item 
$u\rho(v(t))u^*=v(t)$.
\end{enumerate}
\end{lem}
\begin{proof}
By \cite[Lemma 2.12 (ii)]{ILP}, we can
take $u\in U(M)$ satisfying (1).
Replace $\rho$ with $\Ad u\circ \rho$.
Then $\si_t^\vph \circ \rho=\rho\circ\si_t^\vph$
by (\ref{eq: si-ph}),
and the unitary $w(t):=\rho(v(t))v(t)^*$ is contained in $M_\vph$.
Then $\{w(t)\}_t$ is a $\th$-cocycle.
By using the stability of $\{M_\vph,\th\}$
\cite[Theorem III.5.1 (ii)]{CT}, 
there exists $\nu\in U(M_\vph)$
such that $w(t)=\nu\th_t(\nu^*)$.
Then $(\vph,\Ad\nu^*\circ\rho)$ is
an invariant pair, and $\nu^*\rho(v(t)))\nu=v(t)$.
\end{proof}

Replacing $\rho$ with $\Ad u\circ\rho$,
we assume that $\rho$ satisfies the conditions in the above lemma.
Now we discuss how the canonical extension
$\trho\in \End(\tM)$ can be transformed to an endomorphism on $M_\vph$.
We let $M$ act on a Hilbert space $H$.
By the Takesaki duality, we have the isomorphism
$\tM\ra M_\vph\oti B(L^2(\R))$ satisfying
$x v(t) \la^\vph(t)\mapsto \pi_\th(x)(1\oti \mu_\th(t))(1\oti \nu(t))$
for $x\in M_\vph$ and $t\in\R$,
where $\pi_\th(x)\in M_\vph\oti L^\infty(\R)$,
$\mu_\th(t)\in \C\oti L(\R)$ (the group von Neumann algebra of $\R$)
and $\nu(t)\in \C\oti L^\infty(\R)$ are defined by
\[
(\pi_\th(x)\xi)(s)=\th_{-s}(x)\xi(s),
\quad
(\mu_\th(t)\xi)(s)=\xi(-t+s),
\quad
(\nu(t)\xi)(s)=e^{-its}\xi(s)
\]
for all $\xi\in L^2(\R,H)$ and $s,t\in\R$.

Since $\trho(x v(t) \la^\vph(t))=\rho(x)v(t)\la^\vph(t)$ 
and $(\rho\oti\id)\circ\pi_\th=\pi_\th\circ\rho$, 
$\trho$ is transformed to $\rho\oti\id\in \End(M_\vph\oti B(L^2(\R)))$ 
through the isomorphism. 
The dual action $\wdh{\si^\vph}$ on $\tM$ is given by 
$\th_t\oti \Ad \mu_\th(t)$ on $M_\vph\oti B(L^2(\R))$. 
Hence we have the following isomorphism 
between the covariant systems with an endomorphism $\rho$: 
\[
\{\tM,\wdh{\si^\vph},\trho\}
\cong\{M_\vph\oti B(L^2(\R)),\th\oti\Ad \mu_\th,\rho\oti\id\},
\]
which means there exists an isomorphism 
$\Th\col \tM\ra M_\vph\oti B(L^2(\R))$ such that 
$\Th\circ \wdh{\si^\vph}_s=(\th_s\oti\Ad \mu_\th(s))\circ \Th$ 
and $\Th\circ \trho=(\rho\oti\id)\circ \Th$ for all $s\in\R$. 

\begin{lem}\label{lem: duality}
Let $M, \vph, \{v(s)\}_{s\in\R}$ and $\th$ be as before.
Then for any $\rho\in\End(M)_0$,
there exists $u\in M$ and an isomorphism
$\Ps_\rho\col \tM\ra M_\vph$ such that
\begin{enumerate}
\item $(\vph,\Ad u\circ \rho)$ is an invariant pair,

\item $u\rho(v(s))u^*=v(s)$,

\item $\Ps_\rho$ is an isomorphism between the following covariant systems: 
\[
\Ps_\rho\col \{\tM,\wdh{\si^\vph},\trho\}
\ra\{M_\vph,\th,\Ad u\circ\rho|_{M_\vph}\}. 
\]
\end{enumerate}
\end{lem}
\begin{proof}
We may assume that $(\vph,\rho)$ is an invariant pair and
$\rho(v(s))=v(s)$ for all $s\in\R$ as before.
Since $\th$ is a dual action \cite[Theorem III.5.1 (ii)]{CT}, 
$(M_\vph)^\th$ is isomorphic to $M$. 
Hence we can take an infinite 
dimensional Hilbert space $\meH\subs (M_\vph)^\th$ with support $1$.
Let $\{\xi_i\}_{i=1}^\infty$ be an orthonormal basis of $\meH$.
Let 
$t_\meH\col \rho_\meH(M_\vph)\oti B(\meH)\ra M_\vph$
be the isomorphism such that
$t_\meH(\rho_\meH(x)\oti \xi_i\xi_j^*)
=\rho_\meH(x)\xi_i\xi_j^*=\xi_i x\xi_j^*$ 
for all $x\in M_\vph$ and $i,j\in\N$.
We define the unitary $u=\sum_{i=1}^\infty \rho(\xi_i)\xi_i^*$.
Then $u\in (M_\vph)^\th$ and $u\meH=\rho(\meH)$.
We also define the isomorphism $\Ps\col B(L^2(\R))\ra B(H)$
such that $\Ps(e_{ij})=\xi_i\xi_j^*$, where $\{e_{ij}\}_{ij=1}^\infty$
is a system of matrix units of $B(L^2(\R))$.
 
Now we introduce the isomorphism $\Ph\col M_\vph\oti B(L^2(\R))\ra M_\vph$
defined by
\[
\Ph\col M_\vph\oti B(L^2(\R))\stackrel{\rho_\meH\oti\Ps}{\lra}
\rho_\meH(M_\vph)\oti B(\meH)\stackrel{t_\meH}{\lra} M_\vph
\stackrel{\Ad u}{\lra}M_\vph. 
\]
Then we have 
$\Ph(x\oti e_{ij})=u\rho_\meH(x)\xi_i\xi_j^* u^*=u\xi_i x\xi_j^*u^*$ 
for all $x\in M_\vph$ and $i,j\in\N$. 
We will check that 
\[
\Ph\circ (\th_s\oti \id)\circ \Ph^{-1}=\th_s,
\quad
\Ph \circ (\rho\oti\id)\circ \Ph^{-1}
=\Ad \rho(u^*)\circ \rho=\rho\circ\Ad u^*. 
\]
Indeed, for $x\in M_\vph$ and $i,j\in\N$, we have 
\begin{align*}
\Ph((\th_s\oti\id)(x\oti e_{ij}))
=&\,
\Ph(\th_s(x)\oti e_{ij})
=
u\xi_i \th_s(x)\xi_j^*u^*
=
\th_s(u\xi_i x\xi_j^*u^*)
\\
=&\,
\th_s(\Ph(x\oti e_{ij})), 
\end{align*}
and 
\begin{align*}
\Ph((\rho\oti\id)(x\oti e_{ij}))
=&\,
\Ph(\rho(x)\oti e_{ij})
=
u\xi_i \rho(x)\xi_j^*u^*
=
\rho(\xi_i x \xi_j^*)
\\
=&\,
\rho(u^*)\rho(u \xi_i x \xi_j^* u^*)\rho(u)
=
\rho(u^*)\rho(\Ph(x\oti e_{ij}))\rho(u). 
\end{align*}

Set $\mu(s)=\Ph(1\oti \mu_\th(s))$. 
Since $1\oti \mu_\th(s)$ is fixed by $\th\oti\id$ and $\rho\oti\id$, 
so is $\mu(s)$ by $\th$ and $\Ad \rho(u^*)\circ \rho$. 
Hence we have the following isomorphism between the covariant systems 
with endomorphisms:
\[
\Ph\circ\Th\col \{\tM,\wdh{\si^\vph},\trho\}
\ra\{M_\vph,\th\circ\Ad \mu,\Ad \rho(u^*)\circ \rho|_{M_\vph}\}. 
\]

The one-parameter unitary group $\{\mu(s)\}_{s\in\R}$
is contained in $M_\vph^\th$, and this is a $\th$-cocycle.
By stability of $\th$ \cite[Theorem III.5.1 (ii)]{CT},
there exists $w\in U(M_\vph)$ such that
$\mu(s)=w\th_s(w^*)$ for all $s\in\R$.
Then the equality $\Ad \rho(u^*)\circ \rho(\mu(s))=\mu(s)$
implies that $w^*\rho(u^*w)\in M_\vph^\th$. 
Indeed using $\rho(u)\in M_\vph^\th$ and $\rho\circ\th_s=\th_s\circ\rho$, 
we have 
\begin{align*}
(w^*\rho(u^*w))^*\th_s(w^*\rho(u^*w))
=&\,
\rho(w^*u)w\th_s(w^*)\th_s(\rho(u^*w))
\\
=&\,
\rho(w^*)\rho(u)\mu(s)\rho(u^*)\th_s(\rho(w))
\\
=&\,
\rho(w^*)\rho(\mu(s))\th_s(\rho(w))
\\
=&\,
\rho(w^*\mu(s)\th_s(w))=1. 
\end{align*}
Now we have the following isomorphism: 
\[
\{\tM,\hat{\si}^\vph,\trho\}
\stackrel{\Ph\circ \Th}{\lra}
\{M_\vph,\th\circ\Ad \mu,\Ad\rho(u^*)\circ \rho|_{M_\vph}\}
\stackrel{\Ad w^*}{\lra}
\{M_\vph,\th,\Ad(w^*\rho(u^*w))\circ\rho|_{M_\vph}\}. 
\]
This is a desired one.
Indeed,
it is easy to see that $(\vph,\Ad(w^*\rho(u^*w))\circ\rho)$ is
an invariant pair since $w^*\rho(u^*w)\in M_\vph^\th$.
Furthermore we have
\begin{align*}
\Ad(w^*\rho(u^*w))(\rho(v(s)))
=&\,
w^*\rho(u^*w)v(s)\rho(w^*u)w
\\
=&\,
v(s)\th_{-s}(w^*\rho(u^*w))\rho(w^*u)w
\\
=&\,
v(s)w^*\rho(u^*w)\rho(w^*u)w
=v(s). 
\end{align*}
\end{proof}

The following lemma is probably well-known by specialists, 
but we present a proof for readers' convenience. 

\begin{lem}\label{lem: splitting B(H)}
Let $M$ be an infinite factor 
and $\rho\in \End(M)_0$. 
Then there exists an isomorphism $\pi\col M\oti B(\el_2)\ra M$ 
such that $[\rho]=[\pi\circ(\rho\oti\id)\circ \pi^{-1}]$ 
in $\Sect(M)$. 
\end{lem}
\begin{proof}
Take an infinite dimensional Hilbert space $\meH$ in $M$ with support $1$.
Let $\{\xi_i\}_{i=1}^\infty$ be an orthonormal basis of $\meH$.
Define the isomorphism
$t_\meH\col \rho_\meH(M)\oti B(\meH)\ra M$
by $t_\meH(\rho_\meH(x)\oti \xi_i\xi_j^*)
=\rho_\meH(x)\xi_i\xi_j^*=\xi_i x\xi_j^*$
for $x\in M$ and $i,j\in\N$.
Then the isomorphism $\pi\col M\oti B(\meH)\to M$
is defined by $\pi=t_\meH\circ (\rho_\meH\oti\id)$.
Define the unitary
$
u=\sum_{i=1}^\infty\xi_i\rho(\xi_i^*)$.
By direct computation, we see that
$\pi\circ (\rho\oti\id)\circ\pi^{-1}(x)=\Ad u\circ\rho(x)$
for all $x\in M$.
Hence $[\pi\circ (\rho\oti\id)\circ\pi^{-1}]=[\rho]\in \Sect(M)$. 
\end{proof}

\subsection{Connes-Takesaki modules and approximately inner endomorphisms}
\label{subsect: CT-app}

Let $M$ be a factor and $\rho\in\End(M)_0$. 
Following \cite[Definition 4.1]{Iz1}, 
we say that $\rho$ has a Connes-Takesaki module 
if the canonical extension $\trho$ satisfies 
$\trho(Z(\tM))=Z(\tM)$. 
We denote by $\mo(\rho)$ the restriction of $\trho$ to $Z(\tM)$ 
and by $\End(M)_{{\rm CT}}$ the set of endomorphisms 
with Connes-Takesaki modules. 

Let $M$ be an infinite factor. 
Take a dominant weight $\vph$ on $M$ 
and a one-parameter unitary group $\{v(t)\}_{t\in\R}$ such that 
$\si_s^\vph(v(t))=e^{-ist}v(t)$ as before. 
Set $\th_s:=\Ad v(s)|_{M_\vph}\in\Aut(M_\vph)$ as before. 
Taking Lemma \ref{lem: duality} into account, 
we also denote by $\th$ the dual action $\wdh{\si^\vph}$ on $\tM$. 

\begin{lem}\label{lem: gen-trace}
Let $M$ be a type III$_\la$ factor with $0<\la<1$. 
Let $\rho\in \End(M)_{{\rm CT}}$. 
Assume that there exists $s_0\in\R$ such that 
$\mo(\rho)=\th_{s_0}$. 
Then for any generalized trace $\ps$ on $M$, 
there exists a unitary $u\in M$ such that 
in the discrete decomposition $M=M_\ps\rti_\si \Z$,
\begin{enumerate}

\item 
$\ps\circ \ph_{\Ad u\circ \rho}=d(\rho)^{-1} e^{-s_0} \ps$, 

\item $\Ad u\circ \rho(U)=U$, where $U$ is the unitary implementing 
$\si$. 
\end{enumerate}
\end{lem}
\begin{proof}
By Lemma \ref{lem: duality}, 
we may assume that
$(\vph, \rho)$ is an invariant pair,
$\rho(v(s))=v(s)$ for all $s\in\R$
and $\rho|_{Z(M_\vph)}=\th_{s_0}$.
Since each generalized trace is unique up to 
scalar multiplications and inner perturbations, 
to achieve (1) for each generalized trace, 
it suffices to 
construct one generalized trace $\ps$ satisfying (1). 

We put $T=-2\pi/\log\la$. 
Then $\si_T^\vph=\Ad h^{iT}$ for some invertible $h\in Z(M_\vph)_+$. 
Since $\si_T^\vph(v(t))=e^{-itT}v(t)$, 
we have $\th_t(h^{-iT})=e^{-itT} h^{-iT}$, 
We set $\ps:=\vph_{h^{-1}}$. 
Then it is well-known that $\ps$ is a generalized trace. 
Indeed, it is trivial that $\ps$ has the period $T$. 
So, we have to check $\ps(1)=\vph(h^{-1})=\infty$, 
but it is also trivial since $\vph$ is a dual weight. 
Set $w:=v(s_0)^*$. 
Then $\Ad w\circ \rho=\id$ on $Z(M_\vph)$, and $\Ad w\circ \rho(h)=h$. 
Hence we have 
\begin{align*}
\ps\circ \ph_{\Ad w\circ \rho}
=&\,
\vph_{h^{-1}}\circ \ph_{\Ad w\circ \rho}
=
(\vph\circ \ph_{\Ad w\circ \rho})_{h^{-1}}
\\
=&\,
(\vph\circ \ph_{\rho}\circ \Ad w^*)_{h^{-1}}
=
d(\rho)^{-1}(\vph\circ \Ad v(s_0))_{h^{-1}}
\\
=&\,
d(\rho)^{-1}e^{-s_0}\vph_{h^{-1}}
=
d(\rho)^{-1}e^{-s_0}\ps.
\end{align*}
Therefore, we may assume that $\rho$ has the property 
$\ps\circ \ph_\rho=d(\rho)^{-1}e^{-s_0}\ps$. 
As is explained, we also may assume that $\ps$ is a given 
generalized trace. 

Now let $M=M_\ps\rti_\si \Z$ be the discrete decomposition
with the implementing unitary $U$.
Since $\si^\ps$ and $\rho$ commutes,
we see that $\rho(U)U^*$ is in $M_\ps$.
By stability of $\si=\Ad U|_{M_\ps}$ \cite[Theorem III.5.1 (i)]{CT},
we can take a unitary $u\in M_\ps$ such that
$\rho(U)U^*=u^*\si(u)=u^*UuU^*$.
Hence we have $\Ad u\circ \rho(U)=U$.
Since $u\in M_\ps$,
we have $\ps\circ \ph_{\Ad u\circ \rho}=d(\rho)^{-1} e^{-s_0} \ps$. 
\end{proof}

\begin{lem}
Let $M$ be a type III$_0$ factor. 
Let $\rho\in \End(M)_{{\rm CT}}$. 
Assume that there exists $s_0\in\R$ such that 
$\mo(\rho)=\th_{s_0}$. 
Then there exists $\ps\in W_{\rm{lac}}(M)$ and $u\in U(M)$ 
such that 
\begin{enumerate}
\item 
$\ps$ has infinite multiplicity, 

\item 
$\ps\circ \ph_{\Ad u\circ \rho}=d(\rho)^{-1} e^{-s_0} \ps$, 

\item 
$\Ad u\circ \rho|_{Z(M_\ps)}=\id$. 
\end{enumerate}
\end{lem}
\begin{proof}
There exists $u\in U(M)$ satisfying Lemma \ref{lem: duality}. 
Replacing $\rho$ with $\Ad u\circ \rho$, 
we may assume that 
$(\vph,\rho)$ is an invariant pair,
$\rho(v(s))=v(s)$ 
and 
$\rho|_{Z(M_\vph)}=\th_{s_0}|_{Z(M_\vph)}$. 
By perturbing $\rho$ to $\Ad v(s_0)^*\circ \rho$ again, 
we may and do assume that $\rho$ satisfies 
$\vph\circ\ph_\rho=d(\rho)^{-1} e^{-s_0}\vph$, 
$\rho(v(s))=v(s)$ and $\rho|_{Z(M_\vph)}=\id$. 

Take $h\in Z(M_\vph)_+$ such that a normal semifinite weight 
$\chi=\vph_h$ is lacunary. 
Let $e\in Z(M_\vph)$ be the support projection of $h$. 
Then $\chi$ is faithful on $eM e$. 
Since $\rho|_{Z(M_\vph)}=\id$, 
we have $\rho(h)=h$ and $\ph_\rho(h)=\ph_\rho(\rho(h))=h$. 
Then 
\[
\chi\circ \ph_\rho=\vph_h\circ \ph_\rho
=(\vph\circ\ph_\rho)_{\rho(h)}
=d(\rho)^{-1} e^{-s_0}\vph _h
=d(\rho)^{-1} e^{-s_0}\chi. 
\]
Since $e M_\vph e\subs M_\chi\subs eMe$
 and $M_\vph '\cap M=Z(M_\vph)$ 
by Connes-Takesaki relative commutant theorem 
\cite[Theorem II.5.1]{CT}, 
we have 
$Z(M_\chi)\subs (e M_\vph e)'\cap e M e=Z(M_\vph)e$. 
Hence $\rho|_{Z(M_\chi)}=\id$. 

Take an isometry $w\in M$ such that $ww^*=e$. 
Then the map $\pi:M\ni x\mapsto wxw^* \in eMe$ is an isomorphism. 
We set $\chi':=\chi\circ \pi\in W_{\rm lac}(M)$ 
and $\rho':=\pi^{-1}\circ \rho\circ \pi\in \End(M)_0$. 
Then $\rho'=\id$ on $Z(M_{\chi'})$. 
It is easy to see that $v:=w^* \rho(w)\in M$ is unitary, and 
$\rho'=\Ad v\circ \rho$. 
Hence $\ph_{\rho'}=\ph_\rho\circ \Ad v^*$ and 
\begin{align*}
\chi'\circ \ph_{\rho'}
=&\,(\chi\circ\Ad w)\circ (\ph_{\rho}\circ \Ad v^*)
=\chi\circ \ph_{\rho}\circ \Ad \rho(w)\rho(w^*)w
\\
=&\,d(\rho)^{-1} e^{-s_0}\chi\circ \Ad \rho(e)w
=d(\rho)^{-1} e^{-s_0}\chi\circ\Ad w
=d(\rho)^{-1} e^{-s_0}\chi'. 
\end{align*}

Note that $\chi'$ may not have infinite multiplicity. 
Consider the weight $\ps':=\chi'\oti \Tr$ and 
the endomorphism $\rho'\oti\id$ on $M\oti B(\el_2)$. 
Then by Lemma \ref{lem: splitting B(H)}, 
there exists an isomorphism $\pi'\col M\oti B(\el_2)\ra M$ 
and a unitary $u\in U(M)$ such that 
$\Ad u\circ\rho'
=\pi'\circ (\rho'\oti\id)\circ {\pi'}^{-1}$. 
Set $\ps:=(\chi'\oti \Tr)\circ {\pi'}^{-1}$. 
Then this $\ps$ and $u\in U(M)$ are desired ones. 
\end{proof}

Let $M=M_\ps\rti_\si \Z$ be
the discrete decomposition of $M$ with the implementing unitary $U$.
Then the same proof of \cite[Lemma 2]{KST} works in
our case, and
we can take a unitary $v\in U(M_\ps)$ with
$\Ad vu\circ \rho(U)=U$. Hence the following holds.

\begin{lem}\label{lem: lacunary}
Let $M$ be a factor of type III$_0$. 
Let $\rho\in \End(M)_0$. 
Suppose that $\mo(\rho)=\th_{s_0}$ for some $s_0\in \R$. 
Then there exist $\ps\in W_{\rm{lac}}(M)$ with infinite multiplicity 
and $u\in U(M)$ such that 
\begin{enumerate}
\item 
In the discrete decomposition $M=M_\ps\vee\{U\}''$, 
we have 
$\Ad u\circ\rho|_{Z(M_\ps)}=\id$,

\item $\ps\circ \ph_{\Ad u\circ \rho}=d(\rho)^{-1}e^{-s_0}\ps$, 

\item $\Ad u\circ \rho(U)=U$. 
\end{enumerate}
\end{lem}

Now we prove the main result of this section. 

\begin{thm}\label{thm: main-app}
Let $M$ be a hyperfinite factor. 
Let $\rho\in \End(M)_0$ and $r>0$. 
Then the following conditions are equivalent: 
\begin{enumerate}
\item $\rho\in \oInt_r(M)$, 

\item 
$\rho\in\End(M)_{\rm CT}$ and 
$\mo(\rho)=\th_{\log(r/d(\rho))}$.
\end{enumerate}
\end{thm}

\noindent 
$\bullet$ \textit{Proof of (2)$\Rightarrow$(1) in 
Theorem \ref{thm: main-app} for type I factors.} 
 
Assume that $\rho\in \End(M)_{\rm CT}=\End(M)$ 
has the Connes-Takesaki module $\mo(\rho)=\th_{\log(r/d(\rho))}$. 
Since any endomorphism on $M$ is inner, 
$\rho\in\Int_{d(\rho)}(M)$ and $\th_{\log(r/d(\rho))}=\id$. 
The space of the flow of weights of a type I factor 
is isomorphic to $\R$ with the additive translation flow. 
Hence $\log(r/d(\rho))=0$, and $r=d(\rho)$. 
\hfill$\Box$

\noindent 
$\bullet$ \textit{Proof of (2)$\Rightarrow$(1) in 
Theorem \ref{thm: main-app} for the hyperfinite type II$_1$ factor.} 

Let $\ta\in M_*$ be the tracial state. 
Then the canonical core $\tM=M\rti_{\si^\ta}\R$ is naturally 
regarded as $M\oti \{\la^\ta(t)\}_{t\in\R}''$, 
and $Z(\tM)=\{\la^\ta(t)\}_{t\in\R}''$. 
Hence $\rho$ has the Connes-Takesaki module 
with $\mo(\rho)=\th_{\log(r/d(\rho))}$ 
if and only if 
$d(\rho)^{it}[D\ta\circ \ph_\rho:D \ta]_t=d(\rho)^{it}r^{-it}$ 
for all $t\in\R$. 
This implies $\ta\circ \ph_\rho=r^{-1}\ta$, 
but $\ta(1)=1$ yields that $r$ must be equal to $1$. 
Then $\ta\circ \ph_\rho=\ta$. 
Since $\rho$ preserves the tracial state $\ta$, 
we see that $\rho$ is approximately inner of rank $1$ 
by Lemma \ref{lem: finite-int}. 
\hfill$\Box$

\noindent 
$\bullet$ \textit{Proof of (2)$\Rightarrow$(1) in 
Theorem \ref{thm: main-app} for the hyperfinite type II$_\infty$ factor.} 

Let $\ta$ be the normal semifinite tracial weight on $M$.
Then $Z(\tM)=\{\la^\ta(t)\}_{t\in\R}''$,
and $\ta\circ \ph_\rho=r^{-1}\ta$ holds. 
Then Lemma \ref{lem: II-infty factor}
implies that $\rho\in\oInt_r(M)$. 
\hfill$\Box$

\noindent 
$\bullet$ \textit{Proof of (2)$\Rightarrow$(1) in 
Theorem \ref{thm: main-app} for the hyperfinite type III$_1$ factor.} 

We make use of Popa's result on approximate innerness of hyperfinite 
subfactors of type III$_1$. 
Since the flow of weights is trivial, $\End(M)_0=\End(M)_{\rm CT}$ 
and the modules of endomorphisms are trivial. 
Hence we have to prove $\End(M)_0=\oInt_r(M)$ for all $r>0$. 
Let $\rho\in \End(M)_0$.
Take any $\mu_0,\mu_1>0$ with $\mu_0+\mu_1=1$. 
Set $r:=\mu_1/\mu_0$. 
Consider the locally trivial subfactor 
$N^{(\id,\rho)}\subs M^{(\id,\rho)}$ 
with the expectation $E^{(\mu_0,\mu_1)}$. 
Then the subfactor is approximately inner 
by \cite[Theorem 2.9 (iv)]{Po1}. 
We note that that Popa's result states for minimal expectations, 
but the same proof is applicable for general expectations 
because we can prove that 
the Jones projections given in his proof are contained in the centralizer of 
the given state. 
Hence $\rho\in\oInt_r(M)$ by Theorem \ref{thm: rho-app}. 
\hfill$\Box$

\noindent 
$\bullet$ \textit{Proof of (2)$\Rightarrow$(1) in 
Theorem \ref{thm: main-app} for hyperfinite type III$_0$ factors.} 

We make use of the discrete decomposition of $M$ to reduce the problem 
to that of a type II von Neumann algebra. 
By Lemma \ref{lem: lacunary} for $s_0=\log(r d(\rho)^{-1})$, 
after perturbing $\rho$ by an inner automorphism, 
we may assume that 
there exists a lacunary weight $\ps$ on $M$ with infinite multiplicity 
such that 
$\ps\circ \ph_{\rho}=d(\rho)^{-1}e^{-s_0}\ps=r^{-1}\ps$, 
$\rho|_{Z(M_\ps)}=\id$ and $\rho(U)=U$, 
where $U$ is the implementing unitary 
in the discrete decomposition $M=M_\ps\rti_\si\Z$. 
Since $\ph_\rho\col M\ra M$ is the standard left inverse, 
we have $\Ad U\circ \ph_\rho=\ph_\rho\circ \Ad U$ on $M$ by uniqueness. 

Set $\ta:=\ps|_{M_\ps}$, which is a faithful normal semifinite trace 
on the type II$_\infty$ von Neumann algebra $N:=M_\ps$. 
Lemma \ref{lem: typeII} shows that 
$\rho|_{N}$ is approximately inner 
of rank $r$ with respect to $\ph_\rho|_N$. 
Hence there exist partial isometries $\{v_i^\nu\}_{i=1}^{[r]+1}$, 
$\nu\in\N$, in $N$ such that 
$(v_i^\nu)^* v_j^\nu=\de_{i,j}1$ for $1\leq i,j\leq [r]$, 
$\sum_{i=1}^{[r]+1}v_i^\nu (v_i^\nu)^*=1$ and
\begin{equation}\label{eq: rv}
\displaystyle \lim_{\nu\to\infty}
\|r^{-1} v_i^\nu \cdot \chi-\chi\circ \ph_\rho \cdot v_i^\nu\|_{N_*}=0
\quad\mbox{for all}\ \chi\in N_*.
\end{equation}
Since $\si\circ \ph_\rho=\ph_\rho\circ \si$ on $N$,
we also have 
\[
\lim_{\nu\to\infty}
\|r^{-1} \si(v_i^\nu) \cdot \chi
-\chi\circ \ph_\rho \cdot \si(v_i^\nu)\|_{N_*}
=0\quad
\mbox{for all}\ \chi\in N_*. 
\]
This implies that $((v_i^\nu)^*\si(v_j^\nu))_\nu$ 
and $((v_i^\nu)^* v_j^\nu)_\nu$ are central sequences in $N$. 
Recall the quotient map $\pi_\om\col \meN(\meT_\om)\ra N^\om$. 
We set the following elements: 
\[
v_i:=\pi_\om((v_i^\nu)_\nu), 
\quad 
w_{i j}:=v_i^* \si^\om(v_j), 
\quad 
p_i:=v_i^* v_i. 
\] 
Then we have $v_i^* v_j=\de_{ij}p_i$. 
Moreover, $w_{ij}$ is in $N_\om$, 
and $p_i=1$ for $1\leq i\leq [r]$ 
and $p_{[r]+1}$ is a projection in $N_\om$.
On $w_{i,j}$, we have the following relations: 
\begin{equation}\label{eq: ww*}
\sum_{j=1}^{[r]+1}w_{ij}w_{kj}^*=\de_{ik}p_i,\quad 
\sum_{j=1}^{[r]+1}w_{jk}^*w_{ji}=\de_{ik}\si_\om(p_i) 
\quad\mbox{for all}\ 1\leq i,k\leq [r]+1, 
\end{equation}
and 
\begin{equation}\label{eq: si-v}
\si^\om(v_j)=\sum_{i=1}^{[r]+1} v_i w_{ij}. 
\end{equation}

We will prove the following two claims. 

\noindent\textbf{Claim 1.} 
We can replace a sequence $(v_{[r]+1}^\nu)_\nu$ 
so that $\si_\om(p_{[r]+1})=p_{[r]+1}$. 

\noindent(Proof of Claim 1.) 
Consider the von Neumann algebra $N_\om\oti B(\C^{[r]+1})$. 
We set
$
w:=\sum_{i,j=1}^{[r]+1}w_{ij}\oti e_{ij}$
and
$
p:=\sum_{i=1}^{[r]+1}p_i\oti e_{ii}$,
where $\{e_{ij}\}_{i,j=1}^{[r]+1}$
is a system of matrix units of $B(\C^{[r]+1})$.
The equality (\ref{eq: ww*}) yields
\[
ww^*=p,\quad w^*w=(\si_\om\oti\id)(p).
\]
In particular, $p$ and $(\si_\om\oti\id)(p)$ are equivalent 
in $N_\om\oti B(\C^{[r]+1})$. 
Since $N_\om\oti B(\C^{[r]+1})$ is finite, 
$1-p=p_{[r]+1}\oti e_{[r]+1,[r]+1}$ and 
$1-(\si_\om\oti\id)(p)=\si_\om(p_{[r]+1})\oti e_{[r]+1,[r]+1}$ 
are also equivalent 
in $N_\om\oti B(\C^{[r]+1})$. 
Hence $p_{[r]+1}$ and $\si_\om(p_{[r]+1})$ are equivalent in $N_\om$. 
Take a unitary $v\in N_\om$ such that $\si_\om(p_{[r]+1})=v^*p_{[r]+1}v$. 

We note that the $\Z$-action $\si_\om$ on $N_\om$ is stable 
\cite[Lemma 4]{KST}. 
Hence we can take a unitary $u\in N_\om$ such that $v=u\si_\om(u^*)$. 
Then we have $\si_\om(u^* p_{[r]+1} u)=u^* p_{[r]+1} u$. 
Let $(u^\nu)_\nu$ be a representing sequence of $u$ 
such that $u^\nu$ is a unitary for all $\nu\in\N$. 
When we replace $v_{[r]+1}^\nu$ with $v_{[r]+1}^\nu u^\nu$ 
and choose a subsequence of $(v_{[r]+1}^\nu u^\nu)_\nu$, 
we see that the new family 
$\{v_i^\nu\}_{i=1}^{[r]}\cup \{v_{[r]+1}^\nu u^\nu\}$ 
also satisfies the above conditions (1), (2), (3) 
and also $\si_\om(p_{[r]+1})=p_{[r]+1}$. 
\hfill$\Box$

By using Claim 1, we assume that $\si_\om(p_{[r]+1})=p_{[r]+1}$. 

\noindent\textbf{Claim 2.} 
We can replace the sequences $(v_i^\nu)_\nu$ for $1\leq i\leq [r]+1$ 
so that $\si(v_i)-v_i\to0$ strongly* as $\nu\to\infty$ 
for all $1\leq i\leq [r]+1$. 

\noindent(Proof of Claim 2.) 
Since $ww^*=p=(\si_\om\oti\id)(p)=w^*w$, 
$w$ is a unitary in $p(N_\om\oti B(\C^{[r]+1}))p$. 
By our assumption, we can consider the reduced $\Z$-action 
$(\si_\om\oti\id)^p$. 
It is easy to see that $(\si_\om\oti\id)^p$ also has stability 
by using Rohlin towers in $Z(N)$ 
as in the proof of \cite[Lemma 3, 4]{KST}. 
Hence there exists a unitary $\mu \in p(N_\om\oti B(\C^{[r]+1}))p$ 
such that 
\begin{equation}\label{eq: wmu}
w=\mu (\si_\om\oti\id)(\mu^*).
\end{equation}
Now we set 
\begin{equation}\label{eq: ov}
\ovl{v}_i:=\sum_{j=1}^{[r]+1}v_j \mu_{ji}\in N^\om,
\end{equation}
where $\mu_{ji}$ is the $(j,i)$-entry of $\mu$.
Then we have 
\begin{align*}
(\ovl{v}_i)^*\ovl{v}_j
=&\,
\sum_{k,\el=1}^{[r]+1} \mu_{ki}^*v_k^* v_\el \mu_{\el j}
=
\sum_{k=1}^{[r]+1} \mu_{ki}^* p_k \mu_{k j}
\\
=&\,
(\mu^* p \mu)_{ij}
=
(\mu^* \mu)_{ij}
=
\de_{ij}p_i. 
\end{align*} 

Using (\ref{eq: si-v}), (\ref{eq: wmu}) and (\ref{eq: ov}), we obtain 
\begin{align}
\si^\om(\ovl{v}_i)
=&\,
\sum_{j=1}^{[r]+1}\si^\om(v_j) \si_\om(\mu_{ji})
=
\sum_{j,k=1}^{[r]+1}v_k w_{kj}\si_\om(\mu_{ji})
\notag\\
=&\,
\sum_{k=1}^{[r]+1}v_k \mu_{ki}
=\ovl{v}_i.
\label{eq: si-ov} 
\end{align}

Now we take a representing sequence $(\mu^\nu)_\nu\in N\oti B(\C^{[r]+1})$ 
of $\mu$ such that 
\[
(\mu^\nu)^*\mu^\nu
=\sum_{i=1}^{[r]}(1\oti e_{ii})
+(v_{[r]+1}^\nu)^*(v_{[r]+1}^\nu)\oti e_{[r]+1,[r]+1}
=\mu^\nu (\mu^\nu)^*.
\] 
Using $(\mu^\nu)_\nu$, 
we take a representing sequence $(\ovl{v}_i^\nu)_\nu$ 
of $\ovl{v}_i$ defined by 
\[
\ovl{v}_i^\nu:=\sum_{j=1}^{[r]+1}v_j^\nu \mu_{ji}^\nu 
\]
for all $1\leq i\leq[r]+1$ and $\nu\in\N$. 
Then we have 
\begin{align*}
(\ovl{v}_i^\nu)^* \ovl{v}_j^\nu
=&\,
\sum_{k,\el=1}^{[r]+1} 
(\mu_{k i}^\nu)^*(v_k^\nu)^*v_\el^\nu \mu_{\el j}^\nu
=
\sum_{k=1}^{[r]+1} 
(\mu_{k i}^\nu)^*\mu_{k j}^\nu
=
\de_{ij} (v_i^\nu)^* v_j^\nu, 
\end{align*}
and for $\chi\in N_*$,
\begin{align}
\lim_{\nu\to\om}
\|r^{-1} \ovl{v}_i^\nu \cdot \chi
-\chi\circ \ph_\rho \cdot \ovl{v_i}^\nu
\|_{N_*}
\leq &\,
\lim_{\nu\to\om}
\sum_{j=1}^{[r]+1}
\|r^{-1} (v_j^\nu \mu_{ji}^\nu)\cdot \chi
-\chi\circ \ph_\rho \cdot (v_j^\nu \mu_{ji}^\nu)
\|_{N_*}
\notag\\
=&\,
\lim_{\nu\to\om}
\sum_{j=1}^{[r]+1}
\|r^{-1} v_j^\nu \cdot \chi\cdot \mu_{ji}^\nu
-\chi\circ \ph_\rho \cdot (v_j^\nu \mu_{ji}^\nu)
\|_{N_*}
\notag\\
\leq&\,
\lim_{\nu\to\om}
\sum_{j=1}^{[r]+1}
\|r^{-1} v_j^\nu \cdot \chi
-\chi\circ \ph_\rho \cdot v_j^\nu 
\|_{N_*}=0, 
\label{leq: rov}
\end{align}
where we have used (\ref{eq: rv}) 
and $\pi_\om((\mu_{ji}^\nu)_\nu)=\mu_{ji}\in N_\om$. 
Moreover $\si^\om(\ovl{v}_i)=\ovl{v}_i$ 
shows the following strong* convergence: 
\begin{equation}\label{eq: siov}
\lim_{\nu\to\om}
\si(\ovl{v}_i^\nu)-\ovl{v}_i^\nu=0. 
\end{equation}

Therefore there exists a subsequence of $(\ovl{v}_i^\nu)_\nu$ 
such that the above limits (\ref{leq: rov}) and (\ref{eq: siov}) 
are taken by $\nu\to\infty$. 
This is a desired one in Claim 2. 
\hfill$\Box$

Let us take $(v_i^\nu)_\nu$ satisfying Claim 2. 
We show that 
\begin{equation}\label{eq: rov}
\lim_{\nu\to\infty}
\left\|r^{-1}v_i^\nu \cdot \chi
-\chi\circ \ph_\rho \cdot v_i^\nu
\right\|_{M_*}=0
\quad\mbox{for all}\ 1\leq i\leq [r]+1,\  \chi\in M_*. 
\end{equation}
This implies that $\rho\in\oInt_r(M)$. 

For $k\in\Z$, 
set a bounded linear map $\mE_k$ on $M$ such that 
$\mE_k(xU^\el)=\de_{k\el}xU^\el$ for all $x\in N$ and $\el\in\Z$. 
Since $\{\chi\circ\mE_k\mid \chi\in M_*,\ k\in\Z\}$ is total in $M_*$,
it suffices to prove (\ref{eq: rov}) for $\chi\circ\mE_k$, $\chi\in M_*$.
Let $x\in M$ and
$
x=\sum_{k\in\Z} x_k U^k$ be
the (formal) 
expansion of $x$.
Then 
\begin{align*}
(r^{-1}v_i^\nu \cdot \chi\circ\mE_k
-\chi\circ\mE_k \circ \ph_\rho\cdot v_i^\nu)(x)
=&\,
r^{-1} \chi(\mE_k(x v_i^\nu))
-
\chi(\mE_k(\ph_\rho(v_i^\nu x)))
\\
=&\,
r^{-1}\chi(x_k \si^k(v_i^\nu)U^k)
-
\chi(\ph_\rho(v_i^\nu x_k)U^k)
\\
=&\,
(r^{-1} \si^k(v_i^\nu) \cdot (U^k \chi )
-
(U^k \chi)\circ\ph_\rho \cdot v_i^\nu )(x_k). 
\end{align*}
Since $\|x_k\|\leq \|x\|$, we have
\[
\|r^{-1}v_i^\nu \cdot \chi\circ\mE_k
-\chi\circ\mE_k \circ \ph_\rho\cdot v_i^\nu\|_{M_*}
\leq 
\|r^{-1} \si^k(v_i^\nu) \cdot (U^k \chi)|_N
-
(U^k \chi)|_N \circ\ph_\rho \cdot v_i^\nu\|_{N_*}. 
\]

Since $\si^k(v_i^\nu)-v_i^\nu\to0$ strongly* 
as $\nu\to\infty$, 
we have 
\begin{align*}
&\varlimsup_{\nu\to\infty}
\|r^{-1}v_i^\nu \cdot \chi\circ\mE_k
-\chi\circ\mE_k \circ \ph_\rho\cdot v_i^\nu\|_{M_*}\\
\leq &\,
\varlimsup_{\nu\to\infty}
\|r^{-1} \si^k(v_i^\nu) \cdot (U^k \chi)|_N
-
(U^k \chi)|_N \circ\ph_\rho \cdot v_i^\nu\|_{N_*}
\\
\leq &\,
\varlimsup_{\nu\to\infty}
\|r^{-1} \si^k(v_i^\nu) \cdot (U^k \chi)|_N
-
r^{-1} v_i^\nu \cdot (U^k \chi)|_N\|_{N_*}
\\
&\quad+
\varlimsup_{\nu\to\infty}
\|r^{-1} v_i^\nu \cdot (U^k \chi)|_N
-
(U^k \chi)|_N \circ\ph_\rho \cdot v_i^\nu\|_{N_*}
=0. 
\end{align*}
This shows (\ref{eq: rov}). 
\hfill$\Box$

\noindent 
$\bullet$ \textit{Proof of (2)$\Rightarrow$(1) in 
Theorem \ref{thm: main-app} for the hyperfinite type III$_\la$ factor 
with $0<\la<1$.} 

We prove $\rho\in\oInt_r(M)$ along with the proof
of type III$_0$ case by using Lemma \ref{lem: gen-trace}.
Let $\ps$ be a generalized trace and
$M=M_\ps\rti_\si\Z$ the discrete decomposition.
Then the $\Z$-action $\si_\om$ on $(M_\ps)_\om$
is stable \cite[Theorem 2.1.3]{C2}.
Here we note that $\si$ is centrally free, i.e., $\si^n\in\Cnt(M_\ps)$ 
if and only if $n=0$, 
which follows from the fact that $\Cnt(M_\ps)=\Int(M_\ps)$
\cite[Lemma 5]{C2}.
So, Claim 1 still holds. 
Since any reduction of outer automorphism on a factor is still outer, 
Claim 2 also holds.
Thus we can prove that $\rho\in\oInt_r(M)$ as in type III$_0$ case.
\hfill$\Box$

Therefore we have proved the implication 
(2)$\Rightarrow$(1) in Theorem \ref{thm: main-app}.

\noindent 
$\bullet$ \textit{Proof of (1)$\Rightarrow$(2) in 
Theorem \ref{thm: main-app}.} 
 
First we assume $r\in\N$. 
Take $r$-dimensional Hilbert spaces $\meH^\nu\subs M$ with support $1$
for $\nu\in\N$, 
such that $\rho_{\meH^\nu}$ converges to $\rho$ as $\nu\to\infty$. 
Recall that the normalized canonical extension is
continuous by Theorem \ref{thm: cont-nce}.
Hence $\rho_{\meH^\nu}\to\rho$ implies 
$\th_{\log(d(\rho_{\meH^\nu}))}(\trho_{\meH^\nu}(x))
\to\th_{\log(d(\rho))}(\trho(x))$ strongly* as $\nu\to\infty$ 
for all $x\in\tM$. 
Using $d(\rho_{\meH^\nu})=r$ and $\trho_{\meH^\nu}|_{Z(\tM)}=\id$, 
we have 
$\th_{\log(d(\rho))}\trho|_{Z(\tM)}=\th_{\log r}$. 
This shows that $\rho$ has the Connes-Takesaki module 
$\mo(\rho)=\th_{\log(r/d(\rho))}$. 

Next we treat a general case. 
Let $r>0$. 
Take $s\in\R$ such that $e^s r\in\N$. 
Since
$\mo\col \Aut(M)\ra \Aut_\th(Z(\tM))$ is surjective 
\cite{Ham, ST2},
there exists $\al\in\Aut(M)$ such that $\tal=\th_s$ on $Z(\tM)$. 
By using (2)$\Rightarrow$(1) of Theorem \ref{thm: main-app}, which 
we have already proved, 
we see that $\al\in\oInt_{e^s}(M)$. 
Then $\al\circ \rho\in\oInt_{e^s r}(M)$ by 
Lemma \ref{lem: composition}. 
Since $e^s r\in\N$, 
$\al\rho$ has the Connes-Takesaki module and 
$\mo(\al\rho)=\th_{\log(e^s r/d(\rho))}$ on $Z(\tM)$. 
This also shows that $\rho$ has a Connes-Takesaki module, 
and 
\begin{align*}
\mo(\rho)
=&\,\mo(\al^{-1})\mo(\al\rho)
=\tal^{-1}\circ \th_{\log(e^s r/d(\rho))}
=\th_{-s}\circ\th_{\log(e^s r/d(\rho))}
\\
=&\,\th_{\log(r/d(\rho))}.
\end{align*} 
\hfill$\Box$

\begin{cor}
Let $M$ be a hyperfinite factor and $r,s>0$.
\begin{enumerate}

\item 
When $M$ is of type I, then 
\[
\End(M)_0=\bigcup_{n\in\N}\Int_n(M). 
\]

\item 
When $M$ is of type II$_1$ with the tracial state $\ta$, then 
\begin{align*}
\oInt_1(M)
=&\,\{\rho\in\End(M)_0\mid\ta\circ \ph_\rho=\ta\},
\\
\oInt_r(M)
=&\,
\emptyset\ \ \mbox{if}\ r\neq1. 
\end{align*}

\item 
When $M$ is of type II$_\infty$ with the trace $\ta$, then 
\[
\oInt_r(M)=\{\rho\in \End(M)_0\mid \ta\circ \ph_\rho=r^{-1}\ta\}. 
\]

\item 
When $M$ is of type III$_1$, then 
\[
\End(M)_0=\oInt_r(M).
\] 

\item 
When $M$ is of type III$_\la$ with $0<\la<1$ 
with the generalized trace $\vph$, 
then 
\begin{align*}
\oInt_r(M)
=&\,
\{\rho\in\End(M)_0\mid \vph\circ \ph_\rho=r^{-1}\vph\circ\Ad u 
\ \mbox{for some}\ u\in U(M)\}\\
=&\,
\oInt_{r\la^n}(M)\ \mbox{for all}\ n\in\Z,
\\
&\hspace{-1.9cm}\oInt_r(M)\cap \oInt_s(M)=\emptyset 
\ \mbox{if}\  s\neq r\la^n\ \mbox{for all}\ n\in\Z
\end{align*}

\item When $M$ is of type III$_0$,
\[ 
\oInt_r(M)\cap \oInt_s(M)=\emptyset\ \mbox{if}\ r\neq s. 
\] 
\end{enumerate}
\end{cor}

In the next section, we study a relation between
centrally trivial endomorphisms 
and modular endomorphisms \cite[Definition 3.1]{Iz1}.
Note that the statistical dimension of a modular endomorphism
is an integer.

\begin{cor}
Let $M$ be a hyperfinite factor and $\rho$ a modular endomorphism 
on $M$. 
Then $\rho\in \oInt_{d(\rho)}(M)$.
\end{cor}
\begin{proof}
Since $\trho$ is inner, $\mo(\rho)=\id=\th_{\log(d(\rho)/d(\rho))}$. 
By Theorem \ref{thm: main-app}, we see that $\rho\in \oInt_{d(\rho)}(M)$. 
\end{proof}

For conjugation of an endomorphism, a rank of approximate innerness 
behaves as follows. 
It seems that the assumption on hyperfiniteness is unnecessary, 
but we have no proof so far. 

\begin{cor}[Conjugation rule]\label{cor: conjugation-rule}
Let $M$ be a hyperfinite infinite factor and 
$\rho\in\oInt_r(M)$ for some $r>0$. 
Then $\brho\in\oInt_{r^{-1}d(\rho)^2}(M)$. 
In particular, 
$\rho\brho\in\oInt_{d(\rho)^2}(M)$. 
\end{cor}
\begin{proof}
By Theorem \ref{thm: main-app}, 
we have $\mo(\rho)=\th_{\log(r/d(\rho))}$, 
and 
\[
\mo(\brho)
=\mo(\rho)^{-1}=\th_{-\log(r/d(\rho))}
=\th_{\log(d(\rho)/r)}
=\th_{\log(d(\rho)^2/rd(\brho))}. 
\]
Hence $\brho\in\oInt_{d(\rho)^2/r}(M)$ 
again by Theorem \ref{thm: main-app}. 
Then Lemma \ref{lem: composition} 
implies that $\rho\brho\in\oInt_{d(\rho)^2}(M)$.
\end{proof}

\section{Centrally trivial endomorphisms}

\subsection{Centrally trivial endomorphisms}
Let $M$ be a factor and $\rho\in\End(M)_0$.
As is shown in \cite[Lemma 3.3]{MT1},
we can define $\rho^\om\in \End(M^\om)$ by
\[
\rho^\om(\pi_\om((x^\nu)_\nu))=\pi_\om((\rho(x^\nu))_\nu)
\quad\mbox{for all}\ (x^\nu)_\nu\in \meN(\meT_\om(M)). 
\]
The map
$\End(M)_0\ni \rho\mapsto \rho^\om\in\End(M^\om)$
is a semigroup homomorphism.

\begin{defn}
Let $M$ be a factor and $\rho\in \End(M)_0$.
We say that $\rho$ is \emph{centrally trivial} if and only if
$\rho^\om=\id$ on $M_\om$.
\end{defn}

We only consider centrally trivial endomorphisms with finite index.
However, we should mention that it can be defined for an endomorphism
with infinite index if this has a left inverse.
We denote by $\Cnd(M)$ the set of centrally trivial endomorphisms on $M$.
Since $\Int(M)\subset \Cnd(M)$,
the central triviality is a property for sectors.
We show that the set $\Cnd(M)$ is closed under the composition, decomposition,
direct sum and conjugation when $M$ is infinite.

\begin{defn} Let $N\subset M$ be an inclusion of factors.
\begin{enumerate}

\item 
$\meC_\omega(M,N):= \{(x^\nu)_\nu \in \el^\infty(\mathbf{N}, N)
\mid
\displaystyle\lim_{\nu\to\omega}\|[\varphi,x^\nu]\|=0, \varphi\in M_*\}$.
\item 
$C_\omega(M,N):=\meC_\omega(M,N)/\meT_\om(N)$.
\end{enumerate}
\end{defn}

If $N=M$, then it is obvious that $C_\omega(M,M)=M_\omega$.

\begin{lem}\label{lem: downward}
Let $N\subset M$ be an inclusion of factors with finite index.
Let $M\supset N\supset N_1$ be a downward basic construction
with respect to the minimal expectation $E\col M\ra N$.
Then $C_\omega(M,N)=C_\omega(M,N_1)$.
\end{lem}
\begin{proof} 
Let $(x^\nu)_\nu\in \meC_\omega(M,N)$,
$E_1$ be the
minimal expectation from $N$ onto $N_1$, and
$e\in M$ the Jones projection for $N_1\subset N$.
Then $ex^\nu e-x^\nu e=[e,x^\nu]e$
converges to $0$ $\sigma$-strongly*. Since $ex^\nu e=E_1(x^\nu)e$ and
$E(e)=[M:N]_0^{-1}$, $(E_1(x^\nu)-x^\nu)_\nu$ converges to $0$ strongly*.
Hence any element in $C_\omega(M,N)$ is represented by an element
in $\meC_\omega(M,N_1)$. 
\end{proof}

\begin{lem}\label{lem: rho}
Let $\rho\in\Cnd(M)$. 
Then $C_\omega(M,\rho(M))= M_\omega$ holds.
\end{lem}
\begin{proof}
Let $(\rho(x^\nu))_\nu \in \meC_\omega(M,\rho(M))$.
Since $\ph_\rho$ preserves $M_\om$,
$(x^\nu)_\nu=(\ph_\rho(\rho(x^\nu)))_\nu$
is an $\om$-centralizing sequence of $M$.
The central triviality of $\rho$
implies that
$(\rho(x^\nu))_\nu$ and $(x^\nu)_\nu$ are equivalent.
Hence $C_\omega(M,\rho(M))\subs M_\omega$ holds.
Since $M_\om=\rho^\om(M_\om)$, $M_\om\subs C_\omega(M,\rho(M))$ follows.
\end{proof}

\begin{lem}\label{lem: brho}
Assume that $M$ is an infinite factor.
Let $\rho\in\Cnd(M)$. Then one has $\brho\in\Cnd(M)$.
\end{lem}
\begin{proof}
First we show $\rho\brho\in\Cnd(M)$.
By Lemma \ref{lem: downward} and \ref{lem: rho},
$C(M,\rho\brho(M))=C(M,\rho(M))=M_\omega$.
This implies that for any $(x^\nu)_\nu\in\meC_\omega(M,M)$,
there exists $(y^\nu)_\nu\in\ell^\infty(\mathbf{N}, M )$
such that $x^\nu-\rho\brho(y^\nu )\to0$ strongly*
as $\nu\to\om$.
Take an isometry $v\in(\id, \rho\brho)$.
Then $v^*x^\nu v-v^*\rho\brho(y^\nu )v\to0$
as $\nu\to\om$.
Here $v^*x^\nu v-x^\nu\to 0$ as $\nu\to\om$
and $v^*\rho\brho(y^\nu )v=y^\nu $.
This yields that $x^\nu-y^\nu\to0$,
and $\rho\brho(x^\nu)-x^\nu\to0$ strongly*
as $\nu\to\om$.
Hence $\rho\brho\in\Cnd(M)$.

Second we show that $\brho\in\Cnd(M)$.
Since $\rho, \rho \bar{\rho}\in\Cnd(M)$,
$\rho(x^\nu)-\rho\bar{\rho}(x^\nu)\to0$
strongly* as $\nu\to\om$.
Applying $\phi_\rho$ to $\rho(x^\nu)-\rho\bar{\rho}(x^\nu)$,
we get the conclusion.
\end{proof} 

\begin{thm}\label{thm: cend-structure}
Assume that $M$ is an infinite factor.
Then the subset $\Cnd(M)\subs\End(M)_0$
is closed under the composition, decomposition,
direct sum and conjugation.
Namely one has the following:
\begin{enumerate}
\item 
If $\rho,\si\in\Cnd(M)$, then $\rho\si\in\Cnd(M)$.

\item 
If $\rho\in\End(M)_0$, $\si\in\Cnd(M)$ and $\rho\prec \si$,
then $\rho\in\Cnd(M)$.

\item 
If $\rho,\si\in\Cnd(M)$, then $\rho\oplus\si\in \Cnd(M)$.

\item If $\rho\in\Cnd(M)$, then $\brho\in\Cnd(M)$.
\end{enumerate}
\end{thm}
\begin{proof}
(1) It is trivial because $(\rho\si)^\om=\rho^\om\si^\om=\id$ on $M_\om$.

\noindent(2) 
Let $w\in (\rho,\si)$ be an isometry. 
Since $w\in (\rho^\om,\si^\om)$, we have for $x\in M_\om$, 
\[w\rho^\om(x)=\si^\om(x)w=xw=wx.\] 
Hence $\rho^\om(x)=w^*w\rho^\om(x)=w^*wx=x$. 

\noindent(3) 
Let $v,w\in M$ be isometries with $vv^*+ww^*=1$. 
We set $\th(x):=v\rho(x)v^*+w\si(x)w^*$ for $x\in M$. 
Then for $x\in M_\om$, we have 
\[
\th^\om(x)=v\rho^\om(x)v^*+w\si^\om(x)w^*=vxv^*+wxw^*=(vv^*+ww^*)x=x. 
\]
Hence $\rho\oplus\si\in\Cnd(M)$. 

\noindent(4) 
It has been proved in the previous lemma. 
\end{proof}

Connes showed that any centrally trivial automorphism
commutes with any approximately inner automorphism
up to inner automorphism \cite[Lemma 2.2.2]{C3}.
We generalize this to the case of endomorphisms as follows.

\begin{lem}\label{lem: commutes}
Let $M$ be a factor, $\rho\in \Cnd(M)$
and $\theta\in \overline{\Int}(M)$.
Then $\theta$ and $\rho$ commute up to inner automorphism, 
that is, 
$[\theta\circ \rho \circ \theta^{-1}]=[\rho]$ in $\Sect(M)$. 
\end{lem}
\begin{proof} 
Let $\varphi_0 $ be a faithful normal state on $\rho(M)$,
and set $\varphi :=\varphi_0\circ E_\rho$.
Let $\{\mathcal{V}_n\}_{n\in\N}$ be 
a fundamental system of the neighborhood of $\id$
in $\Aut(M)$ such that
for any $u\in U(M)$ with $\Ad u \in \mathcal{V}_n$, we have
\[
\|\rho(u)-u\|_{\varphi\circ \theta^{-1}}
<2^{-n},
\quad
\|\rho(u)-u\|_{\varphi\circ\rho\circ \theta^{-1}
\circ\phi_\rho}
<2^{-n}.
\]

If $\beta_n$ converges to $\theta$ in $\Aut(M)$, then $\rho\circ \beta_n
\circ \rho^{-1}$ converges to $\rho\circ \theta\circ \rho^{-1}$ in
$\Aut(\rho(M))$. So we can choose a monotone decreasing system of
the neighborhood of $\theta$ denoted by $\mathcal{W}_n$ 
such that $\mathcal{W}_n\mathcal{W}_n^{-1}\subset \mathcal{V}_n$ and 
for $\beta\in \mathcal{W}_n$, we have 
\[
\|\varphi\circ \beta^{-1}-\varphi\circ \theta^{-1}\|
<2^{-n},
\quad 
\|\varphi_0\circ \rho\circ\beta^{-1}\circ \rho^{-1}-
\varphi_0\circ \rho\circ\theta^{-1}\circ \rho^{-1}\|
<2^{-n}.
\]
Since $\varphi_0=\varphi$ on $\rho(M)$, we get
\[
\|\varphi\circ \rho\circ\beta^{-1}\circ \phi_\rho-
\varphi\circ \rho\circ\theta^{-1}\circ \phi_\rho\|
=
\|\varphi\circ \rho\circ\beta^{-1}\circ \rho^{-1}\circ E_\rho-
\varphi\circ \rho\circ\theta^{-1}\circ \rho^{-1}\circ E_\rho\|
<
2^{-n}.
\]
Let $\theta=\displaystyle \lim\limits_{n\rightarrow\infty }\Ad u_n$ 
with $\Ad u_n\in \mathcal{W}_n$.
Set $v_n:=u_{n+1}u_n^*$. 
Then we see that $\Ad v_n\in
\mathcal{W}_{n+1}\mathcal{W}_{n+1}^{-1}\subset \mathcal{V}_n$.
Hence 
$\|\rho(v_n)-v_n\|_{\vph\circ\th^{-1}}<2^{-n}$ and 
$\|\rho(v_n)-v_n\|_{\vph\circ \rho\circ\th^{-1}\circ\ph_\rho}<2^{-n}$. 

Now we prove
$u_n^*\rho(u_n)$ converges in $U(M)$. 
We have 
\begin{align*}
\|\rho(u_{n+1}^*)u_{n+1}-\rho(u_n^*)u_n\|_{\varphi}^2
=&\, 
\|\rho(u_{n}^*v_n^*)v_nu_{n}-\rho(u_n^*)u_n\|_{\varphi}^2 
= 
\|\rho(v_n^*)v_n-1\|_{\varphi\circ\Ad u_n^*}^2 \\ 
\leq&\, 
2\|\varphi\circ \Ad u_n^*-\varphi\circ \theta^{-1}\|+
\|\rho(v_n^*)v_n-1\|_{\varphi\circ\theta^{-1}}^2 \\ 
\leq &\, 2^{1-n}+
\|v_n-\rho(v_n)\|_{\varphi\circ\theta^{-1}}^2
\leq  2^{1-n}+4^{-n} \\
\leq&\, 2^{2-n},
\end{align*}
and 
\begin{align*}
\|u_{n+1}^*\rho(u_{n+1})-u_n^*\rho(u_n)\|_{\varphi}^2 
=&\, 
\|u_n^*v_n^*\rho(v_nu_{n})-u_n^*\rho(u_n)\|_{\varphi}^2
=
\|v_n^*\rho(v_n)-1\|_{\varphi\circ \Ad \rho(u_n^*)}^2 \\ 
\leq&\, 
2 \|\varphi\circ\Ad \rho(u_n^*)
-\varphi\circ \rho\circ \theta^{-1}\circ \phi_\rho\|
\\
&\quad
+\|v_n^*\rho(v_n)-1\|_{\varphi\circ\rho\circ\theta^{-1}\circ \phi_\rho}^2 
\\ 
\leq&\, 
2^{1-n} +4^{-n} \\ 
\leq &\, 2^{2-n},
\end{align*}
where we have used 
\[
\varphi \circ \Ad \rho(u_n^*)= 
\varphi \circ \Ad \rho(u_n^*)\circ E_\rho
= \varphi \circ \rho\circ \Ad u_n^* \circ \phi_\rho.
\]
Hence the strong* limit
$w:=\displaystyle\lim_{n\to\infty} u_n^* \rho(u_n) \in U(M)$ exists,
and we have
$\theta^{-1}\circ\rho \circ\theta
=\displaystyle\lim_{n\to\infty} \Ad u_n^*\rho(u_n)\rho =\Ad w\circ\rho$.
\end{proof}

Let $M$ be a McDuff factor, i.e.,
the von Neumann algebra $M_\om$ is of type II$_1$.
When $\al$ is a centrally trivial automorphism of $M$,
we know that $\al$ is outer conjugate to $\al\oti \id_{R_0}$,
where $R_0$ denotes the hyperfinite type II$_1$ factor.
This property also holds for centrally trivial endomorphisms.

\begin{prop}\label{prop: relMc}
Let $M$ be an infinite McDuff factor
and
$\rho\in\Cnd(M)$.
Then there exists an isomorphism $\th\col M\oti R_0\ra M$ such that
$[\rho]=[\theta\circ(\rho\otimes \id_{R_0})\circ \theta^{-1}]$
in $\Sect(M)$.
\end{prop}
\begin{proof}
If $[\rho]=[\id]$, then the proposition is trivial.
We assume that
$[\rho]\ne [\id]$.
Take $\sigma\in\End(M)$ such that
$[\sigma]=[\id]\oplus[\rho]\oplus [\rho]$.
Then $\sigma$ is centrally trivial
by Theorem \ref{thm: cend-structure}.
Hence
$C(M,\sigma( M))=M_\omega$ holds by Lemma \ref{lem: rho}.
Since $M_\omega$ is of type II$_1$,
$\sigma(M)\subset M$ is relatively McDuff in the sense of Bisch \cite{B}.
Thus $\sigma(M)\subset M$ is
isomorphic to
$(\sigma\otimes \id_{R_0})(M\otimes R_0)\subset M\otimes R_0$.
Hence we can find isomorphisms
$\theta_1$ and $\theta_2$ from $M$ onto $M\otimes R_0$ such that
$\theta_1\circ\sigma= (\sigma\otimes \id_{R_0})\circ \theta_2$.
So,
\[
[\theta_1\theta^{-1}_2]\oplus 2[\theta_1\circ\rho\circ \theta_1^{-1}]
=
[\theta_1 \sigma \theta_2^{-1}]
=
[\sigma\otimes \id_{R_0}]
=
[\id\oti \id_{R_0}]\oplus
2[\rho\otimes \id_{R_0}]
\] holds. By comparing irreducible
components, we have $[\theta_1 \theta_2^{-1}]=[\id\oti\id_{R_0}]$, and
$[\theta_1\circ\rho\circ \theta_1^{-1}]=[\rho\otimes \id_{R_0}]$.
\end{proof}

Combining Lemma \ref{lem: splitting B(H)}
with Proposition \ref{prop: relMc},
we have the following result, where $R_{0,1}$ denotes
the hyperfinite type II$_\infty$ factor.

\begin{prop}\label{prop: splitting R01}
Let $M$ be an infinite McDuff factor
and $\rho\in \Cnd(M)$.
Then there exists an isomorphism $\th\col M\oti R_{0,1}\ra M$
such that $[\rho]=[\th\circ(\rho\oti\id)\circ \th^{-1}]$
in $\Sect(M)$.
\end{prop}

\subsection{Modular endomorphisms}

Let $M$ be a factor and $\rho\in\End(M)_0$. 
We say that $\rho$ is a \emph{modular endomorphism} 
when the canonical extension $\trho$ is inner, that is,
there exists a finite dimensional Hilbert space in $\tM$
which implements $\trho$ \cite[Definition 3.1]{Iz1}.
By $\End(M)_{\rm m}$ and $\Sect(M)_{\rm m}$,
we denote the sets of modular endomorphisms on $M$
and sectors of modular endomorphisms, respectively.
We also denote by $\End(M)_{\rm m,\, irr}$ the set of 
irreducible modular endomorphisms on $M$. 
Then $\End(M)_{\rm m}$ is closed under the composition, 
decomposition, direct sum and conjugation. 

In type III case,
Izumi characterizes modular endomorphisms
in terms of $U(n)$-valued cohomology classes 
$H^1(F^M,U(n))$ with respect to the flow of weights $F^M$
\cite[Theorem 3.3]{Iz1}.
We complement his result for each type. 

\begin{lem}
Let $M$ be a factor. Then the following statements hold. 
\begin{enumerate}

\item 
When $M$ is of type II$_1$, $\End(M)_{\rm m}=\Int(M)$. 

\item 
When $M$ is of type II$_\infty$, 
$\End(M)_{\rm m,\, irr}=\Int(M)$. 

\item 
When $M$ is of type III$_\la$ $(0<\la\leq1)$, 
\[
\End(M)_{\rm m,\, irr}
=\{\Ad u\circ \si_t^\vph\mid u\in U(M),\ \vph\in W(M),\ t\in\R\}. 
\]

\item
When $M$ is of type III$_0$, 
\[
\End(M)_{\rm m}
=
\{
\rho\in \End(M)_0\mid \de_{\rm m}([\rho])\in H^1(F^M,U(n)),
\ n\in\N
\},
\]
where $\de_{\rm m}$ is a bijection 
from $\Sect(M)_{\rm m}$ onto $\displaystyle\bigcup_{n\in\N}H^1(F^M,U(n))$ 
introduced in \cite[p.10]{Iz1}. 
\end{enumerate}
\end{lem}
\begin{proof}
(3) and (4) is already proved \cite[Theorem 3.3]{Iz1}. 

\noindent (1)
Let $\ta$ be the normal tracial state on $M$.
Then the natural isomorphism
$\Pi_\ta\col \tM\ra M\rti_{\si^\ta}\R$
maps $x\in M$ to $x\oti1$.
Let $\rho\in \End(M)_{\rm m}$.
Take $\{V_i\}_{i=1}^d\subs \tM$ be
an implementing orthonormal system of $\trho$.
Then for $x\in M$,
$(\rho(x)\oti1)V_i=\trho(x\oti1)V_i
=V_i (x\oti1)$.
This implies that
$(\id,\rho)\neq 0$.
Since $M$ is finite, $(\id,\rho)$ contains a unitary,
and
$\rho$ is an inner automorphism.

\noindent (2)
The above proof for $\rho\in\End(M)_{\rm m,\,irr}$ also works
in this case.
\end{proof}

\begin{lem}\label{lem: mod-in-cent}
Let $M$ be a factor.
Then $\End(M)_{\rm m}\subs \Cnd(M)$.
\end{lem}
\begin{proof}
Let $\rho\in\End(M)_{\rm m}$ and $\ps\in W(M)$.
We identify $\tM$ with $M\rti_{\si^\ps}\R$.
Let $(x^\nu)_\nu$ be an $\om$-centralizing sequence in $M$.
We will show that $\rho(x^\nu)-x^\nu\to0$ strongly*
as $\nu\to\om$.

Note that $\pi_{\si^\ps}(x^\nu)-x^\nu\oti1\to0$
strongly as $\nu\to\om$ in $M\oti B(L^2(\R))$.
It suffices to show
$\|(\pi_{\si^\ps}(x^\nu)-x^\nu\oti1)(\xi\oti f)\|\to0$
for all $\xi\in L^2(M)$ and $f\in L^2(\R)$ with compact support.
It is checked as follows:
\begin{align*}
\|(\pi_{\si^\ps}(x^\nu)-x^\nu\oti1)(\xi\oti f)\|^2
=&\,
\int_{-\infty}^\infty 
|f(s)|^2 \|(\si_{-s}^\ps(x^\nu)-x^\nu)\xi\|^2 ds
\\
\to&\,0\quad \mbox{as}\ \nu\to\om,
\end{align*}
since $\|(\si_{-s}^\ps(x^\nu)-x^\nu)\xi\|$ uniformly converges to $0$
on the compact support of $f$
as $\nu\to\om$ by \cite[Proposition 2.3 (1)]{C1}.
Similarly we can prove
$\pi_{\si^\ps}((x^\nu)^*)-(x^\nu)^*\oti1\to0$.
Hence $\pi_{\si^\ps}(x^\nu)-x^\nu\oti1\to0$ strongly* as $\nu\to\om$.
In particular, $(\pi_{\si^\ps}(x^\nu))_\nu$
is an $\om$-centralizing
sequence in $M\oti B(L^2(\R))$.
Since $\trho$ is inner, $\trho$ acts trivially
on $\om$-centralizing sequences of $\tM$.
Thus $\pi_\ps(\rho(x^\nu))-\pi_\ps(x^\nu)
=\trho(\pi_\ps(x^\nu))-\pi_\ps(x^\nu)$ converges to 0 strongly*
as $\nu\to\om$.
\end{proof}

Our main result in this section is the following. 

\begin{thm}\label{thm: mod-cent}
Let $M$ be a hyperfinite factor. 
Then one has
\[
\End(M)_{\rm m}=\Cnd(M). 
\]
\end{thm}

We present a proof separately for hyperfinite factors of 
type I, II$_1$, II$_\infty$, III$_1$, III$_\la$ ($0<\la<1$) 
and III$_0$.
We denote by $\Cnd(M)_{\rm irr}$ the set of endomorphisms on $M$
which are with finite index, centrally trivial and irreducible.
By Theorem \ref{thm: cend-structure} and
Lemma \ref{lem: mod-in-cent}, it suffices to show that
$\Cnd(M)_{\rm irr}\subs \End(M)_{\rm m}$.

\subsection{Semifinite case}

\noindent$\bullet$
\textit{Proof of Theorem \ref{thm: mod-cent} 
for a hyperfinite type I factor.}

It is trivial because every endomorphism on $M$ is inner. 
\hfill$\Box$

\vspace{10pt}
\noindent$\bullet$
\textit{Proof of Theorem \ref{thm: mod-cent} 
for the hyperfinite type II$_\infty$ factor.}

Let $M$ be the hyperfinite type II$_\infty$ factor. 
We show that $\Cnd(M)_{\rm irr}\subs \Int(M)$.
Let $\rho\in \Cnd(M)_{\rm irr}$. 
We may assume $\rho$ is of the form $\rho\otimes
\id_{R_{0,1}}$ by Proposition \ref{prop: splitting R01}. 
If $\sigma$ is a flip automorphism of $R_{0,1}\otimes R_{0,1}$,
then $\sigma$ is approximately inner.
Hence 
$[\sigma\circ(\brho\otimes\id_{R_{0,1} })\circ\sigma^{-1}]
=[\brho\otimes \id_{R_{0,1}}]$
by
Theorem \ref{thm: cend-structure}
and Lemma \ref{lem: commutes}.
Then, 
\begin{align*}
[\rho\otimes \brho]
=&\,
[
(\rho \otimes \id_{R_{0,1}})\circ \sigma \circ 
(\brho\otimes \id_{R_{0,1}}) \circ \sigma^{-1 }
] \\
=&\, 
[
(\rho \otimes \id_{R_{0,1}})\circ 
(\brho\otimes \id_{R_{0,1}}) 
]\\
=&\, 
[
\rho\brho\otimes \id_{R_{0,1}}
] \\
\succ&\, 
[
\id_{R_{0,1}}\otimes \id_{R_{0,1}}
].
\end{align*}
Since $\rho$ is irreducible,
$\rho$ is an inner automorphism.
\hfill$\Box$

\vspace{10pt}
\noindent$\bullet$ \textit{Proof of Theorem \ref{thm: mod-cent} 
for the hyperfinite type II$_1$ factor.}

Let $M$ be the hyperfinite type II$_1$ factor. 
We show that $\Cnd(M)=\Int(M)$. 
Let $\rho\in \Cnd(M)$. 
On the type II$_\infty$ factor $N=M\oti B(\el_2)$,
we define the endomorphism $\si:=\rho\oti \id$.
We claim that $\si$ is centrally trivial.

If not, there exists $x\in N_\om$ such that $\si(x)\neq x$.
Let $\{p_i\}_{i=1}^\infty$ be a partition of unity in $B(\el_2)$
which consists of minimal projections.
Then there exists $i\geq 1$ such that 
$\si^\om(x)(1\oti p_i)\neq x(1\oti p_i)$. 
Since $x (1\oti p_i)=(1\oti p_i) x$
and $\si(1\oti p_i)=1\oti p_i$, 
we have 
$\si^\om((1\oti p_i)x (1\oti p_i))\neq (1\oti p_i)x(1\oti p_i)$. 
Let $(x^\nu)_{\nu}$ be an $\om$-centralizing sequence for $x$. 
For each $\nu\in\N$, we take $y^\nu\in M$ such that 
$(1\oti p_i)x^\nu (1\oti p_i)=y^\nu\oti p_i$. 
Then $(y^\nu)_\nu$ is an $\om$-centralizing sequence in $M$. 
Indeed, for $\vph\in M_*$, $\ps=p_i \ps p_i\in B(\el_2)_*$ and $z\in M$, 
we have 
\begin{align*}
\ps(p_i)[\vph,y^\nu](z)
=&\,
[\vph\oti \ps, y^\nu\oti p_i](z\oti1)
\\
=&\,
[\vph\oti \ps, (1\oti p_i)x^\nu (1\oti p_i)](z\oti1)
\\
=&\,
[\vph\oti \ps, x^\nu ](z\oti1), 
\end{align*}
and $|\ps(p_i)|\|[\vph,y^\nu]\|\leq \|[\vph\oti \ps, x^\nu ]\|$. 
Hence $(y^\nu)_\nu$ is $\om$-centralizing. 

Then $\si^\om((1\oti p_i)x (1\oti p_i))$ is represented by 
$\si(y^\nu \oti p_i)=\rho(y^\nu)\oti p_i$ which is equivalent to 
$y^\nu\oti p_i$ because $\rho$ is centrally trivial. 
This is a contradiction with 
$\si^\om((1\oti p_i)x (1\oti p_i))\neq (1\oti p_i)x(1\oti p_i)$. 

Hence $\rho\oti\id $ is centrally trivial, and inner
by the above result on the hyperfinite type II$_\infty$ factor. 
This implies $(\id,\rho)\neq 0$. 
Since $M$ is finite, $(\id,\rho)$ contains a unitary. 
Therefore $\rho$ is an inner automorphism. 
\hfill$\Box$

For type III cases, we prove Theorem \ref{thm: mod-cent} in the following 
subsections. 

\subsection{Type III$_1$ case}

Central freeness for subfactors is introduced by Popa
\cite[Definition 3.1]{Po1}.
Following this notion,
we prepare the notion of central non-triviality for subfactors.

\begin{defn}
Let $N\subs M$ be an inclusion of factors.
Assume that there exists a faithful normal conditional expectation
from $M$ onto $N$.
We say that $N\subs M$ is \emph{centrally non-trivial} 
when the following condition is satisfied: 
for any $\vep>0$ and any finite subset $\mF\subs M_*$, 
there exists a partition of unity $(q_j)_{j=1}^t$ in $N$ such that
\[
\left\| 
\sum_{j=1}^t q_j \vph q_j-\vph\circ E_{N\vee (N'\cap M)}^M
\right\|
<\vep
\ \mbox{for all}\ \vph\in\mF, 
\]
where $E_{N\vee (N'\cap M)}^M$ is the unique faithful
normal conditional expectation from $M$ onto $N\vee (N'\cap M)$.
\end{defn}

By definition, the central non-triviality is a condition
that is independent of the choice of a particular conditional
expectation from $M$ onto $N$.
The following lemma is proved in the same way as
\cite[Proposition 3.2]{Po1}.
\begin{lem}
Let $N\subs M$ be an inclusion of factors.
Let $E_N^M$ be a faithful normal conditional expectation from $M$
onto $N$.
Fix a faithful normal state $\vph$ on $N$.
Then the following statements are equivalent:
\begin{enumerate}
\item
$N\subs M$ is centrally non-trivial.

\item
For any $\de,\vep>0$, any finite set $\meF\subs N_*$,
any finite family $(x_i)_{i=1}^m$ in the unit ball of $M$,
there exists a partition of unity $(q_j)_{j=1}^t$ in $N$
such that $t\leq (10^{4}\vep^{-1})^{6m\log\vep^{-1}}$ and
\begin{enumerate}
\item
$\displaystyle
\left\|\sum_{j=1}^t q_j \ps q_j-\ps
\right\|<\de$
for all $\ps\in\meF$,

\item
$\displaystyle
\bigg{\|}\sum_{j=1}^t q_j x_i q_j-E_{N\vee (N'\cap M)}^M(x_i)
\bigg{\|}_{\vph\circ E_N^M}<\vep$
for all $1\leq i\leq m$.
\end{enumerate}
\end{enumerate}
\end{lem}

\begin{lem}\label{lem: rho-cent-nontri}
Let $M$ be an infinite factor. 
Let $\rho\in\Cnd(M)_{\rm irr}$ be non-inner. 
Set the endomorphism $\si:=\id\oplus \rho$. 
Then the inclusion $\si(M)\subs M$ is 
not centrally non-trivial. 
\end{lem}
\begin{proof}
Suppose that $\si(M)\subs M$ is centrally non-trivial. 
Then the locally trivial subfactor 
$N^{(\id,\rho)}\subs M^{(\id,\rho)}$ is centrally non-trivial
because this inclusion is isomorphic to $\si(M)\subs M$. 
Since $\rho$ is non-inner, 
\[
N^{(\id,\rho)}\vee \big{(}(N^{(\id,\rho)})'\cap M^{(\id,\rho)}\big{)}
=\{x\otimes e_{00}+\rho(y)\otimes e_{11}\mid x,y\in M\}. 
\]
This shows that 
the matrix unit $1\oti e_{01}\in M^{(\id,\rho)}$ 
is orthogonal to $N^{(\id,\rho)}\vee (N^{(\id,\rho)})'\cap M^{(\id,\rho)}$. 
Recall
the embedding $\al^{(\id,\rho)}\col M\ra M^{(\id,\rho)}$
and the conditional expectation $E:=E^{(1/2,1/2)}$
as defined in \S\ref{LTS}.
Let $\nu\in\N$ and $0<\vep<1/\sqrt{2}$.
Let $\{\meF_\nu\}_{\nu=1}^\infty$ be an increasing
finite subsets of $M_*$ whose union is norm-dense.
Then by central non-triviality,
for each $\nu\in\N$,
there exists a partition of unity
$(q_j^\nu)_{j=1}^{t}$ in $M$,
with $t\leq(10^4 \vep^{-1})^{6\log\vep^{-1}}$,
such that $\|[q_j^\nu,\chi]\|<2/\nu$ for all $1\leq j\leq t$,
$\chi\in \meF_\nu$ and
\[
\vep>
\Big{\|}
\sum_{j=1}^{t}
\al^{(\id,\rho)}(q_j^\nu)(1\oti e_{01})\al^{(\id,\rho)}(q_j^\nu)
\Big{\|}_{\ps\circ(\al^{(\id,\rho)})^{-1}\circ E}
=
\frac{1}{\sqrt{2}}
\Big{\|}\sum_{j=1}^{t}
q_j^\nu \rho(q_j^\nu)\Big{\|}_{\ps\circ\ph_\rho},
\]
where $\ps$ is a fixed faithful state in $M_*$.
Since $(q_j^\nu)_\nu$ is centralizing,
$\rho(q_j^\nu)-q_j^\nu\to 0$ strongly* as $\nu\to\om$.
Then we have 
\[
\frac{1}{\sqrt{2}}>\vep\geq \lim_{\nu\to\om}
\frac{1}{\sqrt{2}}
\Big{\|}\sum_{j=1}^{t}
q_j^\nu \rho(q_j^\nu)\Big{\|}_{\ps\circ\ph_\rho}
=
\lim_{\nu\to\om}
\frac{1}{\sqrt{2}}\Big{\|}\sum_{j=1}^{t}
q_j^\nu\Big{\|}_{\ps\circ\ph_\rho}
=\frac{1}{\sqrt{2}}. 
\]
This is a contradiction. 
\end{proof}

Next we recall the following result \cite[Theorem 3.5]{Iz0}. 

\begin{thm}[Izumi]
Let $M$ be a type III$_1$ factor and $\si\in\End(M)_0$. 
Then the following conditions are equivalent: 
\begin{enumerate}
\item $\si(M)'\cap M= \tsi(\tM)'\cap \tM$. 

\item $[\si\bsi]$ does not contain an outer modular automorphism. 
\end{enumerate}
\end{thm}

Consider the following towers:
\[
\C=M'\cap M\subs \si(M)'\cap M\subs \si\bsi(M)'\cap M \subs \cdots,
\]
\[
\C=\tM'\cap \tM\subs \tsi(\tM)'\cap \tM
\subs \tsi\wdt{\bsi}(\tM)'\cap \tM\subs \cdots,
\]
We say that the graph change occurs at the stage $n$
if
the $n$-th algebras ($\C$ is the $0$th algebra) do not coincide.
The following lemma is a direct consequence of the previous theorem.

\begin{lem}\label{lem: graph-outer}
Let $M$ be a type III$_1$ factor and $\si\in\End(M)_0$. 
Then the graph change of $\si(M)\subs M$ occurs 
at the stage $n\geq0$ if and only if 
$[(\si\bsi)^n]$ contains an outer modular automorphism. 
\end{lem}

The following theorem proved by Popa 
\cite[Theorem 3.5 (ii), Corollary 4.4]{Po1} 
is our key to characterize centrally
trivial endomorphisms on the hyperfinite type III$_1$ factor. 
We note that Popa's theorem is proved for centrally free subfactors, 
but his proof still works
for centrally non-trivial subfactors without any change. 

\begin{thm}[Popa]\label{thm: Popa-cent}
Let $M$ be the hyperfinite type III$_1$ factor 
and $\si\in\End(M)_0$. 
Let $\si(M)\subs M\subs M_1\subs \cdots$ be the basic extension. 
Let $n\in\N$. 
Then the following statements are equivalent. 
\begin{enumerate}

\item The $n$-th subfactor $M\subs M_n$ is centrally non-trivial. 

\item The graph change does not occur at the stage less than or equal to $n$. 

\item 
The sector $[(\si\bsi)^n]$ does not contain an outer modular automorphism. 
\end{enumerate}
\end{thm}

\begin{cor}\label{cor: cent-endo-partial}
Let $M$ be the hyperfinite type III$_1$ factor 
and $\rho\in\Cnd(M)_{{\rm irr}}$. 
Then either the following (1) or (2) occurs:
\begin{enumerate}
\item $[\rho]=[\si_t^\vph]$ for some $t\in\R$. 

\item $[\rho\brho]$ contains an outer modular automorphism. 
\end{enumerate}
\end{cor}
\begin{proof}
If $\rho$ is inner, (1) follows. 
We consider the case that $\rho$ is not inner. 
We set $\si:=\id\oplus \rho$. 
Then $\bsi$ is also centrally trivial by Theorem \ref{thm: cend-structure}. 
By Lemma \ref{lem: rho-cent-nontri}, 
the inclusion $\bsi(M)\subs M$ is not centrally non-trivial. 
Since the inclusion $\bsi(M)\subs M$ is isomorphic to $M\subs M_1$, 
the graph change occurs 
at the 1st stage by Theorem \ref{thm: Popa-cent}. 
Hence 
$\si\bsi$ contains an outer modular automorphism 
by Lemma \ref{lem: graph-outer}. 
Using $[\si\bsi]=[\id]\oplus[\rho]\oplus [\brho]\oplus[\rho\brho]$, 
we see that one of $[\rho]$, $[\brho]$ and $[\rho\brho]$
contain an outer modular automorphism. 
\end{proof}

\noindent$\bullet$
\textit{Proof of Theorem \ref{thm: mod-cent} 
for the hyperfinite type III$_1$ factor.} 

It suffices to prove $\Cnd(M)_{\rm irr}\subs \End(M)_{\rm m}$ as before.
Let $\rho\in \Cnd(M)_{\rm irr}$. 
We show that the condition (2) 
in Corollary \ref{cor: cent-endo-partial} does not occur. 
Since $\rho$ is irreducible, 
$\si_t^\vph\prec \rho\brho$ if and only if 
$[\rho]=[\si_{t}^\vph \rho]$ in $\Sect(M)$. 
We set $H:=\{t\in\R\mid [\rho]=[\si_{t}^\vph \rho]\}$. 
Then $H$ is a subgroup of $\R$. 
Hence $t\in H$ implies $\si_{nt}^\vph\prec \rho\brho$ for all $n\in\Z$. 
Since $M$ is of type III$_1$ and $\rho\brho$ has finite index, 
this is possible only in the case $t=0$,
that is,
the condition (1) in Corollary \ref{cor: cent-endo-partial} holds.
Hence $\rho\in \End(M)_{\rm m}$.
\hfill$\Box$

\subsection{Type III$_\la$ case ($0<\la<1$)}

\begin{lem}\label{lem: rho-al-commute}
Let $M$ be a McDuff factor of type III$_\la$, $0<\la<1$, 
and $\rho\in\Cnd(M)$. 
Then for any $\al\in\Aut(M)$, $[\al\rho]=[\rho\al]$ in $\Sect(M)$. 
\end{lem}
\begin{proof}
Since $\rho$ is centrally trivial, 
there exists an isomorphism $\Ps\col M\ra M\oti R_{0,1}$ such that 
$[\rho]=[\Ps^{-1}\circ (\rho\oti\id_{R_{0,1}})\circ \Ps]$ in $\Sect(M)$ 
by Proposition \ref{prop: splitting R01}. 
Take $\th_\mu\in\Aut(R_{0,1})$ with $\mo(\th_\mu)=\mu>0$.
Define the automorphism 
$\al_\mu=\Ps^{-1}\circ(\id\oti \th_\mu)\circ \Ps\in \Aut(M)$. 
Then we have 
$[\rho\al_\mu]=[\Ps^{-1}\circ (\rho\oti\th_\mu)\circ\Ps]=[\al_\mu\rho]$. 
Note that $\mo(\al_\mu)=\mu$ holds. 

Let $\al\in\Aut(M)$. 
Then there exists $\mu>0$ such that $\mo(\al)=\mo(\al_\mu)$. 
We set $\be:=\al\al_\mu^{-1}$. 
Since $\mo(\be)=1$, $\be$ is approximately inner
\cite[Theorem 1 (1)]{KST}. 
Hence we have $[\rho\be]=[\be\rho]$ by Lemma \ref{lem: commutes}. 
Then,
\[
[\rho\al]
=
[\rho\be\al_\mu]
=
[\rho\be][\al_\mu]
=
[\be\rho][\al_\mu]
=
[\be][\rho\al_\mu]
=
[\be][\al_\mu\rho]
=
[\be\al_\mu\rho]
=[\al\rho]. 
\]
\end{proof}

Let $R$ be the hyperfinite type III$_1$ factor and $\vph$ a dominant weight. 
Then $M=R\rti_{\si_T^\vph}\Z$ is a hyperfinite type III$_\la$ factor with 
$T=-2\pi/\log\la$. 
Denote the implementing unitary of $\sigma^T_\varphi$ by $U$.
Let $\th\col \T=\R/[0,1)\ra \Aut(M)$ be the dual action of $\si_T^\vph$ 
and $\ps=\hat{\vph}$ the dual weight of $\vph$. 
Then $\si_T^\ps=\Ad U$ holds. 

\begin{lem}\label{lem: RM}
One has $R_\om\subs M_\om$ via the embedding $R\subs M$. 
\end{lem}
\begin{proof}
Let $\chi\in R_*$ be a faithful state and $\hat{\chi}$ the dual state. 
Let $E_\th\col M\ra R$ be the expectation obtained 
by averaging the $\T$-action $\th$. 
Then for any $x\in R$ and $y\in M$, we have 
$
[\hat{\chi},x](y)=[\chi,x](E_\th(y))
$
This shows that $\|[\hat{\chi},x]\|\leq \|[\chi,x]\|$. 
Let $(x^\nu)_\nu$ be an $\om$-centralizing sequence in $R$. 
Then we have $[\hat{\chi},x]\to 0$ as $\nu\to\om$. 

Since $\hchi$ is faithful, 
a set $\{\hchi\cdot (a U^n)\mid a\in R,\ n\in\Z\}$ spans 
a norm dense subset in $M_*$. 
To prove $(x^\nu)_\nu$ is an $\om$-centralizing sequence in $M$, 
it suffices to show that $[aU^n, x^\nu]\to0$ strongly* 
for any $a\in R$ and $n\in\Z$.
We have
\begin{align*}
[aU^n, x^\nu]=&\,
[a,x^\nu]U^n+a[U^n,x^\nu]\\
=&\,[a,x^\nu]U^n+a(\si_T^\vph(x^\nu)-x^\nu)U^n. 
\end{align*}
Since $(x^\nu)_\nu$ is an $\om$-centralizing sequence in $R$ 
and $\si_T^\vph\in \Cnt(R)$, 
the both terms converge to $0$ strongly* as $\nu\to\om$. 
\end{proof}

\begin{lem}\label{lem: theta-cnt} One has
$\th_p\nin\oInt(M)$ for all $p\in\T\setm\{0\}$. 
In particular, 
$\th_p \nin \Cnt(M)$ for all $p\in \T\setm \{0\}$. 
\end{lem}
\begin{proof}
It is well known that $\theta$ faithfully acts on the flow of weights
of $M$. Hence $\theta_p$ is not in $\oInt(M)$ for $p\ne 0$.
\end{proof}

Now let $\rho\in \Cnd(M)_{\rm irr}$. 
For $p\in\T$, $[\th_p\,\rho]=[\rho\,\th_p]$ holds 
from Lemma \ref{lem: rho-al-commute}. 
Hence there uniquely 	exists $\al_p\in \Int(M)$ such that 
\[
\th_p\,\rho=\al_p \, \rho\,\th_p. 
\]

Let $\{\vph_n\}_{n=1}^\infty$ be a norm dense sequence 
in the set of normal states on $M$. 
We prepare a function $\de_\rho\col \Aut(M)\times \Aut(M)\ra [0,+\infty)$ 
defined by 
\[
\de_\rho(\al,\be)
=\sum_{n=1}^\infty 
\frac{1}{2^n}
\|\vph_n\circ \ph_\rho\circ \al^{-1}-\vph_n\circ \ph_\rho\circ \be^{-1}\|. 
\]
Then $\de_\rho$ defines a metric on $\Int(M)$. 
Indeed, if $\al, \be\in\Int(M)$ satisfies $\de_\rho(\al,\be)=0$, 
then $\al\rho=\be\rho$. 
Take $w\in U(M)$ such that $\be^{-1}\al=\Ad w$. 
Then $\Ad w \circ \rho=\rho$, and $w\in \rho(M)'\cap M=\C$. 
Hence $\al=\be$. 
We call that topology the $\de_\rho$-topology. 

\begin{lem}
The map $\al\col \T\ra \Int(M)$ is continuous with respect to 
the $\de_\rho$-topology. 
\end{lem}
\begin{proof}
Note that 
$\th_p\circ \ph_\rho\circ \th_p^{-1}
=\ph_{\th_p\rho\th_p^{-1}}
=\ph_{\al_p \rho}
=\ph_\rho\circ \al_p^{-1}$. 
Then the statement is trivial because the map
$\T\ni p\mapsto\chi\circ \ph_\rho\circ \al_p^{-1}
=\chi\circ\th_p\circ \ph_\rho\circ \th_p^{-1}\in M_*$ is norm-continuous 
for any $\chi\in M_*$. 
\end{proof}

\begin{lem}\label{lem: th-w-rho}
There exists $w\in U(M)$ such that 
\[
\th_p\circ (\Ad w\circ \rho)=(\Ad w\circ \rho)\circ \th_p
\]
holds for all $p\in\T$. 
\end{lem}
\begin{proof}
Let $U(M)/U(\C)$ be the quotient Polish group. 
We have a bijective map $\ovl{\Ad}\col U(M)/U(\C)\ra \Int(M)$ 
such that $\ovl{\Ad}([u])=\Ad u$ 
for $[u]\in U(M)/U(\C)$, $u\in U(M)$. 
The map is continuous with respect to the $\de_\rho$-topology. 
By \cite[Corollary A.10]{Ta2}, 
the inverse map $\ovl{\Ad}^{-1}$ is Borel. 
Let $s\col U(M)/U(\C)\ra U(M)$ be a Borel cross section. 
We set $v_p:=s(\ovl{\Ad}^{-1}(\al_p))^* \in U(M)$. 
Then the map $v\col \T\ra U(M)$ is Borel and 
satisfies 
\[
\th_p\,\rho=\Ad v_p^*\circ \rho\,\th_p
\quad\mbox{for all}\ p\in\T. 
\]

Now we set $\mu_{p,q}:=v_p \th_p(v_q)v_{p+q}^*$ for $p,q\in \T$. 
Then for any $x\in M$, 
\begin{align*}
\mu_{p,q}\rho(x)\mu_{p,q}^*
=&\,
v_p \th_p(v_q)v_{p+q}^*\rho(x)v_{p+q}\th_p(v_q^*)v_p^*
\\
=&\,
v_p \th_p(v_q) \th_{p+q}(\rho(\th_{-p-q}(x)))\th_p(v_q^*)v_p^*
\\
=&\,
v_p \th_p(v_q \th_q(\rho(\th_{-p-q}(x)))v_q^*)v_p^*
\\
=&\,
v_p \th_p(\rho(\th_{-p}(x)))v_p^*
=\rho(x). 
\end{align*}
Hence $\mu_{p,q}\in \rho(M)'\cap M=\C$. 
It is easy to see that $\mu$ satisfies 
\[
\mu_{p,q}\mu_{p+q,r}=\mu_{q,r}\mu_{p,q+r},
\quad \mu_{p,0}=1=\mu_{0,p}. 
\]
Hence $\mu\col \T\times \T\ra \C$ is a 2-cocycle. 
By triviality of $H^2(\T,U(\C))$ 
(see \cite[Proposition 2.1]{Mo1}), 
there exists a Borel map $\la\col \T\ra U(\C)$ such that 
\[
\mu_{p,q}=\la_p \la_q \la_{p+q}^* 
\]
holds for $(p,q)\in\T\times \T$ 
in the outside of a null set $N\subs \T\times \T$ 
with respect to the Haar measure. 
We set $\ovl{v}_p:=\la_p^* v_p$ for $p\in\T$. 
Then the map $\ovl{v}\col \T\ra U(M)$ is Borel and 
satisfies $\th_p\,\rho=\Ad \ovl{v}_p^*\circ \rho\,\th_p$ 
for every $p\in\T$ and 
\[
\ovl{v}_p \th_p(\ovl{v}_p)=\ovl{v}_{p+q}
\quad \mbox{for all}\ (p,q)\in \T\times\T\setm N. 
\]

Hence $\ovl{v}$ satisfies a 1-cocycle relation almost everywhere, 
and it coincides with a 1-cocycle $\ovl{v}'$ almost everywhere 
(see \cite[Remark III.1.9]{CT} and \cite{Mo2}). 
Since $\th$ is a minimal action of the compact group $\T$
and $M^\th=R$ is purely infinite, 
$\th$ is stable (see \cite[Propostion 5.2]{Iz1}). 
Hence there exists $w\in U(M)$ such that 
$\ovl{v}'_p=w^*\th_p(w)$ holds for almost every $p\in\T$, 
and 
\[
\th_p\circ (\Ad w\circ \rho)=(\Ad w\circ \rho)\circ \th_p
\]
holds for almost every $p\in\T$. 
By continuity of $\th$, it holds for every $p\in \T$. 
\end{proof}

\noindent$\bullet$
\textit{Proof of Theorem \ref{thm: mod-cent} 
for the hyperfinite type III$_\la$ factor $(0<\lambda<1)$.}

It suffices to prove $\Cnd(M)_{\rm irr}\subs \End(M)_{\rm m}$
as before.
Let $\rho\in \Cnd(M)_{\rm irr}$.
By Lemma \ref{lem: th-w-rho}, 
we may and do assume that $\rho$ commutes with $\th$:
\[
\th_p\,\rho=\rho\,\th_p. 
\]
Hence $\si:=\rho|_R\in \End(R)$ has the meaning. 
Note that $U$ and $\rho(U)$ are normalizing $R$ 
in $M$. 
Indeed, the commutativity of $\th$ with $\rho$ yields 
$\rho(U)R \rho(U^*)\subs M^\th=R$. 

We consider the relative commutant $\si(R)'\cap R$. 
This is finite dimensional by the Pimsner-Popa inequality 
for $\rho\circ \ph_\rho$. 
Since $\rho$ is irreducible, 
\[
(\si(R)'\cap R)^{\Ad \rho(U)}=\rho(M)'\cap R=\C. 
\]
This means that the finite dimensional von Neumann algebra
$\si(R)'\cap R$ admits the ergodic $\Z$-action $\Ad \rho(U)$.
Hence $\si(R)'\cap R$ is abelian. 
Let $\dim(\si(R)'\cap R)=n$ and $p_1,\dots,p_n$ be the minimal projections of 
$\si(R)'\cap R$ such that 
\[
\si(R)'\cap R=\C p_1+\dots+\C p_n. 
\]

By Lemma \ref{lem: RM}, $R_\om\subs M_\om$ holds. 
Hence $\si\in\End(R)_0$ is centrally trivial. 
Applying Theorem \ref{thm: mod-cent}
for the hyperfinite type III$_1$ factor $R$,
we see that $\si$ is a modular endomorphism,
that is, there exists $t_1,\dots,t_n\in\R$ such that 
\[
[\si]=\bigoplus_{i=1}^n[\si_{t_i}^\vph] \in \Sect(R). 
\]
Take an isometry $w_i\in (\si_{t_i}^\vph,\si)$
with $w_i w_i^*=p_i$ for each $i$. 
Then we have 
\[
\si(x)=\sum_{i=1}^n w_i\si_{t_i}^\vph(x)w_i^*
\quad\mbox{for all}\ x\in R. 
\]
Set
$
\mu=\sum_{i=1}^n \si_T^\vph(w_i)w_i^*\in U(R)$. 
Then we have 
\[
\si_T^\vph\,\si\,\si_{-T}^\vph=\Ad \mu\circ \si. 
\]
Since $\Ad U|_R=\si_T^\vph$, we have 
\[
\Ad U\rho(U^*)\circ \si=\Ad \mu\circ \si,
\]
where we note that $U\rho(U^*)\in M^\th=R$, 
and $\mu^* U\rho(U^*)\in \si(R)'\cap R$. 
Hence there exists $\chi_1,\dots,\chi_n\in \T$ such that 
\[
U\rho(U^*)
=\mu\cdot \Big{(}\sum_{i=1}^n\chi_i p_i\Big{)}
=
\sum_{i=1}^n \chi_i \si_T^\vph(w_i)w_i^*
=
\sum_{i=1}^n \chi_i U w_i U^*w_i^*,
\]
and we have
\[
\rho(U)=\sum_{i=1}^n \ovl{\chi_i}\,w_i U w_i^*. 
\]
Taking $s_i\in \T$ with $\th_{s_i}(U)=\ovl{\chi_i}\,U$, we have
\[
\rho(U)=\sum_{i=1}^n w_i \th_{s_i}(U) w_i^*. 
\]
Therefore $\rho$ has the following decomposition:
\[
\rho(x)=\sum_{i=1}^n w_i \si_{t_i}^\ps(\th_{s_i}(x)) w_i^*
\quad\mbox{for all}\ x\in M. 
\]
Since $\rho$ is irreducible, 
we have $n=1$ and
$[\rho]=[\si_{t_1}^\ps\th_{s_1}]$.
This shows $\th_{s_1}$ must be centrally trivial.
By Lemma \ref{lem: theta-cnt}, $s_1=0$,
and we have $[\rho]=[\si_{t_1}^\ps]$.
Hence $\rho$ is a modular automorphism.
$\hfill\Box$

\subsection{Type III$_0$ case}

\subsubsection{Reduction of the problem to the study of type II inclusions}

Let $M=N\rti_\th \Z$ be the discrete decomposition of a type III$_0$ factor 
$M$ with the implementing unitary $U$. 
Let $\ta$ be a trace on $N$ and $\vph=\hat{\ta}$ be the dual weight on $M$. 
Then it is known that the weight $\vph$ is lacunary. 
Let $\hth$ be the dual action of the torus $\T$ on $M$. 
By \cite[Theorem 3.1, Corollary 4.6]{HS-Pt1},
$\hth_p$ is approximately inner for $p\in\T$.

\begin{lem}\label{lem: normalize-rho}
Let $M=N\rti_\th \Z$ be the discrete decomposition as before. 
Let $\rho\in \Cnd(M)_{\rm irr}$. 
Then there exists a unitary $w\in M$ such that 
$\Ad w \rho$ and $\hth$ commute and $\Ad w \rho(U)=U$. 
\end{lem}
\begin{proof}
Since $\hth_p$ is approximately inner, 
$\rho$ and $\hth_p$ commute in $\Sect(M)$ by Lemma \ref{lem: commutes}.
Through the $\T$-valued 2-cohomology vanishing
as in type III$_\la$ case,
we can take a $\hth$-cocycle $v$ such that 
\[
\hth_p\rho\hth_p^{-1}=\Ad v_p^* \rho 
\quad\mbox{for all}\ p\in \T. 
\]

We show that $v$ is a $\hth$-coboundary by using the 2-by-2 matrix argument.
Set $P:=M_2(\C)\oti M$, the $\T$-action $\al:=\id\oti \hth$
and the $\al$-cocycle $w:=e_{11}\oti 1+e_{22}\oti v$.
Let $\be=\Ad v \al$ be the perturbed action.
We will show $p:=e_{11}\oti 1$ and $q:=e_{22}\oti1$
are equivalent in $P^\be$.

First we show that $p$ and $q$ are properly infinite projections in $P^\be$.
Since $p P^\be p=\C e_{11}\oti M^\hth$, $p$ is properly infinite.
For $q$, we have $q P^\be q=\C e_{22}\oti M^{\Ad v \hth}$.
Since $M^{\Ad v \hth}\supset \rho(M^\hth)$, $q$ is properly infinite.

Second we show that the central supports of $p$ and $q$ are equal to $1$.
We set a unitary $V:=e_{11}\oti U+e_{22}\oti \rho(U)$.
The equality $\be_p(V)=e^{-ip}V$ shows that $\be$ is a dual action.
Hence $P$ is naturally isomorphic to $P^\be\rti_{\Ad V}\Z$.
We regard $P=P^\be\rti_{\Ad V}\Z$.
On one hand, $P\rti_\be \T$ is isomorphic to $P^\be\oti B(\el^2(\Z))$ 
by Takesaki duality. 
Since $p,q\in P^\be$, $\pi_\be(p)$ and $\pi_\be(q)$ 
are mapped to $p\oti1$ to $q\oti1$. 
Hence it suffices to show that the central supports of
$\pi_\be(p)$ and $\pi_\be(q)$ are equal to $1$
in $P\rti_\be \T$. 
On the other hand, $P\rti_\be\T$ is naturally isomorphic to 
$P\rti_\al \T=M_2(\C)\oti (M\rti_\hth\T)$. 
Since $p,q\in P^\al\cap P^\be$, 
$\pi_\be(p)$ and $\pi_\be(q)$ 
are mapped to $\pi_\al(p)=p\oti1=e_{11}\oti1\oti1$ 
and $\pi_\al(q)=q\oti1=e_{22}\oti1\oti1$. 
Hence the central supports of $\pi_\be(p)$
and $\pi_\be(q)$ in $P\rti_\be\T$
are equal to $1$. 

Therefore $p$ and $q$ are equivalent in $P^\be$,
and $v$ is a $\hth$-coboundary. 
Take a unitary $w\in M$ such that $v_p=w^*\th_p(w)$. 
Then $\hth_p \circ(\Ad w \rho)=(\Ad w \rho)\circ\hth_p$. 
Hence we may assume that $\rho$ and $\hth_p$ commute. 
Then $\rho(U)U^*$ is contained in $M^\hth=N$. 
Since $\th$ is stable \cite[Theorem III.5.1 (i)]{CT},
there exists $w_1\in U(N)$ such that
$\rho(U)U^*=w_1^*\th(w_1)$.
Then $\Ad w_1 \rho$ satisfies the desired properties.
\end{proof}

\begin{lem}\label{lem: inner-modular}
Let $\rho$ be an irreducible endomorphism
with finite index on $M$.
Assume that $\ph_\rho$ and $\hth$ commute and $\rho(U)=U$.
If an endomorphism $\rho|_N\th^n$ on $N$ is inner
for some $n\in\Z$, then $\rho$ is a modular endomorphism.
\end{lem}
\begin{proof}
The proof is similar to that of \cite[Theorem 5.2]{HS-Pt1}. 
We may and do assume that $\rho|_N$ is inner by 
perturbing $\rho$ to $\rho\Ad U^n$. 
Let $\meH\subs N$ be a Hilbert space such that $\rho|_N=\rho_\meH$. 
Since $\rho$ has finite index, $\meH$ is finite dimensional. 
Put $d:=\dim(\meH)$. 
Let $(V_i)_{i=1}^d$ be an orthonormal base of $\meH$. 
Note that $\vph\rho=d\vph$ holds. 
Indeed for $x\in M_+$, we have 
\begin{align*}
\vph(\rho(x))
=&\,
\ta\left( \int_{\T}\hth_p(\rho(x))\right)
=
\ta\left(\rho\left(\int_{\T}\hth_p(x)\right)\right)
\\
=&\,
\sum_{i=1}^d
\ta\left(V_i\left(\int_{\T}\hth_p(x)\right)V_i^*\right)
=
\sum_{i=1}^d
\ta\left(V_i^* V_i\left(\int_{\T}\hth_p(x)\right)\right)
\\
=&\,
d \ta\left(\int_{\T}\hth_p(x)\right)
=d\vph(x). 
\end{align*}

Next we will show $\ta\circ\ph_\rho=d^{-1}\ta$ on $N$. 
Let $h$ be the positive operator affiliated with $N$ such that 
$\ta\circ\ph_\rho=\ta_h$. 
Since $\ph_\rho\rho|_N=\id_N$, $h$ is indeed contained 
in $\rho(N)'\cap N$. 
Moreover we have 
\begin{align*}
\th(h^{it})
=&\,
\th([D\ta\circ\ph_\rho:D\ta]_t)
=
[D\ta\circ\ph_\rho\circ\th^{-1}:D\ta\th^{-1}]_t
=
[D\ta\th^{-1}\circ\ph_\rho:D\ta\th^{-1}]_t
\\
=&\,
[D\ta\th^{-1}\circ\ph_\rho:D\ta\circ\ph_\rho]_t
[D\ta\circ\ph_\rho:D\ta]_t
[D\ta:D\ta\th^{-1}]_t
\\
=&\,
\rho([D\ta\th^{-1}:D\ta]_t)
[D\ta\circ\ph_\rho:D\ta]_t
[D\ta:D\ta\th^{-1}]_t
\\
=&\,
[D\ta\th^{-1}:D\ta]_t
[D\ta\circ\ph_\rho:D\ta]_t
[D\ta:D\ta\th^{-1}]_t
\\
=&\,
[D\ta\circ\ph_\rho:D\ta]_t
[D\ta\th^{-1}:D\ta]_t
[D\ta:D\ta\th^{-1}]_t
\\
=&\,
[D\ta\ph_\rho:D\ta]_t=h^{it}. 
\end{align*}
Hence $h$ is affiliated with $(\rho(N)'\cap N)^\th=\rho(M)'\cap N=\C$, 
and $h$ is a positive scalar. 
Using $\ta\rho=d\ta$ and $\ta\circ\ph_\rho=h \ta$, 
we have $h=d^{-1}$.
Since $\hth$ and $\ph_\rho$ commute,
we have $\vph\circ\ph_\rho=d^{-1}\vph$.

Define
$\ph_1=d^{-1}\sum_{i=1}^d V_i^* xV_i$ for $x\in N$.
Then $\ph_1$ is a left inverse for $\rho|_N$ and satisfies
$\ta\circ\ph_1=d^{-1}\ta$. 
The equality $\ta\circ\ph_\rho|_N=\ta\circ\ph_1$
implies that $\ph_\rho|_N=\ph_1$.

Now we show that $d(\rho)=d$.
From the Pimsner-Popa inequality,
the map $\rho\circ \ph_\rho|_N-d(\rho)^{-2}\id_N$ is
a completely positive on $N$.
Set a projection 
$p:=\sum_{i,j=1}^d d^{-1} V_i V_i V_j^* V_j^* \in N$.
Then $\rho(\ph_\rho(p))=\rho_{\meH}(\ph_1(p))=d^{-2}$.
Hence $d(\rho)\geq d$.

Note that $\rho(\ph_\rho(U))=U$.
Regarding $M=N\rti_\th \Z \subs N\oti B(\el^2(\Z))$,
we can see that $(\rho_\meH\circ\ph_1\oti\id)|_{M}$
is a conditional expectation from $M$ onto $\rho(M)$
whose index is equal to $d^2$.
By the minimality of $E_\rho$, $d^2\geq \Ind(E_\rho)=d(\rho)^2$.
Thus $d(\rho)=d$.

The equality $\rho(U)=U$ is equivalent to 
$\meH^* \th(\meH)\subs Z(N)$. 
We set $c_{i,j}:=V_i^* \th(V_j)$ for $1\leq i,j\leq d$. 
Then $c:=(c_{i,j})_{i,j}$ is unitary 
in $M_d(\C)\oti Z(N)$. 

Consider the canonical core $\tM=M\rti_\vph \R\subs M\oti B(L^2(\R))$ 
and the canonical extension $\trho$.
Note that $Z(N)\oti L(\R)\subs \tM$.
Let $\be=\pi_{\si^\vph}(U)\in \Aut(\tM)$. 
Then it is known that the covariant system $\{Z(N)\oti L(\R),\be\}$
has a fundamental domain, that is, there exists a projection 
$e\in Z(N)\oti L(\R)$ such that
$
\sum_{n\in\Z}\be^n(e)=1$. 
In particular, $\{M_d(\C)\oti Z(N)\oti L(\R),\id\oti \be\}$ is stable. 
Hence there exists a unitary $\nu\in M_d(\C)\oti Z(N)\oti L(\R)$ 
such that $c\oti1=\nu(\id\oti\be)(\nu^*)$. 
Set
$
W_i:=\sum_{j=1}^d  \pi_{\si^\vph}(V_j)\nu_{ji}$. 
It is easy to see that a family $(W_i)_{i=1}^d$ 
span a Hilbert space in $\tM$. 
We show that $\trho$ is implemented by $(W_i)_{i=1}^d$. 
Take $x\in N$. 
Since $\nu_{ij}\in Z(N)\oti L(\R)$ for all $i,j$
and $Z(N)\oti L(\R)\subs \pi_{\si^\vph}(N)'$, 
we have 
\[
W_i \pi_{\si^\vph}(x)=\pi_{\si^\vph}(\rho(x))W_i=\trho(\pi_{\si^\vph}(x))W_i.
\]
Since $\vph\circ\ph_\rho=d^{-1}\vph=d(\rho)^{-1}\vph$,
we have $\trho(\la^\vph(t))=\la^\vph(t)$.
Then it is trivial that 
\[
W_i\la^\vph(t)=\la^\vph(t)W_i=\trho(\la^\vph(t))W_i
\quad\mbox{for all}\ t\in\R. 
\]
Finally we have 
\begin{align*}
\sum_{i=1}^d W_i \pi_{\si^\vph}(U)W_i^*
=&\,
\sum_{i,j,k=1}^d  
\pi_{\si^\vph}(V_j)\nu_{ji} \pi_{\si^\vph}(U) 
\nu_{ki}^* \pi_{\si^\vph}(V_k^*)
\\
=&\,
\sum_{i,j,k=1}^d  
\pi_{\si^\vph}(V_j)\nu_{ji}  
\be(\nu_{ki}^* )\pi_{\si^\vph}(U)\pi_{\si^\vph}(V_k^*)
\\
=&\,
\sum_{j,k=1}^d  
\pi_{\si^\vph}(V_j)
\,(\nu (\id\oti\be)(\nu^*))_{jk}\,\pi_{\si^\vph}(\th(V_k^*) U)
\\
=&\,
\sum_{j,k=1}^d  
\pi_{\si^\vph}(V_j)
\,(c\oti1)_{jk}\,\pi_{\si^\vph}(\th(V_k^*) U)
\\
=&\,
\sum_{j,k}^d  
\pi_{\si^\vph}(V_j)
(c_{jk}\oti1)\pi_{\si^\vph}(\th(V_k^*) U)
\\
=&\,
\sum_{j,k}^d  
\pi_{\si^\vph}(V_j c_{jk} \th(V_k^*)U)
=
\sum_{j,k}^d  
\pi_{\si^\vph}(V_j V_j^*  \th(V_k V_k^*) U)
\\
=&\,
\pi_{\si^\vph}(U)
=\trho(\pi_{\si^\vph}(U)). 
\end{align*}
Since the core $\tM$ is generated by $\pi_{\si^\vph}(N)$, 
$\pi_{\si^\vph}(U)$ and $\la^\vph(t)=1\oti \la(t)$, $t\in\R$,
we see that $\trho$ is implemented by $(W_i)_{i=1}^d$. 
\end{proof}

\subsubsection{Central non-triviality of hyperfinite
type II inclusions}

\begin{lem}
Let $N\subs M$ be an inclusion of
von Neumann algebras with a faithful normal conditional expectation 
$E_N^M$.
Assume that $N$ is hyperfinite and of type II.
Then the inclusion is centrally non-trivial in the following sense: 
For any $\de,\vep>0$, 
any faithful state $\vph$ on $N$,
any finite subset $\meF\subs N_*$ and
any finite family $(x_i)_{i=1}^m$ in the unit ball of $M$
with $E_{N\vee (N'\cap M)}^M(x_i)=0$,
there exists a partition of unity $(q_j)_{j=1}^{t}$in $N$ 
such that
\begin{enumerate}
\item
the partition number $t$ does not depend on $\de$ and $\meF$.

\item
$\displaystyle
\left\| \sum_{j=1}^{t}
q_j x_i q_j\right\|_{\vph\circ E_N^M}<\vep$
for all $1\leq i\leq m$, 

\item 
$\displaystyle\left\|
\sum_{j=1}^{t}
q_j \chi q_j-\chi\right\|<\de$ 
for all $\chi\in \meF$. 
\end{enumerate}
\end{lem}
\begin{proof}
This can be similarly proved as in the case of subfactors \cite{Po1}. 
We may and do assume that $N$ is of type II$_1$
(see \cite[Proposition 3.4 (i)]{Po1}).
Let $\ta$ be a faithful tracial state on $N$.
Let $\meC_N$ be the center of $N$.
Since $N$ is hyperfinite,
we can regard $N=\meC_N\oti R_0$, 
where $R_0$ denotes the hyperfinite type II$_1$ factor. 
Let $R_n$ be a type I$_{2^n}$ subfactor of $R_0$ 
such that $\{R_n\}_{n=1}^\infty$ is an increasing sequence and 
$\cup_{n\geq1} R_n$ is weakly dense in $R_0$. 

We set $\mN_n:=R_n'\cap N$ and $\mM_n:=R_n'\cap M$. 
The inclusion $N\subs M$ is isomorphic to 
$\mN_n\oti R_n\subs \mM_n\oti R_n$,
and the state $\ps$ is $\ps|_{\mM_n}\oti \ta_{R_n}$. 
Using this identification, we have 
\[
N\vee(N'\cap M)=(\mN_n\vee(\mN_n'\cap \mM_n))\oti R_n
=\mN_n\vee(\mN_n'\cap M). 
\]
In particular, 
this implies 
$E_{N\vee(N'\cap M)}^M=E_{\mN_n\vee(\mN_n'\cap M)}^M$. 

Let $\meF:=\{\chi_i\}_{i=1}^k \subs N_*$ be a finite set. 
Let $\de>0$. 
Since the union of an increasing sequence 
$\{\meC_N\oti R_n\}_{n=1}^\infty$ is strongly dense in $N$, 
the set $\{\ta a\mid a\in \meC_N\oti R_n, n\geq1\}$ is dense in $N_*$. 
Hence there exists $p_0\in\N$ and $a_i\in \meC_N\oti R_{p_0}$, 
$1\leq i\leq k$ 
such that $\|\chi_i-\ta a_i\|<\de/2$ for $1\leq i\leq k$. 

Let $\vep>0$. 
By spectral analysis, 
there exists a positive invertible $h$ in 
$N$ such that $\|\vph-\ta h\|<\vep^2/4$ and $\ta(h)=1$. 
By taking enough large $p_1\in\N$, there exists a positive invertible 
$h_1\in \meC_N\oti R_{p_1}$ 
such that $\|\ta h-\ta h_1\|<\vep^2/4$ and $\ta(h_1)=1$. 
Then $\|\vph-\ta h_1\|<\vep^2/2$. 
We set $\ps:=\ta h_1\in N_*$. 
Putting $p:=\max(p_0,p_1)$, we have $h, a_i\in \meC_N\oti R_p$. 

Next let $(x_i)_{i=1}^m$ be given as in the statement.
Then $(x_i)_{i=1}^m$ are orthogonal to $\mN_p\vee(\mN_p'\cap M)$.
Applying Popa's local quantization to
$\mN_p\subs M_{\ps\circ E_N^M}$ \cite[Theorem A.1.2]{Po1},
we have a partition of unity $(q_j)_{j=1}^t$ in $\mN_p$ such that
\[
\left\|\sum_{j=1}^t q_j x_i q_j \right\|_{\ps\circ E_N^M}
<\frac{\vep}{\sqrt{2}} 
\quad\mbox{for all}\ 1\leq i\leq m,
\]
where
$t\leq (40\sqrt[4]{2}\vep^{-1/2})^{(m\log (2\vep^{-2})/\log\frac{4}{3})}$.
We check that $(q_j)_{j=1}^t$ has the desired property. 
Since
$
\left\|\sum_{j=1}^t q_j x_i q_j\right\|\leq \|x_i\|\leq1$, 
we have 
\begin{align*}
\left\|\sum_{j=1}^t q_j x_i q_j \right\|_{\vph\circ E_N^M}^2
=&\,
\vph\circ E_N^M\left(\left|\sum_{j=1}^t q_j x_i q_j\right|^2\right)
\\
=&\,
(\vph\circ E_N^M-\ps\circ E_N^M)
\left(\left|\sum_{j=1}^t q_j x_i q_j\right|^2\right)
+
\left\|\sum_{j=1}^t q_j x_i q_j \right\|_{\ps\circ E_N^M}^2
\\
\leq&\,
\|\vph\circ E_N^M-\ps\circ E_N^M\|
+\vep^2/2
\\
=&\,
\|\vph-\ps\|
+\vep^2/2
<\vep^2/2+\vep^2/2=\vep^2. 
\end{align*}
Hence we have 
\[
\left\|\sum_{j=1}^t q_j x_i q_j \right\|_{\vph\circ E_N^M}<\vep. 
\]
Moreover, since $q_j\in \mN_p$ and $a_i\in \meC_N\oti R_p$ commute, 
we have
\[
\|[\chi_i,q_j]\|=\|[\chi_i-\ta a_i,q_j]\|\leq 2\|\chi_i-\ta a_i\|<\de. 
\]
\end{proof}

The previous lemma immediately implies the following lemma. 

\begin{lem}\label{lem: cent-nont-II}
Let $N\subs M$ be an inclusion of
von Neumann algebras with a faithful normal conditional expectation 
$E_N^M$.
Assume that $N$ is hyperfinite and of type II.
Then for any $\vep>0$, 
any faithful state $\vph$ on $N$ and 
any finite finite family $(x_i)_{i=1}^m$ in the unit ball of $M$ 
which satisfies $E_{N\vee (N'\cap M)}^M(x_i)=0$ and $x_i\neq0$,
there exist $t\in\N$ 
and a partition of unity $(q_j)_{j=1}^t$in $N_\om$ such that
\[
\left\| \sum_{j=1}^t q_j x_i q_j\right\|_{(\vph\circ E_N^M)^\om}
<\vep\|x_i\|_{\vph\circ E_N^M} 
\quad\mbox{for all}\ 1\leq i\leq m.
\]
\end{lem}
\begin{proof}
In the previous lemma, we take 
$(\vep/2) \min\{\|x_i\|_{\vph\circ E_N^M}\mid 1\leq i\leq m\}$ 
for $\vep$. 
Note that we can take the partition number $t$ 
which does not depend on $\de$ and $\meF$. 
Let $\{\meF_\nu\}_{\nu=1}^\infty$ be an increasing sequence of 
finite sets in $N_*$ whose union is dense in $N_*$. 
Letting $\de=1/\nu$ and $\meF=\meF_\nu$ for $\nu\in\N$ in the previous 
lemma, we obtain a corresponding partition of unity $(q_j^\nu)_{j=1}^t$. 
Then it is trivial that the sequence $(q_j^\nu)_\nu$ is centralizing, 
and $q_j:=\pi_\om((q_j^\nu)_\nu)$, $1\leq j\leq t$ is a desired partition 
of unity. 
\end{proof}

\subsubsection{Endomorphisms on type II von Neumann algebras}

Let $N$ be a hyperfinite type II von Neumann algebra 
with a centrally ergodic automorphism $\th$.
We denote by $\meC$ the center of $N$. 
We consider a subset $\mE_\th\subs \End(N)$ which consists of 
endomorphisms with left inverses commuting $\th$. 
For each $\rho\in \mE_\th$, 
we choose a left inverse $\ph_\rho$ on $N$ such that 
$\ph_\rho \th=\th\ph_\rho$. 
Note that this equality implies $\rho\th=\th\rho$. 
For endomorphisms $\rho,\si$ on $N$, we write $(\rho,\si)_N$ for 
$(\rho,\si)$ to specify $N$. 

\begin{lem}\label{lem: partition}
Assume that $\rho\in\mE_\th$ satisfies $(\rho,\th^n)_N=0$ for any $n\in\Z$. 
Then for any $\vep>0$, $m\in\N$ and a faithful state $\ps\in N_*$, 
there exists $t\in \N$ and a partition of unity $(q_i)_{i=0}^t$ in $N_\om$ 
such that 
\begin{enumerate}

\item 
$\displaystyle\sum_{|n|\leq m}
|q_i \rho^\om(\th_\om^n(q_i))|_{(\ps\circ\ph_\rho)^\om}
<\vep|q_i|_{(\ps\circ\ph_\rho)^\om}$ 
for all $1\leq i\leq t$, 

\item $|q_0|_{(\ps\circ\ph_\rho)^\om}<\vep$. 
\end{enumerate}
\end{lem}
\begin{proof}
Set $\vph:=\ps\circ\ph_\rho\in N_*$ and $\meM:=M_{2m+2}(\C)\oti M$. 
Let $\{e_{ij}\mid -m\leq i,j\leq m+1\}$ be
a system of matrix units of $M_{2m+2}(\C)$.
We introduce a homomorphism $\pi\col N\ra \meM$ defined by
\[
\pi(x)
=\sum_{i=-m}^{-1} e_{ii}\oti \rho(\th^{i}(x))
+e_{00}\oti x
+\sum_{i=1}^{m+1} e_{ii}\oti \rho(\th^{i-1}(x))
\]
for $x\in N$. Put $\meN:=\pi(N)$. 
We use a conditional expectation $E_{\meN}^\meM$ defined by 
\[
E_{\meN}^\meM((e_{ij}\oti x_{ij})_{ij})
=
\frac{1}{2m+2}
\,\pi
\left(
\sum_{i=-m}^{-1} \ph_\rho(\th^{-i}(x_{ii}))
+
x_{00}
+
\sum_{i=1}^{m+1} \ph_\rho(\th^{-i+1}(x_{ii}))
\right).
\]
We set $\tvph:=\vph\pi^{-1}E_{\meN}^\meM$, and we have
\[
\tvph((e_{ij}\oti x_{ij})_{ij})
=
\frac{1}{2m+2}
\,\vph
\left(\sum_{i=-m}^{-1} \ph_\rho(\th^{-i}(x_{ii}))
+
x_{00}
+
\sum_{i=1}^{m+1} \ph_\rho(\th^{-i+1}(x_{ii}))
\right).
\]

Since $(\rho,\th^n)_N=0$ for any $n\in\Z$ by assumption, 
we have 
\[
\meN'\cap \meM
=
\sum_{i=-m}^{-1}\C e_{ii}\oti (\rho\th^i,\rho\th^i)_N
+
\C e_{00}\oti (\id,\id)_N
+
\sum_{i=1}^{m+1} \C e_{ii}\oti (\rho\th^{i-1},\rho\th^{i-1})_N. 
\] 
Hence 
$\{e_{ii}\}_{i=-m}^{m+1}\subs\meN\vee (\meN'\cap \meM)\subs (\{e_{ii}\}_{i=-m}^{m+1})'\cap \meM$. 
In particular, $e_{ij}$, $i\neq j$ is orthogonal to 
$\meN\vee (\meN'\cap \meM)$. 

Let $\vep>0$. 
By applying Lemma \ref{lem: cent-nont-II} to $\meN\subs \meM$, 
we obtain a partition of unity 
$(q_r)_{r=1}^t$ of $N_\om$ such that for all $i$ with $i\neq 0$, 
\[
\left\| 
\sum_{r=1}^t \pi^\om(q_r)(e_{i0}\oti1)\pi^\om(q_r)\right\|_{\tvph^\om}
<\frac{\vep}{2m+1}
\|e_{i0}\oti1\|_{\tvph^\om},
\] 
where $\pi^\om\col N^\om\ra \meM^\om$ is the natural extension 
of $\pi\col N\ra \meM$. 
By direct computation, we obtain 
\[
\left\|\sum_{r=1}^t \pi^\om(q_r)(e_{i0}\oti1)\pi^\om(q_r)\right\|_{\tvph^\om}^2
\hspace{-5pt}=
\begin{cases}
\displaystyle
\frac{1}{2m+2}\sum_{r=1}^t
\|\rho^\om(\th_\om^{i}(q_r))q_r \|_{\vph^\om}^2 
\ \mbox{if}\ -m\leq i\leq-1,\\
\displaystyle
\frac{1}{2m+2}\sum_{r=1}^t
\|\rho^\om(\th_\om^{i-1}(q_r))q_r \|_{\vph^\om}^2 
\ \mbox{if}\ 1\leq i\leq m\!+\!1.\\
\end{cases}
\]
Since 
$\rho^\om(\th_\om^{i-1}(q_r)), q_r \in (N^\om)_{\vph^\om}$,
we have
$\|\rho^\om(\th_\om^{i}(q_r))q_r \|_{\vph^\om}
=
\|q_r\rho^\om(\th_\om^{i}(q_r)) \|_{\vph^\om}
$.
Hence we have 
\[
\sum_{r=1}^t\|q_r \rho^\om(\th_\om^{i}(q_r))\|_{\vph^\om}^2
<\frac{\vep^2}{(2m+1)^2}
\quad\mbox{for all}\ |i|\leq m. 
\]
Summing up with $i$, we obtain
\begin{equation}\label{ineq: sum-q-rho}
\sum_{r=1}^t
\left(\sum_{|i|\leq m}\|q_r \rho^\om(\th_\om^{i}(q_r))\|_{\vph^\om}^2
\right)
<\frac{\vep^2}{2m+1}\sum_{r=1}^t\|q_r\|_{\vph^\om}^2. 
\end{equation}
We set an index subset in $\N$
\[
I_0:=\Big{\{}
r\in \Z\mid 1\leq r \leq t,\ 
\sum_{|i|\leq m}
\|q_r \rho^\om(\th_\om^{i}(q_r))\|_{\vph^\om}
< \frac{\vep}{(2m+1)^{1/2}} \|q_r\|_{\vph^\om}
\Big{\}}. 
\]
We set
$q_0:=1-\sum_{r\in I_0}q_r$. 
We check that the family $\{q_0\}\cup \{q_r\}_{r\in I_0}$ is 
a desired one. 
On the size of $q_0$, we have 
\begin{align*}
|q_0|_{\vph^\om}
=&\,
\sum_{r\in I\setminus I_0}|q_r|_{\vph^\om}
<
\sum_{r\in I\setminus I_0}
\frac{(2m+1)^{1/2}}{\vep}
\sum_{|i|\leq m}
\|q_r \rho^\om(\th_\om^{i}(q_r))\|_{\vph^\om}
\\
\leq&\,
\frac{(2m+1)^{1/2}}{\vep}
\sum_{r=1}^t
\sum_{|i|\leq m}
\|q_r \rho^\om(\th_\om^{i}(q_r))\|_{\vph^\om}
\\
<&\,
\frac{(2m+1)^{1/2}}{\vep}
\frac{\vep^2}{2m+1}\sum_{r=1}^t\|q_r\|_{\vph^\om}^2
\qquad(\mbox{by}\ (\ref{ineq: sum-q-rho}))
\\
=&\,
\frac{\vep}{(2m+1)^{1/2}}<\vep. 
\end{align*}
Note that $N_\om$ and $\rho^\om(N_\om)$ are contained 
in $(N^\om)_{\vph^\om}$. 
For any $x,y\in (N^\om)_{\vph^\om}$, the inequality 
$|xy|_{\vph^\om}\leq\|x\|_{\vph^\om}\|y\|_{\vph^\om}$ holds. 
Using this inequality, we have, for $r\in I_0$, 
\begin{align*}
\sum_{|i|\leq m}
|q_r\rho^\om(\th_\om^{i}(q_r))|_{\vph^\om}
\leq &\,
\sum_{|i|\leq m}
\|q_r\|_{\vph^\om}\|q_r\rho^\om(\th_\om^{i}(q_r))\|_{\vph^\om}
\\
\leq&\,
\left(\sum_{|i|\leq m}\|q_r\|_{\vph^\om}^2 \right)^{1/2}
\left(\sum_{|i|\leq m}
\|q_r\rho^\om(\th_\om^{i}(q_r))\|_{\vph^\om}^2\right)^{1/2}
\\
<&\,
(2m+1)^{1/2}\|q_r\|_{\vph^\om}
\cdot \frac{\vep}{(2m+1)^{1/2}}\|q_r\|_{\vph^\om}
\\
=&\,
\vep|q_r|_{\vph^\om}. 
\end{align*}
\end{proof}

For an endomorphism trivially acting on $N_\om^{\th_\om}$, 
we have the following. 

\begin{lem}
Assume that $\rho\in\mE_\th$ acts on $N_\om^{\th_\om}$ trivially. 
Then for any $\vep>0$ and any faithful state $\vph\in N_*$, 
there exists $\de>0$ and a finite set of states $\meS\subs N_*$ 
such that 
if $u\in U(N)$ satisfies $\|[u,\chi]\|<\de$ for all $\chi\in\meS$ 
and $|\th(u)-u|_\vph<\de$, 
then we have $\|\rho(u)-u\|_\vph<\vep$. 
\end{lem}
\begin{proof}
Take an increasing sequence of finite subsets 
$\{\meF_\nu\}_{\nu=1}^\infty$ in $N_*$ 
so that $\cup_{\nu=1}^\infty\meF_\nu\subs N_*$ is dense. 
Suppose that the statement is false. 
Then there exist $\vep_0>0$ and a faithful state $\vph\in N_*$
such that 
for any $\nu\in\N$, we can take $u^\nu\in U(N)$ such that 
$\|[u^\nu,\chi]\|<1/\nu$ for all $\chi\in\meF_\nu$, 
$|\th(u^\nu)-u^\nu|_{\vph}<1/\nu$ 
and $\|\rho(u^\nu)-u^\nu\|_\vph\geq \vep_0$ hold. 
Then the sequence $(u^\nu)_\nu$ is a centralizing sequence in $N$,
and set $u:=\pi_\om((u^\nu)_\nu)\in N_\om$. 
Since we have $|\th(u^\nu)-u^\nu|_{\vph}\to0$ as $\nu\to\om$ 
and $u$ is centralizing, 
we easily see that $\th(u^\nu)-u^\nu\to0$ strongly* as $\nu\to\om$. 
Hence $\th_\om(u)=u$. 
However we have $\|\rho^\om(u)-u\|_{\vph^\om}\geq \vep_0$, 
and this shows that $\rho^\om$ is not trivial on $N_\om^{\th_\om}$. 
This a contradiction. 
\end{proof}

The following lemma is directly proved from the previous lemma. 

\begin{lem}\label{lem: cent-theta}
Assume that $\rho\in\mE_\th$ acts on $N_\om^{\th_\om}$ trivially. 
Then for any $\vep>0$ and any faithful state $\vph\in N_*$, 
there exists $\de>0$ such that 
if $u\in U(N_\om)$ satisfies $|\th_\om(u)-u|_{\vph^\om}<\de$, 
then $\|\rho^\om(u)-u\|_{\vph^\om}<\vep$. 
\end{lem}

The following lemma plays an important role for our work
in type III$_0$ case.

\begin{lem}\label{lem: rho-th-n}
If $\rho\in\mE_\th$ acts on $N_\om^{\th_\om}$ trivially, 
then $(\rho, \th^n)_N\neq0$ for some $n\in\Z$. 
\end{lem}
\begin{proof}
Assume that $(\rho,\th^n)_N=0$ for all $n\in\Z$.
We will derive a contradiction.

\noindent\textbf{Step I}. 
We prepare $\de,\vep,\vep_1>0$, $m\in\N$ and the states
$\ps, \wdt{\ps},\vph$ and $\tvph$ on $N$.

Take $0<\vep<1$ and a faithful state $\ps\in N_*$. 
Set $\vph:=\ps\circ\ph_\rho$. 
For $\vep$ and the state $\vph$, we can take $\de>0$ 
as in Lemma \ref{lem: cent-theta}. 
Take $m\in\N$ enough large to satisfy $\sqrt{3/m}<\de/2$ 
and $1/2<1-1/m-\sqrt{3/m}<1$.
Take $\vep_1>0$ such that 
\begin{equation}\label{ineq: 1/2-m}
1/2<1-1/m-\sqrt{3/m}-2m\vep_1<1. 
\end{equation}

Define the following faithful states on $N$
\[
\wdt{\ps}:=\frac{1}{m} \sum_{j=0}^{m-1} \ps \th^j,
\quad 
\tvph:=\wdt{\ps}\circ\ph_\rho. 
\] 
Since $\th$ and $\ph_\rho$ commute, we trivially have 
\[
\wdt{\vph}=\frac{1}{m} \sum_{j=0}^{m-1} \vph \th^j.
\]
Note that $N_\om$ and $\rho^\om(N_\om)$ are contained
in $(N^\om)_{\tvph^\om}$. 

\noindent\textbf{Step II}. 
We take a ``Rohlin partition'' $(q_i)_{i=0}^t$ in $N_\om$
for $\{\rho\th_\om^n\}_{n=-m}^m\subs \End(N^\om)$. 

For $\vep_1$, $m$ and $\tvph$, 
we apply Lemma \ref{lem: partition}. 
Then there exists $t\in \N$ 
and a partition of unity $(q_i)_{i=0}^t$ in $N_\om$
such that
\begin{align}
&\sum_{|n|\leq m}
|q_i \rho^\om(\th_\om^n(q_i))|_{\tvph^\om}<\vep_1|q_i|_{\tvph^\om} 
\quad \mbox{for all}\ 1\leq i\leq t, 
\label{ineq: q-rho-th}
\\
&|q_0|_{\tvph^\om}<\vep_1. 
\label{ineq: q0-vep1}
\end{align}

\noindent\textbf{Step III}. 
We average each $q_i$ and take $\wdt{q_i}$
so that this is almost invariant under $\th_\om$
and almost orthogonal to $\rho^\om(\wdt{q_i})$.

Take a Rohlin partition for $\{\meC,\th\}$ as \cite[Lemma 10]{KST}, 
i.e., a family of orthogonal projections $\{E_i\}_{i=0}^m$ 
in $\meC$ such that 
\begin{enumerate}

\item $\th(E_i)=E_{i+1}$ for $0\leq i\leq m-1$, 

\item $\displaystyle\left|\sum_{i=0}^{m-1}E_i\right|_\vph\geq 1-1/m$, 

\item $|E_0|_\vph<1/m$, $|E_m|_\vph<2/m$. 
\end{enumerate}

Next we average $q_i$ along with the Rohlin partition as follows. 
For each $0\leq i\leq t$, we define $\wdt{q_i}\in N_\om$ by 
\[
\wdt{q_i}=\sum_{j=0}^{m-1} \th_\om^j(q_i)E_j. 
\]
It is clear that $\{\wdt{q_i}\}_{i=0}^t$ are orthogonal projections,
and we have 
\begin{align}
\sum_{i=0}^t|\wdt{q_i}|_{\vph^\om}
=&\,
\sum_{i=0}^t\vph^\om(\wdt{q_i})
=
\sum_{i=0}^t\sum_{j=0}^{m-1} \vph^\om(\th_\om^j(q_i)E_j)
\notag\\
=&\,
\sum_{j=0}^{m-1}
\vph^\om\Big{(}\th_\om^j\Big{(}\sum_{i=0}^t q_i\Big{)}E_j\Big{)}
=
\sum_{j=0}^{m-1} \vph(E_j)\notag\\
>&\,1-1/m. \label{ineq: tilde-q}
\end{align}
For all $0\leq i\leq t$, we have 
\begin{align*}
\wdt{q_i}\rho^\om(\wdt{q_i})
=&\,
\sum_{j,k=0}^{m-1}
\th_\om^j(q_i)E_j \rho^\om(\th_\om^k(q_i)E_k)
\notag\\
=&\,
\sum_{j,k=0}^{m-1}
\th_\om^j(q_i)\rho^\om(\th_\om^k(q_i))E_j \rho^\om(E_k)
\notag\\
=&\,
\sum_{j,k=0}^{m-1}
(\th^\om)^j\big{(}q_i\rho^\om(\th_\om^{k-j}(q_i))\big{)}E_j \rho^\om(E_k). 
\notag\\
\end{align*}
Since $N_\om$ and $\rho^\om(N_\om)$ are contained 
in $(N_\om)_{\vph^\om}$, 
we have for $1\leq i\leq t$, 
\begin{align}
|\wdt{q_i}\rho^\om(\wdt{q_i})|_{\vph^\om}
=&\,
\left|
\sum_{j,k=0}^{m-1}
(\th^\om)^j\big{(}q_i\rho^\om(\th_\om^{k-j}(q_i))\big{)}E_j \rho^\om(E_k)
\right|_{\vph^\om}
\notag\\
\leq&\,
\sum_{j,k=0}^{m-1}
\left|
(\th^\om)^j\big{(}q_i\rho^\om(\th_\om^{k-j}(q_i))\big{)}E_j \rho^\om(E_k)
\right|_{\vph^\om}
\notag\\
\leq&\,
\sum_{j,k=0}^{m-1}
\left|
(\th^\om)^j\big{(}q_i\rho^\om(\th_\om^{k-j}(q_i))\big{)}
\right|_{\vph^\om}
\|E_j \rho^\om(E_k)\|
\notag\\
\leq&\,
\sum_{j=0}^{m-1}
\sum_{|n|\leq m}
\left|
(\th^\om)^j\big{(}q_i\rho^\om(\th_\om^{n}(q_i))\big{)}
\right|_{\vph^\om}
\notag\\
=&\,
\sum_{|n|\leq m}
m\left|
q_i\rho^\om(\th_\om^{n}(q_i))
\right|_{\tvph^\om}
\notag\\
<&\,
m\vep_1|q_i|_{\tvph^\om} 
\qquad(\mbox{by}\ (\ref{ineq: q-rho-th})). 
\label{ineq: q-rho-q}
\end{align}

By definition of $\wdt{q_i}$, 
we obtain $\th_\om(\wdt{q_i})-\wdt{q_i}=-q_i E_0+\th_\om^m(q_i) E_m$. 
Hence we have 
\begin{align*}
\sum_{i=0}^t\left|\th_\om(\wdt{q_i})-\wdt{q_i}\right|_{\vph^\om}
=&\,
\sum_{i=0}^t|-q_i E_0+\th_\om^m(q_i) E_m|_{\vph^\om}
=
\sum_{i=0}^t\left(\vph^\om(q_i E_0)+\vph^\om(\th_\om^m(q_i)E_m)\right)
\\
=&\,
\vph(E_0)+\vph(E_m)
\\
<&\, 1/m+2/m=3/m. 
\end{align*}
By using (\ref{ineq: tilde-q}), we have 
\[
\sum_{i=0}^t\left|\th_\om(\wdt{q_i})-\wdt{q_i}\right|_{\vph^\om}
< \frac{3}{m} 
\left(\frac{1}{m} +\sum_{i=0}^t|\wdt{q_i}|_{\vph^\om}\right). 
\]

\noindent\textbf{Step IV}. 
We show that most of $\{\wdt{q_i}\}_{i=0}^t$ are almost invariant
under $\th_\om$. 

We set an index set $I:=\{i\in\Z\mid -1\leq i\leq t\}$. 
We also put $X_{-1}:=0$, $Y_{-1}=1/m$, 
$X_i:=|\th_\om(\wdt{q_i})-\wdt{q_i}|_{\vph^\om}$ 
and $Y_i:=|\wdt{q_i}|_{\vph^\om}$ for $1\leq i\leq t$. 
Then the above inequality yields 
\[
\sum_{i\in I}X_i<\frac{3}{m} \sum_{i\in I}Y_i. 
\]
We introduce the following index subsets of $I$:
\[
I_0:=\{i\in I\mid X_i<\sqrt{3/m}\,Y_i\},\quad 
I_1:=\{i\in I_0\mid 1\leq i\leq t\}. 
\]
By definition of $I_1$, $\wdt{q_i}\neq0$ if $i\in I_1$ and
we have
\begin{equation}\label{ineq: th-tq}
|\th_\om(\wdt{q_i})-\wdt{q_i}|_{\vph^\om}
<\sqrt{3/m}|\wdt{q_i}|_{\vph^\om}<(\de/2)|\wdt{q_i}|_{\vph^\om}. 
\end{equation}
We estimate of the total size of $Y_i$ for $i\in I\setm I_0$
as follows.
\begin{align*}
\sum_{i\in I\setm I_0}Y_i
\leq&\,
\sum_{i\in I\setm I_0}\sqrt{m/3}X_i
\leq
\sum_{i\in I}\sqrt{m/3}X_i
\\
<&\,
\sqrt{m/3} \,(3/m)\sum_{i\in I}Y_i
=
\sqrt{3/m}\sum_{i\in I}Y_i. 
\end{align*}
This implies 
\[
\sum_{i\in I_0}Y_i
=\sum_{i\in I}Y_i-\sum_{i\in I\setm I_0} Y_i
>
(1-\sqrt{3/m})\sum_{i\in I}Y_i. 
\]
It is trivial that $-1$ is in $I_0$. 
Hence we have 
\begin{align}
\sum_{i\in I_0, 0\leq i\leq t}|\wdt{q_i}|_\vph
=&\, 
\sum_{i\in I_0}Y_i-1/m
\notag\\
>&\,
(1-\sqrt{3/m})\sum_{i\in I}Y_i-1/m
\notag\\
=&\,
(1-\sqrt{3/m})(1/m+\sum_{i=1}^t Y_i)-1/m
\notag\\
=&\,
(1-\sqrt{3/m})\sum_{i=1}^t \vph^\om(\wdt{q_i})-(1/m)\sqrt{3/m}
\notag\\
>&\,
(1-\sqrt{3/m})(1-1/m)-(1/m)\sqrt{3/m} 
\qquad
({\rm by}\ (\ref{ineq: tilde-q}))
\notag\\
=&\,1-1/m-\sqrt{3/m}.
\label{ineq: tilde-q2} 
\end{align}
Now we estimate the size of $\wdt{q_0}$ as follows. 
\begin{align*}
|\wdt{q_0}|_{\vph^\om}
=&\,
\sum_{j=0}^{m-1}\vph^\om(\th^j(q_0)E_j)
\leq
\sum_{j=0}^{m-1}\vph^\om(\th^j(q_0))
\\
<&\,
\sum_{j=0}^{m-1}|q_0|_{(\vph\th^j)^\om}
=m |q_0|_{\tvph^\om}
\\
<&\,
m\vep_1. 
\quad(\mbox{by}\ (\ref{ineq: q0-vep1}))
\end{align*}
Together with this inequality and (\ref{ineq: tilde-q2}), 
we have 
\begin{equation}\label{ineq: tilde-q3}
\sum_{i\in I_1}|\wdt{q_i}|_{\vph^\om}
>1-1/m-\sqrt{3/m}-m\vep_1. 
\end{equation}

Now we set
$
p_0:=1-\sum_{i\in I_1}\wdt{q_i}$. 
Then by (\ref{ineq: tilde-q3}), 
\begin{equation}\label{ineq: p0}
|p_0|_{\vph^\om}<1/m+\sqrt{3/m}+m\vep_1, 
\end{equation}
and 
\begin{align}
|\th_\om(p_0)-p_0|_{\vph^\om}
=&\,
\left|\sum_{i\in I_1}(\th_\om(\wdt{q_i})-\wdt{q_i})\right|_{\vph^\om}
\leq \sum_{i\in I_1}|\th_\om(\wdt{q_i})-\wdt{q_i}|_{\vph^\om}
\notag\\
<&\, \sum_{i\in I_1}(\de/2)|\wdt{q_i}|_{\vph^\om}
\qquad(\mbox{by}\ (\ref{ineq: th-tq}))
\notag\\
\leq&\, \de/2. 
\label{ineq: th-p0}
\end{align}

\noindent\textbf{Step V}.
We sum up $\{p_0\}\cup\{\wdt{q_i}\}_{i\in I_1}$
with phases and obtain unitaries almost invariant under $\th_\om$.
We derive a contradiction by studying how $\rho^\om$ acts on
them.

For $z\in \T=\{z\in\C\mid |z|=1\}$, we define the unitary element 
\[
u(z)=p_0+\sum_{i\in I_1}z^i \wdt{q_i}. 
\]
Then we have for all $z\in \T$, 
\begin{align*}
|\th_\om(u(z))-u(z)|_{\vph^\om}
\leq&\,
|\th_\om(p_0)-p_0|_{\vph^\om}
+
\sum_{i\in I_1} |\th_\om(\wdt{q_i})-\wdt{q_i}|
\\
<&\,
\de/2+(\de/2)\sum_{i\in I_1}|\wdt{q_i}|_{\vph^\om}
\leq \de. 
\qquad(\mbox{by}\ (\ref{ineq: th-tq})\ \mbox{and}\ (\ref{ineq: th-p0})) 
\end{align*}
Since we have taken $\vep$ and $\de$ as in Lemma \ref{lem: cent-theta}, 
we have 
\[
\|\rho^\om(u(z))-u(z)\|_{\vph^\om}<\vep
\quad\mbox{for all}\ z\in\T. 
\]
Making use of $\displaystyle\int_\T z^k\,dz=0$ if $k\neq0$, we have 
\begin{align*}
\vep^2\geq&\,
\int_{\T} \|\rho^\om(u(z))-u(z)\|_{\vph^\om}^2 \,dz
\\
=&\,
\int_{\T}
\left(2-\vph^\om(u(z)^* \rho^\om(u(z)))
-\vph^\om(\rho^\om(u(z)^*)u(z))\right)\, dz
\\
=&\, 
2-2\left(\vph^\om(p_0\rho^\om(p_0))+
\sum_{i\in I_1}\vph^\om(\wdt{q_i}\rho^\om(\wdt{q_i}))\right)
\\
>&\,
2-2\left(\vph^\om(p_0)+\sum_{i\in I_1}m\vep_1|q_i|_{\tvph^\om}\right)
\qquad(\mbox{by}\ (\ref{ineq: q-rho-q}))
\\
>&\,
2-2\left(1/m+\sqrt{3/m}+m\vep_1+m\vep_1\right)
>1>\vep^2 
\qquad(\mbox{by}\ (\ref{ineq: 1/2-m})\ \mbox{and}\ (\ref{ineq: p0})). 
\end{align*}
However this is a contradiction. 
\end{proof}

The following lemma is
similar to \cite[Proposition 3.4 (1)]{Iz1},
which is stated about a canonical extension.

\begin{lem}\label{lem: rho-th-inner}
Let $\rho\in \mE_\th$ such that $(\rho(N)'\cap N)^\th=\C$.
If there exists $n\in \N$ such that $(\rho,\th^n)_N\neq0$, 
then there exists a Hilbert space $\meH$ in $N$ such that 
$\rho=\rho_\meH\circ \th^n$. 
\end{lem}
\begin{proof}
We may and do assume that $(\rho,\id)_N\neq0$ by considering $\rho\th^{-n}$ 
for $\rho$ in case of $(\rho,\th^n)_N\neq0$. 
Note the fact that $(\id,\rho)_N$ is $\rho(N)'\cap N$-$Z(N)$-bimodule,
that is,
if $a\in \rho(N)'\cap N$, $X\in(\id,\rho)_N$ and $b\in Z(N)$,
then $aXb\in(\id,\rho)_N$.
Also note that $(\id,\rho)_N$ is globally invariant under $\th$ 
because $\rho$ and $\th$ commute. 

\noindent\textbf{Step I}. 
We show that for any non-zero projection $p\in \rho(N)'\cap N$, 
there exists a non-zero partial isometry $V\in (\id,\rho)_N$ such that 
$pV\neq 0$. 

Assume that such an element does not exist. 
Then for any $X\in (\id,\rho)_N$, we have $pX=0$, and $\th^n(p)\th^n(X)=0$ 
for all $n\in\Z$. 
Since $\th^n((\id,\rho)_N)=(\id,\rho)_N$, we have 
$\th^n(p)(\id,\rho)_N=0$ 
for all $n\in\Z$. 
This implies that 
$
\bigvee_{n\in\Z}\th^n(p)\,(\id,\rho)_N=0$. 
However the projection
$
\bigvee_{n\in\Z}\th^n(p)$ is contained 
in $(\rho(N)'\cap N)^\th=\C$, and it is equal to $1$. 
This shows $(\id,\rho)_N=0$. This is a contradiction. 
Hence there exists a non-zero $X\in (\id,\rho)_N$ such that $pX\neq0$. 
Let $X=V|X|$ be the polar decomposition of $X$. Then it is easy to see that 
$V\in (\id,\rho)_N$. The partial isometry $V$ satisfies $pV\neq0$. 

\noindent\textbf{Step II}. 
We show the following: 
Let $p\in\rho(N)'\cap N$ be a non-zero projection. 
Then there exists a partial isometry $V\in (\id,\rho)_N$ such that 
$VV^*\leq p$ and $V^* V=z_N(p)$, where $z_N(p)$ is the central support 
projection of $p$ in $N$. 

For partial isometries $V,W \in p(\id,\rho)_N$, we define the relation 
$V\prec W$ by $V=WV^*V$, 
and this gives an inductive order on the set $p(\id,\rho)_N$ 
as in the proof of \cite[Proposition 3.4]{Iz1}. 
Take $W$ a partial isometry from $p(\id,\rho)_N$ 
which is maximal with respect to this order. 
Note that $WW^*\leq p$, and $W^* W\leq z_N(p)$. 
Assume that $W^*W\neq z_N(p)$. 
We set a central projection $z_0:=z_N(p)-W^*W$. 
Using Step I for a non-zero projection $pz_0\in \rho(N)'\cap N$, 
we obtain a non-zero partial isometry $W_0\in (\id,\rho)_N$ 
such that $W_0=pz_0 W_0$. 
Then $W_0\in p(\id,\rho)_N$ and $W_0^*W_0\leq z_0$.
Hence $W_0^* W_0 W^* W=0$. 
Using $W^*W\in Z(N)$, we have 
\[
W^* W_0=(W^*W)W^*W_0=W^*W_0(W^*W)
=W^*W_0 z_0 (W^*W)=0. 
\]
Hence $W+W_0\in p(\id,\rho)_N$ is 
a partial isometry such that $W\prec W+W_0$. 
This is a contradiction. Therefore $W^*W=z_N(p)$. 

\noindent\textbf{Step III}. 
Take any non-zero projection $z$ in $Z(N)$. 
We show that there exists a non-zero projection $z_1\leq z$ in $Z(N)$ 
and a family of partial isometries $(V_i)_{i\in I}$ in $(\id,\rho)_N$ 
such that 
\[
V_i^* V_i=z_1\ \mbox{for all}\ i\in I 
\ \mbox{and}\ \sum_{i\in I}V_i V_i^*=z_1. 
\]

Using Step II for the projection $z$, 
we can take a family of partial isometries $(W_i)_{i\in I}$ 
in $(\id,\rho)_N$ whose range projections are maximally orthogonal 
and $W_i^* W_i=z$ for all $i\in I$. 
We set
$
p:=z-\sum_{i\in I}W_i W_i^*$. 
Again by Step II, 
we can take a partial isometry $W\in (\id,\rho)_N$ 
such that $pW=W$ and $W^*W=z_N(p)$. 
If $z_N(p)=z$, then a extended family $(W_i)_{i\in I}\cup \{W\}$ contradicts 
with the maximality. 
Hence $z_N(p)\neq z$, and we set $z_1:=z-z_N(p)\neq0$ and $V_i:=z_1W_i$. 
It is trivial that $(V_i)_{i\in I}$ are contained in $(\id,\rho)_N$ 
and satisfy $V_i^* V_j=\de_{i,j}V_i^* V_i$ for all $i,j\in I$. 
Since $z_1 p=0$, we have
$
z_1=\sum_{i\in I}V_i V_i^*$. 

\noindent\textbf{Step IV}.
We show that
there exists a partition of unity 
$(z_\la)_{\la\in \La}$ in $Z(N)$
and a family of partial isometries
$(V_i^\la)_{i\in I_\la}$ in $(\id,\rho)_N$ 
for each $\la\in\La$
such that 
\[
(V_i^\la)^* V_i^\la=z_\la\ \mbox{for all}\ i\in I_\la 
\ \mbox{and}\ \sum_{i\in I_\la}V_i^\la(V_i^\la)^*=z_\la. 
\]

Let $(z_\la)_{\la\in \La}\subs Z(N)$ 
be a maximal family of orthogonal projections 
which possesses such families of partial isometries. 
Assume that
$\sum_{\la\in\La}z_\la\neq1$.
By Step III, 
there exist a non-zero projection $z_1\in Z(N)$ and
a family of partial isometries $(V_i)_{i\in I}$ in $(\id,\rho)_N$
such that
$z_1 z_\la=0$ for all $\la\in\La$ and 
\[
V_i^* V_i=z_1\ \mbox{for all}\ i\in I 
\ \mbox{and}\ \sum_{i\in I}V_i V_i^*=z_1. 
\]
The family $(z_\la)_{\la\in\La}\cup\{z_1\}$ contradicts 
with the maximality. 
Hence
$
\sum_{\la\in\La}z_\la=1$. 

\noindent\textbf{Step V}. 
Let $z_\la$ and $V_i^\la$ as in Step IV. 
We show that the cardinality of each $I_\la$ is equal to each other. 

Take $\la_1,\la_2\in\La$. 
By ergodicity of $\{\meC,\th\}$, there exists $n\in \Z$ such that 
$z_{12}:=\th^n(z_{\la_1})z_{\la_2}\neq0$. 
We set $W_{i}^{\la_j}:=(\th^n)^{2-j}(V_i^{\la_j}) z_{12}$ 
for $i\in I_{\la_j}$ and $j=1,2$. 
Trivially $W_i^{\la_j}$ is contained in $(\id,\rho)_N$. 
Then we have, for each $j=1,2$, 
\[
(W_{i}^{\la_j})^*W_{i}^{\la_j}=z_{12}
\ \mbox{for all}\ i\in I_{\la_j}
\ \mbox{and}\ 
\sum_{i\in I_{\la_j}}W_{i}^{\la_j}(W_{i}^{\la_j})^*=z_{12}.
\]
Note that $(W_i^{\la_1})^* W_{j}^{\la_2}\in Z(N)$ 
for $i\in I_1$ and $j\in I_2$. 
Let $|I_\la|\leq\infty$ be the cardinality of $I_\la$ for $\la\in\La$. 
Then 
\begin{align*}
|I_{\la_1}|z_{12}
=&\,
\sum_{i\in I_{\la_1}}(W_i^{\la_1})^* z_{12} W_i^{\la_1}
=
\sum_{i\in I_{\la_1}}
\sum_{j\in I_{\la_2}}
(W_i^{\la_1})^* W_{j}^{\la_2}(W_{j}^{\la_2})^*W_i^{\la_1}
\\
=&\,
\sum_{j\in I_{\la_2}}
\sum_{i\in I_{\la_1}}
(W_{j}^{\la_2})^*W_i^{\la_1}(W_i^{\la_1})^* W_{j}^{\la_2}
=
\sum_{j\in I_{\la_2}}
(W_{j}^{\la_2})^*z_{12} W_{j}^{\la_2}
\\
=&\,
\sum_{j\in I_{\la_2}}z_{12}
=
|I_{\la_2}|z_{12}. 
\end{align*}
Hence $|I_{\la_1}|=|I_{\la_2}|$. 

\noindent\textbf{Step VI}. 
By Step V, 
we may and do assume that each index set $I_\la$ is the same index set
$I$. 
We set
$V_i:=\sum_{\la\in\La}V_i^\la$.
It is trivial that $V_i$ is an isometry 
in $(\id,\rho)_N$ and
$\sum_{i\in I}V_i V_i^*=1$. 
Then the Hilbert space spanned by $(V_i)_{i\in I}$ implements $\rho$.
\end{proof}

\noindent$\bullet$ \textit{Proof of Theorem \ref{thm: mod-cent} 
for hyperfinite type III$_0$ factors.}

It suffices to show that 
$\Cnd(M)_{\rm{irr}}\subs \End(M)_{\rm m}$ as before.
Let $M=N\rti_\th\Z$ be the discrete decomposition
with the implementing unitary $U$.
Take $\si\in \Cnd(M)_{\rm{irr}}$. 
By Lemma \ref{lem: normalize-rho}, 
we may and do assume that $\si$ and $\hth$ commute 
and moreover $\si(U)=U$. 
We set $\rho:=\si|_N$. 
Then the restriction $\ph_\si|_N$ is a left inverse of $\rho$ 
and it commutes with $\th$. 
Hence $\rho\in \mE_\th$. 
Since $M_\om$ is naturally identified with $(N_\om)^{\th_\om}$
\cite[Lemma 7]{KST}, 
$\si^\om=\id$ on $(N_\om)^{\th_\om}$. 
Then by Lemma \ref{lem: rho-th-n}, 
$(\si,\th^n)_N\neq0$ for some $n\in\Z$. 
By Lemma \ref{lem: rho-th-inner}, 
we can find a Hilbert space $\meH$ in $N$ such that
$\si=\rho_{\meH}\th^n$.
Thus $\rho$ is a modular endomorphism
by Lemma \ref{lem: inner-modular}.
\hfill$\Box$

\begin{rem}
In this paper, we have studied centrally trivial endomorphisms with 
finite indices. 
However, as we claimed before, 
central triviality has the meaning even for 
endomorphisms with infinite indices if they have left inverses. 
Hence it is natural to consider the generalization
of Theorem \ref{thm: mod-cent}. 
It seems that Lemma \ref{lem: rho-th-n}, \ref{lem: rho-th-inner}
play important roles because
in them we have not assumed that $\rho$ has finite index, 
but that study is beyond our scopes at the present. 
\end{rem}

\vspace{10pt} 
\noindent \textbf{Acknowledgments.} 
The authors would like to thank Yasuyuki Kawahigashi 
and Masaki Izumi for encouragement on this work. 
They are also grateful to Robert Longo and Sebastiano Carpi 
for various useful comments and discussions. 
A part of this work was done while the first named author 
stayed at Fields Institute 
and the second named author stayed 
at Universit\`{a} di Roma ``Tor Vergata'' , 
and completed 
when the second named author stayed at Katholieke Universiteit Leuven. 
They express gratitude for their warm hospitality.

\end{document}